\documentclass[final]{siamltex}
\usepackage{latexsym,amssymb}
\usepackage{amsmath}
\usepackage{appendix}
  \newtheorem{remark}[theorem]{Remark}
  \newtheorem{example}[theorem]{Example}

\allowdisplaybreaks

\title{Lower dimensional invariant tori for multi-scale Hamiltonian systems}

\author{Weichao Qian\thanks{School of Mathematics, Jilin University, Changchun 130012, P.R. China; Institute of Science and Technology Austria, Am Campus 1, 3400 Klosterneuburg, Austria}
({\tt qian\_wc@163.com}).
\and
Shuguan Ji\thanks{Corresponding authors. School of Mathematics and Statistics, and Center for Mathematics and Interdisciplinary Sciences,
Northeast Normal University, Changchun 130024, P.R. China} ({\tt jisg100@nenu.edu.cn;
       jishuguan@hotmail.com}).
      \and
      Yong Li\thanks{School of Mathematics, Jilin University, Changchun 130012, P.R. China; School of Mathematics and Statistics, and Center for Mathematics and Interdisciplinary Sciences,
Northeast Normal University, Changchun 130024, P.R. China} ({\tt liyongmath@163.com}). }

\begin{document}

\maketitle

\begin{abstract}
The ``Fundamental Theorem"  given by Arnold in \cite{Arnold} asserts the persistence of full dimensional invariant tori for 2-scale Hamiltonian systems. However, persistence in multi-scale systems is much more complicated and difficult. In this paper, we explore the persistence of lower dimensional invariant tori for multi-scale Hamiltonian systems, which play an important role in dynamics of resonant Hamiltonian systems. Moreover, using the corresponding results we give a quasi-periodic Poincar\'{e} Theorem for multi-scale Hamiltonian systems, i.e., at least $2^{m_0}$ families resonant tori survive small perturbations, where the integer $m_0$ is the multiplicity of resonance.
\end{abstract}



\pagestyle{myheadings} \thispagestyle{plain} \markboth{\sc W. Qian, S. Ji and Y. Li}{Lower dimensional invariant tori for multi-scale Hamiltonian systems}

\section{Introduction}\setcounter{equation}{0} \label{introduction}
In this paper, we study the persistence of lower dimensional invariant tori of the following Hamiltonian systems
 \begin{eqnarray}
\label{Eq1main}H(x,y,z, \lambda)&=& e+ \langle \omega(\lambda), y\rangle+ \frac{1}{2}\langle \left(
                                                                              \begin{array}{c}
                                                                                y \\
                                                                                z \\
                                                                              \end{array}
                                                                            \right),
          {M}(\lambda)\left(
                                                                              \begin{array}{c}
                                                                                y \\
                                                                                z \\
                                                                              \end{array}
                                                                            \right)\rangle+h+ \varepsilon^2 P(x,y, z,\lambda),~~~~~~~
\end{eqnarray}
 where $\omega(\lambda)= \sum\limits_{\tilde{\iota}_1 = 0}^m\varepsilon_{\tilde{\iota}_1} \omega_{\tilde{\iota}_1}(\lambda),$ $ M(\lambda) = \sum\limits_{\tilde{\iota}_2 = 1}^l \mu_{\tilde{\iota}_2} M_{\tilde{\iota}_2}(\lambda),$ $h= \sum\limits_{\tilde{\iota}_3 = 1}^l \sum\limits_{|i|+|j| \geq 3} \mu_{\tilde{\iota}_3}h_{ij}^{\tilde{\iota}_3}(\lambda) y^i z^j,$ $(x, y, z) \in T^n \times G \subset T^n \times R^n \times R^{2m}$, $G$ is a bounded closed region (closure of a bounded, nonempty open set), $\lambda \in \mathcal{O}$, $\mathcal{O}$ is a bounded closed region, $0< \varepsilon \leq \varepsilon_i, \mu_j\ll 1$, $0\leq i \leq m$, $0\leq j\leq l$, $P(x, y,z,\lambda)$ is a real analytic function. As a by-product, we obtain the persistence of resonant tori for multiscale Hamiltonian system with the following form:
\begin{eqnarray}\label{Eq2main}
H(x,y)= \sum\limits_{\iota =0}^m \varepsilon_{\iota} H_{\iota}(y) + \varepsilon^2 P(x,y),
\end{eqnarray}
where  $(x, y) \in T^d \times G \subset T^d \times R^d$, $G$ is a bounded closed region, $0<\varepsilon\leq \varepsilon_i\ll1$, $0\leq i\leq m$. Here the so-called resonant tori mean the frequency $\breve{\omega} = \sum\limits_{\iota =0}^m \varepsilon_{\iota} \partial_y H_{\iota}$ is resonant for some $y$, i.e., there is a least one $k\in Z^d \setminus \{0\}$ such that $\langle k, \breve{\omega}\rangle =0$.

\subsection{Melnikov's Persistence}


Let us do a brief recall for the Melnikov's persistence, i.e. the persistence of invariant tori for the following Hamiltonian systems:
\begin{eqnarray*}
H(x,y,z) = e + \langle w, y\rangle + \frac{1}{2} \langle z,\check{M}  z\rangle + \varepsilon P(x,y,z),
\end{eqnarray*}
where  $(x, y, z) \in T^n \times G \subset T^n \times R^n \times R^{2m}$, $G$ is a bounded closed region, which originated from Melnikov (\cite{Melnikov1,Melnikov2}) and was proved by Eliasson (\cite{Eliasson}). The Melnikov persistence problem has been extensively studied for hyperbolic case (\cite{Graff,Li8,Moser1}), for normally nondegenerate cases, i.e., $\check{M}$ is nonsingular (\cite{Bourgain,Broer2,Chierchia,Han1,Jorba,Li0,Li7,Liu,You,Zhao1}), and for infinite-dimensional Hamiltonian systems (\cite{Berti,Geng,Kuksin1,Poschel2,Wayne,Yuan,Zhang}). In the study of resonant tori, a particular lower dimensional tori, ones $\cite{Cong,Li}$ got the following Hamiltonian systems
\begin{eqnarray*}
H(x,y,z) = e + \langle w, y\rangle + \frac{\varepsilon^{a}}{2} \langle z, \check{M} z\rangle + \varepsilon P(x,y,z),
\end{eqnarray*}
where $0<a<1$, and studied the persistence of invariant tori. However, to our knowledge, there is rare study on the following persistence of invariant tori in Hamiltonian systems
\begin{eqnarray}\label{Eq64}
H(x,y,z) = e + \varepsilon^a\langle w, y\rangle + \frac{1}{2} \langle z, \check{ M} z\rangle + \varepsilon P(x,y,z),
\end{eqnarray}
where $0<a<1$, which describes many physical systems, for example, harmonic oscillators with resonance, properly degenerate Hamiltonian systems with resonance (we place some examples on Section \ref{EXAM}). To obtain such Melnikov's persistence, one has to deal with some {\it small divisors} of the form: $\varepsilon^a\langle k, w\rangle +l \neq 0$, $\forall k\in \mathbb{Z}^n\setminus\{0\}$. Such a small divisor problem is very difficult and has never been solved before. In the present paper, we will touch this persistence. Obviously, Hamiltonian system (\ref{Eq64}) is a special case of (\ref{Eq1main}). We will show the persistence of invariant tori of Hamiltonian systems (\ref{Eq1main}), Theorem \ref{TH1} on page 4, and imply the Melnikov's persistence of Hamiltonian system (\ref{Eq64}).

\subsection{Multi-scale Hamiltonian Systems}
As is known to all, the celebrated KAM theory due to Kolmogorov \cite{Kolmogorov}, Arnold \cite{Arnold2} and Moser \cite{Moser} asserts the persistence of full dimensional invariant tori for nearly integrable Hamiltonian systems under Kolmogorov nondegenerate condition, i.e., the Hessian  of the integrable part is nondegenerate, which could be weakened to R\"{u}ssmann nondegenerate condition (\cite{Cheng,Chow,Russmann2,Sevryuk2,XuJ}): the image of the frequency map is not contained in any hyperplane passing through the origin.

Lots of systems in celestial mechanics are degenerate, for which the KAM theory does not work due to certain degeneracy resulting from multiple scales. To overcome the degeneracy, Arnold introduced properly degenerate Hamiltonian systems, i.e. $2$-scale Hamiltonian systems, proved ``Fundamental Theorem": {\it most of non-resonant tori for $2$-scale Hamiltonian systems is persistent}, and applied the corresponding result to planar lunar problem (\cite{Arnold}). But 2 scales are not enough to remove the degeneracy of some systems in celestial mechanics, for example, the comet spatially restricted three-body problem (\cite{Meyer3}) and the spatial lunar problem (\cite{Meyer4}). These all require a general version of ``Fundamental Theorem". The paper \cite{han} introduced multi-scale Hamiltonian systems and showed the persistence of full dimensional invariant tori if there is an order relationship for multiple scales, which has many applications (\cite{Cors,Daniela,Meyer3,Meyer4}). When there is no order relationship, the paper \cite{Qian} studied the persistence of full dimensional invariant tori.

Naturally, one asks the question:
\emph{does the general version of ``Fundamental Theorem" hold for lower dimensional invariant tori}?  In detail, is there a family of invariant tori for the following Hamiltonian systems
\begin{eqnarray}\label{Eq54}
H(x,y,z,\lambda)&=& \langle \omega(\lambda), y\rangle+ \frac{1}{2}\langle z,
          M(\lambda) z\rangle + \varepsilon^2 P(x,y, z),
\end{eqnarray}
where $\omega(\lambda)= \sum\limits_{\tilde{\iota}_1 = 0}^m\varepsilon_{\tilde{\iota}_1} \omega_{\tilde{\iota}_1}(\lambda),$ $ M(\lambda) = \sum\limits_{\tilde{\iota}_2 = 1}^l \mu_{\tilde{\iota}_2} M_{\tilde{\iota}_2}(\lambda),$ $0< \varepsilon \leq \varepsilon_i, \mu_j\ll 1$, $0\leq i \leq m$, $0\leq j\leq l$. Obviously, the system also exhibits multiple scales near the relative equilibrium $z=0$. The work \cite{Xu3} studied the persistence of invariant tori with slightly deformed Diophantine frequencies not involving the small divisors such as Hamiltonian $(\ref{Eq64})$. In the present paper, we will consider not only the persistence with the same frequency, but also the persistence of frequency ratio on a given energy surface.
Comparing the classical KAM theory, due to some elements of $\varepsilon \tilde M^{-1}$ maybe unbounded as $\varepsilon\rightarrow0$, it is much more difficult to estimate the inverse of some multi-scale matrix. The proof is quite complicated and will be placed in Appendix \ref{Inverse}. Obviously, (\ref{Eq54}) is also a special case of (\ref{Eq1main}). Therefore, Theorem \ref{TH1} on page 4 imply the persistence of lower dimensional invariant tori for multi-scale Hamiltonian (\ref{Eq54}). In other words, as a by-product, we obtain ``Fundamental Theorem" for lower dimensional invariant tori.

\subsection{Resonant Invariant Tori}
As we all know, lower dimensional invariant tori come from resonant tori, which originated from Poincar\'{e} in studying the persistence of finite periodic orbits  under small perturbation (Poincar\'{e} Theorem). Periodic orbit is a particular type of resonant tori and there is no ``small divisor".
The general case of resonance is extremely complicated. And it is important to study the mechanisms which lead to the destruction of resonant tori and the number of surviving resonant tori under small perturbations (\cite{Cong,Li,Treshchev}).

Naturally, we have to consider the question: \emph{does the quasi-periodic Poincar\'{e} Theorem hold for multi-scale Hamiltonians}? In this paper, we will also touch this problem. When there is an order relationship among multiple scales, papers \cite{Xu1,Xu2} studied the persistence of resonant tori for multi-scale Hamiltonian systems. The present paper will study cases that there is no order relationship among multiple scales, which result from resonant tori at high degeneracy, and study the persistence of resonant tori and how many family of resonant tori will survive small perturbation.

Specifically, on a lower dimensional manifold, as a resonant manifold, the resonant multi-scale Hamiltonians (\ref{Eq2main}) could be reduced to (\ref{Eq1main}) and using the persistence of invariant tori for Hamiltonian (\ref{Eq1main}) we show the persistence of resonant tori on a lower dimensional manifold (see {Theorem \ref{resonant}}).

\subsection{Main Results}
Consider the following Hamiltonian systems
 \begin{eqnarray}
\label{701} H(x,y,z, \lambda)&=& e+ \langle \omega(\lambda), y\rangle+ \frac{1}{2}\langle \left(
                                                                              \begin{array}{c}
                                                                                y \\
                                                                                z \\
                                                                              \end{array}
                                                                            \right),
          {M}(\lambda)\left(
                                                                              \begin{array}{c}
                                                                                y \\
                                                                                z \\
                                                                              \end{array}
                                                                            \right)\rangle+h+ \varepsilon^2 P(x,y, z,\lambda),~~~~~~~
\end{eqnarray}
 where $\omega(\lambda)= \sum\limits_{\tilde{\iota}_1 = 0}^m\varepsilon_{\tilde{\iota}_1} \omega_{\tilde{\iota}_1}(\lambda),$ $ M(\lambda) = \sum\limits_{\tilde{\iota}_2 = 1}^l \mu_{\tilde{\iota}_2} M_{\tilde{\iota}_2}(\lambda),$ $h= \sum\limits_{\tilde{\iota}_3 = 1}^l \sum\limits_{|i|+|j| \geq 3} \mu_{\tilde{\iota}_3}h_{ij}^{\tilde{\iota}_3}(\lambda) y^i z^j,$ $(x, y, z) \in T^n \times G \subset T^n \times R^n \times R^{2m}$, $G$ is a bounded closed region (closure of a bounded, nonempty open set), $\lambda \in \mathcal{O}$, $\mathcal{O}$ is a bounded closed region, $0< \varepsilon \leq \varepsilon_i, \mu_j\ll 1$, $0\leq i \leq m$, $0\leq j\leq l$, $P(x, y,z,\lambda)$ is a real analytic function.

Denote $M_{i}(\lambda)= \left(
    \begin{array}{cc}
      M_{i,11} & M_{i,12} \\
      M_{i,21} & M_{i,22} \\
    \end{array}
  \right),$ where $M_{i,11} = M_{i,11}^T$, $M_{i,12} = M_{i,21}^T$, $M_{i,21}$, $M_{i,22} = M_{i,22}^T$ are $n\times n$, $n\times 2m$, $2m\times n$, $2m \times 2m$ submatrix of $M_i$, $0\leq i\leq l$, respectively. Define $
M_{ij}(\lambda)= \sum\limits_{\iota = 1}^l \mu_\iota M_{\iota,ij},\  i, j = 1,2 ,$ and $\Delta(y,z) = M_{11}(\lambda)y + M_{12}(\lambda)z + \partial_y h(y,z,\lambda).$
In this paper, $J$ denotes symplectic matrices of certain dimensions, which match the symplectic structure of the Hamiltonian systems. Let
\begin{eqnarray}
\label{Eq317}\mathcal{A}_1 &=& \left(
     \begin{array}{cc}
       L_{k0} I_n & -M_{21}^TJ \\
       0 & L_{k1}  \\
     \end{array}
   \right),\\
\label{Eq318}\mathcal{A}_2 &=&
  \left(
    \begin{array}{ccc}
      I_n\otimes L_{k0}I_n & I_n \otimes (M_{21}^T J) & 0 \\
      0 & I_n \otimes L_{k1} & - (M_{21}^T J)\otimes I_{2m} \\
      0 & 0 & L_{k2} \\
    \end{array}
  \right),
\end{eqnarray}
where
\begin{eqnarray}
\label{Eq319}L_{k0}&=&\sqrt{-1} \langle k, \omega\rangle,\\
\label{Eq320}L_{k1}&=&\sqrt{-1} \langle k, \omega\rangle I_{2m} - M_{22}J,\\
\label{Eq321}L_{k2}&=&\sqrt{-1} \langle k, \omega\rangle I_{4m^2} - (M_{22}J) \otimes I_{2m} - I_{2m}\otimes(M_{22}J).
\end{eqnarray}
For convenience to state our results, we first list the following assumptions:
\begin{itemize}
\item[{(D)}] There is a constant $N$ such that $
{\rm rank} \{\partial_\lambda^\alpha \omega: |\alpha| \leq N\}= n,$
where $\omega= \sum\limits_{\tilde{\iota}_1 = 0}^m\varepsilon_{\tilde{\iota}_1} \omega_{\tilde{\iota}_1}(\lambda).$
\end{itemize}

\begin{itemize}
\item[{(M1)}]There is a finite positive integer $N$ such that for $\lambda\in \mathcal{O}$, $
{\rm rank} \{c_i: 1\leq i\leq n+2m\} = n+2m,$
where $c_i$ is a column of $\{\partial_\lambda^\alpha \mathcal{A}_1^i: 0\leq |\alpha|\leq N \}$, $\mathcal{A}_1^i$ is the $i$-th column of $\mathcal{A}_1$.
\end{itemize}

\begin{itemize}
\item[{(M2)}] There is a finite positive integer $N$ such that for $\lambda\in \mathcal{O}$, $
{\rm rank} \{\bar{c}_i: 1\leq i\leq (n+2m)^2\} = (n+2m)^2,$
where $\bar{c}_i$ is a column of $\{\partial_\lambda^\alpha \mathcal{A}_2^i: 0\leq |\alpha|\leq N \}$, $\mathcal{A}_2^i$ is the $i$-th column of $\mathcal{A}_2$.
\end{itemize}

\begin{itemize}
\item [{(C1)}] $M^T M \geq \big(\min\{\mu_1, \cdots, \mu_l\}\big)^2 I_{n+ 2m}$.
 \end{itemize}
 \begin{itemize}

   \item [{(C2)}] $\left(
            \begin{array}{cc}
              M & \omega \\
              \omega^T & 0 \\
            \end{array}
          \right)^T \left(
            \begin{array}{cc}
              M & \omega \\
              \omega^T & 0 \\
            \end{array}
          \right) \geq \big(\min\{\varepsilon_1, \cdots, \varepsilon_m, \mu_1,\cdots, \mu_l\}\big)^2 I_{n+ 2m+1}.$
\end{itemize}

\begin{theorem}\label{TH1}
Consider Hamiltonian system $(\ref{701})$.
\begin{enumerate}
  \item [{(i)}] If ${(D)}$, ${(M1)}$, ${(M2)}$ and ${(C1)}$ hold, then there exist a $\varepsilon_0>0$ and a family of Cantor set $\mathcal{O}_\varepsilon \subset \mathcal{O},$ $0<\varepsilon<\varepsilon_0$, such that each $n$-torus $T_\lambda$, $\lambda \in \mathcal{O}_\varepsilon$, persists and gives rise to a perturbed $n$-torus $T_{\varepsilon,\lambda}$ with the same frequency. Moreover, the relative Lebesgue measure $|\mathcal{O}\setminus \mathcal{O}_\varepsilon| \rightarrow 0$ as $\varepsilon \rightarrow 0.$
  \item [{(ii)}]If
${(D)}$, ${(M1)}$, ${(M2)}$ and ${(C2)}$ hold on a given energy surface $\Xi = \{ \lambda: N(\lambda) = c\}$, then there exist a $\varepsilon_0>0$ and a family of Cantor set $\Xi_\varepsilon \subset \Xi,$ $0<\varepsilon<\varepsilon_0$, such that each $n$-torus $T_\lambda$, $\lambda \in \Xi_\varepsilon$, persists and gives rise to a perturbed $n$-torus $T_{\varepsilon,\lambda}$, on which the frequency ratio is maintained. Moreover, the relative Lebesgue measure $|\Xi\setminus \Xi_\varepsilon| \rightarrow 0$ as $\varepsilon \rightarrow 0.$
\end{enumerate}
\end{theorem}

\begin{remark}\label{ergodic}
A system is called quasi-ergodic if it has an everywhere dense trajectory. The quasi-ergodic hypothesis states that a generic Hamiltonian system is quasi-ergodic on typical connected components of the energy levels {(\cite{Arnold1})}. However, our result ({{Theorem~\ref{TH1}}}) shows that the quasi-ergodic hypothesis is false for multi-scale Hamiltonian system {(\ref{701})}.
\end{remark}

\begin{remark}
Condition ${(C1)}$ could be weakened to
\begin{itemize}
\item [{$(C1')$}] $M_{22}^T M_{22} \geq \big(\min\{\mu_1, \cdots, \mu_l\}\big)^2 I_{n+ 2m}$,
\end{itemize}
which is enough to remove the resonant terms $\langle p_{001}, z\rangle$ during each KAM step.
But there will exist drift for frequency of the perturbed tori. Moreover, condition ${(D)}$ is not enough, because the new frequency after a KAM step will depend on $\mu_j$, $1\leq j\leq l.$ We will study this case later.
\end{remark}
\begin{remark}
Condition ${(C2)}$ is called the isoenergetically nondegenerate condition to the persistence of lower dimensional invariant tori for multi-scale Hamiltonians. The isoenergetically nondegenerate condition was first considered by Arnold ({\cite{Arnold1}}). For isoenergetically nondegenerate conditions on manifolds, we refer to {\cite{Chow,Qian1,Sevryuk1,Zhao}}. Similarly, condition ${(C2)}$ could also be weakened. And we will also study it later.
\end{remark}
\begin{remark}
When having no multi-scale, condition ${(D)}$ is equivalent to a well known R\"{u}ssmann nondegenerate condition. For nondegenerate conditions in classical Hamiltonian systems, refer to {\cite{Cheng,Russmann2,Sevryuk2,XuJ}}.
\end{remark}

We make the following assumptions:
\begin{itemize}
  \item [(M1$'$)] There is a finite positive integer $N$ such that for $\lambda\in \mathcal{O}$,
$rank \{c_i': 1\leq i\leq 2m\}= 2m,$
where $c_i'$ is a column of $\{\partial_\lambda^\alpha L_{k1}^i: 1\leq |\alpha|\leq N \}$, $L_{k1}^i$ is the $i$-th column of $L_{k1}$;
  \item [(M2$'$)] There is a finite positive integer $N$ such that for $\lambda\in \mathcal{O}$, $rank \{\bar{c}_i': 1\leq i\leq 4m^2\} = 4m^2,$
where $\bar{c}_i'$ is a column of $\{\partial_\lambda^\alpha L_{k2}^i: 1\leq |\alpha|\leq N \}$, $L_{k2}^i$ is the $i$-th column of $L_{k2}$.
\end{itemize}
Then, according to {Theorem \ref{TH1}}, we have the following corollary whose proof is placed on Section \ref{POC}.
\begin{corollary}\label{cor.}
Consider Hamiltonian $(\ref{701})$.
\begin{enumerate}
  \item [{(i)}] Assume ${(D)}$, ${(C1)}$, ${(M1')}$ and ${(M2')}$. Then there are a $\varepsilon_0>0$ and a family of Cantor set $\mathcal{O}_\varepsilon \subset \mathcal{O},$ $0<\varepsilon<\varepsilon_0$, such that each $n$-torus $T_\lambda$, $\lambda \in \mathcal{O}_\varepsilon$, persists and gives rise to a perturbed $n$-torus $T_{\varepsilon,\lambda}$ with the same frequency. Moreover, the relative Lebesgue measure $|\mathcal{O}\setminus \mathcal{O}_\varepsilon| \rightarrow 0$ as $\varepsilon \rightarrow 0.$
  \item [{(ii)}] Assume ${(D)}$, ${(C2)}$, ${(M1')}$ and ${(M2')}$ hold on a given energy surface $\Xi = \{ \lambda: N(\lambda) = c\}$. Then there are a $\varepsilon_0>0$ and a family of Cantor set $\Xi_\varepsilon \subset \Xi,$ $0<\varepsilon<\varepsilon_0$, such that each $n$-torus $T_\lambda$, $\lambda \in \Xi_\varepsilon$, persists and gives rise to a perturbed $n$-torus $T_{\varepsilon,\lambda}$, on which the frequency ratio is maintained. Moreover, the relative Lebesgue measure $|\Xi\setminus \Xi_\varepsilon| \rightarrow 0$ as $\varepsilon \rightarrow 0.$
\end{enumerate}
\end{corollary}

\begin{remark}\label{equi}
When $MJ$ is diagonalizable, which is obvious if the geometric multiplicity of some eigenvalue of $MJ$ is strictly less than its algebraic, $L_{k1}$ and $L_{k2}$ are equivalent to $
L_{k1} = \sqrt{-1} \langle k, \omega\rangle I_{2m} - E,$ and $L_{k2} =  \sqrt{-1} \langle k, \omega\rangle I_{4m^2} - E \otimes I_{2m} - I_{2m} \otimes E,$
where $E = S^{-1} MJ S$ is a diagonal matrix, $S$ is a nonsingular matrix. Let $\hat{\lambda}_i$, $1\leq i \leq 2m$, be the eigenvalue of $MJ$. Therefore, conditions ${(M1')}$ and ${(M2')}$ are equivalent to
\begin{itemize}
 \item [${(M1'')}$] There is a finite positive integer $N$ such that for $\lambda\in \mathcal{O}$,
 $\partial_\lambda^l \tilde{L}_{k1} \neq 0~ for ~some~ l,$
 where $ \tilde{L}_{k1}=\sqrt{-1} \langle k, \omega\rangle - \hat{\lambda}_i$, $1\leq i\leq 2m$, $1\leq |l| \leq N$;
 \item [${(M2'')}$] There is a finite positive integer $N$ such that for $\lambda\in \mathcal{O}$,
 $\partial_\lambda^l \tilde{L}_{k2} \neq 0~ for ~some~ l,$
 where $ \tilde{L}_{k2}=\sqrt{-1} \langle k, \omega\rangle - \hat{\lambda}_i - \hat{\lambda}_j$, $1\leq i,j\leq 2m$, $1\leq |l| \leq N$.
\end{itemize}
\end{remark}

 Next, we will state a quasi-periodic Poincar\'{e} Theorem for multi-scale Hamiltonian systems. Consider a resonant multi-scale Hamiltonian systems
\begin{eqnarray}\label{Eq63}
H(x,y)= \sum\limits_{\iota =0}^m \varepsilon_{\iota} H_{\iota}(y) + \varepsilon^2 P(x,y),
\end{eqnarray}
where  $(x, y) \in T^d \times G \subset T^d \times R^d$, $G$ is a bounded closed region, $0<\varepsilon\leq \varepsilon_i\ll1$, $0\leq i\leq m$.

We classify the resonant frequency first. The resonant frequency $\omega$ is $m_0$-resonant if there is  rank $m_0$ subgroup $g$ of $Z^d$ generated by $\tau_1,$ $\cdots,$ $\tau_{m_0}$ such that $\langle k, \omega\rangle = 0$ for all $k \in g$ and $\langle k, \omega\rangle \neq 0$ for all $k \in Z^d/ g$. Introduce the $g$-resonant manifold $O(g, G) = \{y\in G: \langle k, \omega(y)\rangle= 0, k\in g\},$
which is an $n = d - m_0$ dimensional surface. According to group theory, there are integer vectors $\tau_1'$, $\cdots,$ $\tau_n'$ $\in$ $Z^d$ such that $Z^d$ is generated by $\tau_1$, $\cdots,$ $\tau_{m_0}$, $\tau_1'$, $\cdots,$ $\tau_n'$ and $\det K_0 = 1$, where $K_0 = (K_*, K')$, $K_* = (\tau_1', \cdots, \tau_n'),$ $K' = (\tau_1, \cdots, \tau_{m_0})$ are $d\times d$, $d\times n$, $d \times m_0$, respectively (\cite{Cong,Treshchev}).

Denote $\mathcal{H}_1= \sum\limits_{\iota =0}^m \varepsilon_{\iota} H_{\iota}(y)$ and
\begin{eqnarray*}
\breve{M} = \left(
        \begin{array}{cc}
          K_0^T \partial_y^2 \mathcal{H}_1 K_0 & 0 \\
          0 & \varepsilon \partial_{x_2}^2 \int_{T^n}P (x,0) d x_1  \\
        \end{array}
      \right)=\left(
      \begin{array}{cc}
        \breve{M}_{11} & \breve{M}_{12} \\
        \breve{M}_{21} & \breve{M}_{22} \\
      \end{array}
    \right),
\end{eqnarray*}
where $x_1 = K_*^T x,$ $x_2 = (K')^T x,$ the dimensions of $\breve{M}$, $\breve{M}_{11}$, $\breve{M}_{12}$, $\breve{M}_{21}$, $\breve{M}_{22}$ are $(n + 2m_0)\times (n + 2m_0)$, $n \times n$, $n\times 2m_0$, $2m_0 \times n$, $2m_0 \times 2m_0$, respectively. Let
\begin{eqnarray*}
\breve{\mathcal{A}}_1 &=& \left(
     \begin{array}{cc}
       \breve{L}_{k0} I_n & -\breve{M}_{21}^TJ \\
       0 & \breve{L}_{k1}  \\
     \end{array}
   \right),\\
\breve{\mathcal{A}}_2 &=&
  \left(
    \begin{array}{ccc}
      I_n\otimes \breve{L}_{k0}I_n & I_n \otimes (\breve{M}_{21}^T J) & 0 \\
      0 & I_n \otimes \breve{L}_{k1} & - (\breve{M}_{21}^T J)\otimes I_{2m_0} \\
      0 & 0 & \breve{L}_{k2} \\
    \end{array}
  \right),
\end{eqnarray*}
where $\breve{L}_{k0}=\sqrt{-1} \langle k, \omega_*\rangle,$ $\breve{L}_{k1}=\sqrt{-1} \langle k, \omega_*\rangle I_{2m_0} - \breve{M}_{22}J,$
 $\breve{L}_{k2}=\sqrt{-1} \langle k, \omega_*\rangle I_{4m_0^2} - (\breve{M}_{22}J) \otimes I_{2m_0} - I_{2m_0}\otimes(\breve{M}_{22}J),$ $\omega_* = K_*^T \partial_y \mathcal{H}_1 = \varepsilon_0 \omega_0^*+ \cdots+\varepsilon_m \omega_m^*.$

For convenience to state our results on quasi-periodic Poincar\'{e} Theorem, we make the following assumptions:

\begin{enumerate}
   \item [{(S1)}] There is a finite positive integer $N$ such that $
rank \{\partial_\lambda^\alpha \omega_*: 1\leq|\alpha| \leq N\}= n,$
where $\omega_* = K_*^T \partial_y \mathcal{H}_1 = \varepsilon_0 \omega_0^*+ \cdots+\varepsilon_m \omega_m^*$;
   \item [{(S2)}] For some positive constant $\tilde \sigma$ (independent of $\varepsilon$), $|\det \partial_{x_2}^2 \int_{T^n}P (x,0) d x_1| > \tilde \sigma,$ where $x_1 = K_*^T x$, $x_2 = (K')^T x$;
   \item [{(S3)}] $\breve{M}^T \breve{M} \geq \varepsilon^2 I_{n+ 2m_0}$;
\item [{(S4)}] $\left(
            \begin{array}{cc}
              \check{M} & \breve{\omega} \\
              \breve{\omega}^T & 0 \\
            \end{array}
          \right)^T \left(
            \begin{array}{cc}
              \breve{M} & \breve{\omega} \\
              \breve{\omega}^T & 0 \\
            \end{array}
          \right) \geq \varepsilon^2 I_{n+ 2m_0+1},$ where $\breve{\omega} = (\omega_*, 0)\in R^{n + 2m_0}$;
\item[{(S5)}]There is a finite positive integer $N$ such that for $\lambda\in \mathcal{O}$,$rank \{c_i: 1\leq i\leq n+2m_0\} = n+2m_0,$
where $c_i$ is a column of $\{\partial_\lambda^\alpha \breve{\mathcal{A}}_1^i: 1\leq |\alpha|\leq N \}$, $\breve{\mathcal{A}}_1^i$ is the $i$-th column of $\breve{\mathcal{A}}_1$;
\item[{(S6)}]There is a finite positive integer $N$ such that for $\lambda\in \mathcal{O}$, $rank \{\bar{c}_i: 1\leq i\leq (n+2m_0)^2\} = (n+2m_0)^2,$
where  $\bar{c}_i$ is a column of $\{\partial_\lambda^\alpha \breve{\mathcal{A}}_2^i: 1\leq |\alpha|\leq N \}$, $\breve{\mathcal{A}}_2^i$ is the $i$-th column of $\breve{\mathcal{A}}_2$.
 \end{enumerate}

\begin{theorem}\label{resonant}
Let resonant multi-scale Hamiltonian systems $(\ref{Eq63})$ be real analytic on the complex neighborhood of $T^d \times G$.
\begin{itemize}
  \item [(i)]If ${(S1)}$, ${(S2)}$, ${(S3)}$, ${(S5)}$ and ${(S6)}$ hold, then there exist a $\varepsilon_0 >0 $ and a family of Cantor sets $O_\varepsilon(g, G) \subset O (g, G)$, $0<\varepsilon < \varepsilon_0$, such that for each $y \in O_\varepsilon(g, G)$, system $(\ref{Eq63})$ admits $2^{m_0}$ families of invariant tori, possessing  hyperbolic, elliptic or mixed types, associated to nondegenerate relative equilibria. All such perturbed tori corresponding to a same $y \in O_\varepsilon(g,G) $ are symplectically conjugated to the standard quasi-periodic $n$-tori $T^n$ with the Diophantine frequency vector $ {\omega}_*$. Moreover, the relative Lebesgue measure $|O(g,G) \setminus O_{\varepsilon}(g,G)|$ tends to 0 as $\varepsilon \rightarrow 0$.
  \item [(ii)]If ${(S1)}$, ${(S2)}$, ${(S4)}$, ${(S5)}$ and ${(S6)}$ hold on a given energy surface $\Xi = \{ y: \mathcal{H}_1(y) = c\}$, then there exist a $\varepsilon_0 >0 $ and a family of Cantor sets $\Xi_\varepsilon \subset \Xi$, $0<\varepsilon < \varepsilon_0$, for each $y \in \Xi_\varepsilon$,  system $(\ref{Eq63})$ admits $2^{m_0}$ families of invariant tori, possessing  hyperbolic, elliptic or mixed types, associated to nondegenerate relative equilibria. Moreover, the components of the frequency on the unperturbed tori $\omega_*$ and perturbed tori $\bar{\omega}$ satisfy $\bar{\omega} =t \omega_*$, where $t$ is a constant. Moreover, the relative Lebesgue measure $|\Xi \setminus \Xi_{\varepsilon}|$ tends to 0 as $\varepsilon \rightarrow 0$.
\end{itemize}

\end{theorem}

\begin{remark}
{{Theorem~\ref{resonant}}} shows that the quasi-ergodic hypothesis is false for $g$-resonant multi-scale Hamiltonian systems.
\end{remark}
\begin{remark}
According to condition ${(S2)}$ and Morse Theory \emph{(\cite{Milnor})}, $\int_{T^n}P (x,0) d x_1$ has at least $2^{m_0}$ critical points, of which the number is equal to the number of surviving resonant torus.
\end{remark}

We make the following assumptions:
\begin{itemize}
  \item [$(S5')$] There is a finite positive integer $N$ such that for $\lambda\in \mathcal{O}$,
$rank \{c_i: 1\leq i\leq 2m_0\}= 2m_0,$
where $c_i$ is a column of $\{\partial_\lambda^\alpha \breve{L}_{k1}^i: 1\leq |\alpha|\leq N \}$, $\breve{L}_{k1}^i$ is the $i$-th column of $\breve{L}_{k1}$, ;
  \item [$(S6')$] There is a finite positive integer $N$ such that for $\lambda\in \mathcal{O}$, $rank \{\bar{c}_i: 1\leq i\leq 4m_0^2 \} = 4m_0^2,$
where  $\bar{c}_i$ is a column of $\{\partial_\lambda^\alpha \breve{L}_{k2}^i: 1\leq |\alpha|\leq N \}$, $\breve{L}_{k2}^i$ is the $i$-th column of $\breve{L}_{k2}$.
\end{itemize}

Then, according to {Theorem \ref{resonant}}, we have the following corollary.
\begin{corollary}\label{resonant1}
Let Hamiltonian systems $(\ref{Eq63})$ be real analytic on the complex neighborhood of $T^d \times G$.
\begin{itemize}
  \item [(i)]If ${(S1)}$, ${(S2)}$, ${(S3)}$, ${(S5')}$ and ${(S6')}$ hold, then there exist a $\varepsilon_0 >0 $ and a family of Cantor sets $O_\varepsilon(g, G) \subset O (g, G)$, $0<\varepsilon < \varepsilon_0$, such that for each $y \in O_\varepsilon(g, G)$, system $(\ref{Eq63})$ admits $2^{m_0}$ families of invariant tori, possessing  hyperbolic, elliptic or mixed types, associated to nondegenerate relative equilibria. All such perturbed tori corresponding to a same $y \in O_\varepsilon(g,G) $ are symplectically conjugated to the standard quasi-periodic $n$-tori $T^n$ with the Diophantine frequency vector $ {\omega}_*$. Moreover, the relative Lebesgue measure $|O(g,G) \setminus O_{\varepsilon}(g,G)|$ tends to 0 as $\varepsilon \rightarrow 0$.
  \item [{(ii)}]If ${(S1)}$, ${(S2)}$, ${(S4)}$, ${(S5')}$ and ${(S6')}$ hold on a given energy surface $\Xi = \{ y: \mathcal{H}_1(y) = c\}$, then there exist a $\varepsilon_0 >0 $ and a family of Cantor sets $\Xi_\varepsilon \subset \Xi$, $0<\varepsilon < \varepsilon_0$, for each $y \in \Xi_\varepsilon$,  system $(\ref{Eq63})$ admits $2^{m_0}$ families of invariant tori, possessing  hyperbolic, elliptic or mixed types, associated to nondegenerate relative equilibria. Moreover, the components of the frequency on the unperturbed tori $\omega_*$ and perturbed tori $\bar{\omega}$ satisfy $\bar{\omega} =t \omega_*$, where $t$ is a constant. Moreover, the relative Lebesgue measure $|\Xi \setminus \Xi_{\varepsilon}|$ tends to 0 as $\varepsilon \rightarrow 0$.
\end{itemize}

\end{corollary}

\begin{remark}
Let $\check{\lambda}_i$, $1\leq i\leq 2m$, be eigenvalue of $\breve{M}_{22}J$. Conditions ${(S5')}$ and ${(S6')}$ are equivalent to
\begin{itemize}
 \item [(S5$''$)] There is a finite positive integer $N$ such that for $\lambda\in \mathcal{O}$,
 $\partial_\lambda^l \check{{L}}_{k1} \neq 0~ for ~some~ l,$
 where $ \check{{L}}_{k1}=\sqrt{-1} \langle k, \omega_*\rangle - \check{{\lambda}}_i$, $1\leq i\leq 2m$, $1\leq |l| \leq N$;
 \item [(S6$''$)] There is a finite positive integer $N$ such that for $\lambda\in \mathcal{O}$,
 $\partial_\lambda^l \check{{L}}_{k2} \neq 0~ for ~some~ l,$
 where $ \check{{L}}_{k2}=\sqrt{-1} \langle k, \omega_*\rangle - \hat{\lambda}_i - \check{{\lambda}}_j$, $1\leq i,j\leq 2m$, $1\leq |l| \leq N$.
\end{itemize}
\end{remark}

\subsection{Organization of The Article}
The paper is organized as follows.  In Section \ref{038}, we give KAM steps for the multi-scale Hamiltonian systems with high order perturbation. The solvability of homological equations is extremely complicated, due to the general normal form and multiple scales. The most direct difficulty is about the estimate on the inverse of multi-scale matrix. Combining the properties of Hermitian matrices and Weyl Theorem, we introduce nonresonant set $\mathcal{O}_+$ and prove that homological equations are solvable on $\mathcal{O}_+$. Next, we show the nonresonant set after infinite KAM steps is non-null under nondegenerate conditions ${(D)}$, ${(M1)}$ and ${(M2)}$. In Section \ref{041}, we show a procedure to get the Hamiltonian system (\ref{Abs1}) in Section $\ref{038}$, as a Hamiltonian system with high order perturbation, form the initial Hamiltonian system (\ref{701}). In Section \ref{Resonant}, we study the resonant multi-scale Hamiltonian on the resonant manifold and using Theorem \ref{TH1} prove the persistence of resonant tori. In Section \ref{POC}, we give the proof of {Corollary \ref{cor.}} using Theorem \ref{TH1}. Finally, in Section \ref{EXAM}, we give some interesting examples, including the persistence of resonant tori for properly degenerate multi-scale Hamiltonians, the persistence of resonant tori for weakly anharmonic systems without Melnikov's condition and an artificial example.

\section{KAM Steps}\setcounter{equation}{0}\label{038}

Throughout the paper, unless specified explanation, we shall use the same symbol $|\cdot|$ to denote an equivalent (finite dimensional) vector norm and its induced matrix norm, absolute value of functions, and measure of sets, etc., and denote by $|\cdot|_D$ the supremum norm of functions on a domain $D$. Also, for any two complex column vectors $\xi, \zeta$ of the same dimension, $\langle \xi ,\zeta \rangle$ always means $\xi^T \zeta$, i.e. the transpose of $\xi$ times $\zeta$. For the sake of brevity, we shall not specify smoothness orders for functions having obvious orders of smoothness indicated by their derivatives taking. Moreover, all Hamiltonian functions in the sequel will be associated to the standard symplectic structure.

Consider an abstract parameterized Hamiltonian system:
\begin{eqnarray}\label{Abs1}
 H(x,y,z,\lambda)  = e+ \langle \omega , y\rangle + \frac{1}{2}\langle \left(
                                                                              \begin{array}{c}
                                                                                y \\
                                                                                z \\
                                                                              \end{array}
                                                                            \right),
          M  \left(
                                                                              \begin{array}{c}
                                                                                y \\
                                                                                z \\
                                                                              \end{array}
                                                                            \right)\rangle~  + h + \varepsilon P(x,y,z,\lambda),~~
\end{eqnarray}
where  $e$, $\omega$, $M$ and $h$ are defined as those in (\ref{701}), $(x,y,z) \in D(r,s) = \{(x,y,z): |Im x|< r, |y|<s, |z|<s\}$, $\lambda\in\mathcal{O}= \{\lambda: |\lambda|\leq \delta_1, \delta_1~is~a~given~constant\} \subset R^n$, $0<\varepsilon \leq \varepsilon_i, \mu_j \ll 1$, $0\leq i\leq m$, $0\leq j\leq l$. Moreover, $
|P(x,y,z,\lambda)|_{D(r,s) \times \mathcal{O}} \leq \gamma^{3b} s^2 \mu.$
Denote $\bar{\mathcal{O}}= \{\lambda: |\lambda|\leq \delta_1 - \eta\}.$ Then, by Cauchy inequality, $|\partial_\lambda^l P(x,y,z,\lambda)|_{D(r,s) \times \bar{\mathcal{O}}} \leq \frac{\gamma^{3b} s^2 \mu}{\eta^{l_0}},$ where $l_0$ is a given constant satisfying $|l|\leq l_0$.

For the sake of induction, let $r_0 = r,$ $s_0 = s,$ $\gamma_0 = \gamma,$ $\eta_0= \eta,$ $\mu_0 = \mu,$ $\mathcal{O}_0 = \mathcal{O},$ $P_0 = P,$
 where $0 <r, s, \gamma_0, \eta_0\leq 1$. Obviously, for all $l \in Z_+^n $, $|l| \leq l_0,$ $|\partial_\lambda^l P_0|_{{D(r_0, s_0) \times \bar{\mathcal{O}}_0}} \leq \frac{\gamma_0^{3b} s_0^{2} \mu_0}{\eta^{l_0}}.$

As an induction hypothesis, assume that, after $\nu$ steps, we have arrived at
\begin{eqnarray}\label{705}
H_\nu &=& N_\nu(x,y,z,\lambda) + \varepsilon P_\nu(x,y,z,\lambda),\\
\nonumber N_\nu &=& e_\nu(\lambda)+ \langle \omega_\nu(\lambda), y\rangle + \frac{1}{2} \langle \left(
                                                                                              \begin{array}{c}
                                                                                                y \\
                                                                                                z \\
                                                                                              \end{array}
                                                                                            \right), M_\nu(\lambda)\left(
                                                                                                        \begin{array}{c}
                                                                                                          y \\
                                                                                                          z \\
                                                                                                        \end{array}
                                                                                                      \right)
\rangle+ h_\nu(y,z,\lambda).~~~~
\end{eqnarray}
Moreover, for all $l \in Z_+^n $, $|l| \leq l_0$, $|\partial_\lambda^l  P_\nu|_{{D(r_{\nu}, s_{\nu}) \times \bar{\mathcal{O}}_{\nu}}} \leq  \frac{\gamma_\nu^{3b} s_\nu^{2} \mu_\nu}{\eta_\nu^{l_0}}.$

Next, we need to find a symplectic transformation, which will transform $(\ref{705})$, on a smaller domain $D(r_{\nu+1}, s_{\nu+1}) \times {\mathcal{O}}_{\nu+1}$, to a new Hamiltonian
\begin{eqnarray*}
H_{\nu+1} &=& N_{\nu+1} +\varepsilon P_{\nu+1},\\
N_{\nu+1} &=&  e_{\nu+1} + \langle\omega_{\nu+1}, y\rangle + \frac{1}{2} \langle \left(
                                                                       \begin{array}{c}
                                                                         y \\
                                                                         z \\
                                                                       \end{array}
                                                                     \right),
M_{\nu+1} \left(
            \begin{array}{c}
              y \\
              z \\
            \end{array}
          \right)
 \rangle + h_{\nu+1}(y,z),
\end{eqnarray*}
and moreover, for all $l \in Z_+^n $,  $|l| \leq l_0$,$
|\partial_\lambda^l  P_{\nu+1}|_{{D(r_{\nu+1}, s_{\nu+1}) \times \bar{\mathcal{O}}_{\nu+1}}} \leq \frac{\gamma_{\nu+1}^{3b} s_{\nu+1}^{2} \mu_{\nu+1}}{\eta_{\nu+1}^{l_0}}$.

For simplicity, in the rest of this section, we omit the index for all quantities at the $\nu$-th step and use `+' to index all quantities at the $(\nu+1)$-th step. Bellow, all constants are positive and independent of the iteration process. We shall also use the same symbol $c$ to denote any intermediate positive constant which is independent of the iteration process.

To process this KAM step, from $\nu$ to $\nu+1$, we need the following iteration relations: $
s_+ = \frac{1}{8} \alpha s,$ $\mu_+ = (64 c_0)^{\frac{1}{1 - \lambda_0}} \mu^{1 + \sigma},$ $r_+ = r - \frac{r_0}{2^{\nu+1}},$ $\gamma_+ = \gamma - \frac{\gamma_0}{2^{\nu+1}},$ $\eta_+ = \eta - \frac{\eta_0}{2^{\nu+1}},$ $K_+ = ([\log\frac{1}{\mu}]+ 1)^{3{\eta}},$ $\Gamma(r- r_+) = \sum\limits_{0<|k|\leq K_+} |k|^\chi e^{-|k|\frac{r -r_+}{8}},$ $\hat{D}(\lambda) = D(r_++ \frac{7}{8}(r - r_+), \lambda),$ $D(\lambda) = \{(y,z)\in C^n \times C^{2m}:|y|< \lambda,|z|< \lambda\},$ $D_{\frac{i}{8} \alpha} =  D(r_+ + \frac{i -1 }{8}(r - r_+), \frac{i}{8}\alpha s), ~~i = 1,2,\cdots, 8,$ $D_\alpha = D(r_++ \frac{7}{8}(r - r_+), \alpha s),$
where $\alpha = \mu^{\frac{1}{3}},$ $\lambda > 0$, $\chi = (b + 2)\tau +5 l_0 + 10,$ $\sigma = \frac{1}{12}$, $c_0$ only depends on $r_0$, $\beta_0$, and $c_0$ is the maximal among all $c's$ mentioned in this paper.

\subsection{Truncation}\label{Trunc.}
Consider the Taylor-Fourier series \\
$P  = \sum\limits_{\imath\in Z_+^n,\jmath\in Z_+^{2m}, k \in Z^n} p_{k\imath\jmath} y^{\imath} z^{\jmath} e^{\sqrt{-1}\langle k, x\rangle},$
and let $R = \sum\limits_{\substack{\imath+\jmath\leq 2, |k| \leq K_+,\\ \imath\in Z_+^n,\jmath\in Z_+^{2m},  k \in Z^n}} p_{k\imath\jmath} y^{\imath} z^{\jmath} e^{\sqrt{-1}\langle k, x\rangle}$ be the truncation of $P$,
where $K_+$ is defined as above. Standardly, there is  a constant $c$ such that for all $|l| \leq l_0$, $
|\partial_\lambda^l  R|_{{D_\alpha \times \bar{\mathcal{O}}}} \leq c \frac{\gamma^{3b} s^{2} \mu}{\eta^{l_0}}$ and $|\partial_\lambda^l  (P- R)|_{{D_\alpha \times \bar{\mathcal{O}}}} \leq c \frac{\gamma^{3b} s^{2} \mu^{2}}{\eta^{l_0}},$
under the following assumptions
\begin{itemize}
\item[{(H1)}] $K_+ \geq \frac{8(n+ l_0)}{r - r_+}$,
\item[{(H2)}] $\int_{K_+}^\infty \lambda^{n + l_0} e^{-\lambda \frac{r - r_+}{8}} d \lambda \leq \mu.$
\end{itemize}

\subsection{Transformations}\label{Trans.}

Consider the time $1$-map $\phi_F^1$ of the flow generated by a Hamiltonian $F$:
\begin{eqnarray}
\label{709} F &=& \sum\limits_{\imath+\jmath\leq 2, 0< |k| \leq K_+, \imath\in Z_+^n,\jmath\in Z_+^{2m},  k \in Z^n} f_{k\imath\jmath} y^{\imath} z^{\jmath} e^{\sqrt{-1}\langle k, x\rangle}.
\end{eqnarray}
Then Hamiltonian (\ref{705}) arrives at
\begin{eqnarray}
\label{H11}\bar{H}_+=H\circ \phi_F^1 &=&N + \varepsilon R+ \{N, F\} + \bar{P}_+,
\end{eqnarray}
where
\begin{eqnarray}
\label{717} \bar{P}_+&=&  \int_0 ^1 \{R_t,F\}\circ \phi_F^t dt + \varepsilon(P - R )\circ \phi_F^1,\\
\label{Eq425}\{N, F\} &=& \partial_x N \partial_y F - \partial_y N \partial_x F + \partial_z N J \partial_z F,\\
\label{Eq326}R_t &=& (1-t) \{N, F\} +\varepsilon R, \ \  \ J = \left(
        \begin{array}{cc}
          0 & I_{m} \\
          -I_{m} & 0 \\
        \end{array}
      \right).
\end{eqnarray}
Let
\begin{eqnarray}\label{706}
\{N, F\} + \varepsilon (R - [R])= 0,
\end{eqnarray}
where $[R] = \int_{T^n} R(x, \cdot) dx.$
Therefore, (\ref{H11}) is changed to
\begin{eqnarray}
\label{Eq31}\bar{H}_+ =N_+ + \bar{P}_+,
\end{eqnarray}
where $N_+ =\tilde{e} + \langle\tilde{\omega}, y\rangle + \frac{1}{2} \langle \left(
                                                                       \begin{array}{c}
                                                                         y \\
                                                                         z \\
                                                                       \end{array}
                                                                     \right),
\tilde{M} \left(
            \begin{array}{c}
              y \\
              z \\
            \end{array}
          \right)
 \rangle+ h(y,z)+\langle p_{001}, z\rangle,$ $\tilde{e}= e + p_{000},$ $\tilde{\omega}= \omega+ p_{010},$ $\tilde{M} = M +  2\varepsilon\left(
                                                                         \begin{array}{cc}
                                                                           p_{020} & \frac{1}{2} p_{011} \\
                                                                           \frac{1}{2} p_{011}^T & p_{002}\\
                                                                         \end{array}
                                                                       \right).$

Next, find a transformation to remove the first order resonant term $\langle p_{001}, z\rangle$ and the drift of frequency, i.e. $p_{010}$. Consider the following symplectic transformation $
\phi : x \rightarrow x, \left(
                          \begin{array}{c}
                            y \\
                            z \\
                          \end{array}
                        \right) \rightarrow \left(
                                              \begin{array}{c}
                                                y \\
                                                z \\
                                              \end{array}
                                            \right)+ \left(
                                              \begin{array}{c}
                                                y_0 \\
                                                z_0 \\
                                              \end{array}
                                            \right),$
where
\begin{eqnarray}\label{EQ16i}
M \left(
                                              \begin{array}{c}
                                                y_0 \\
                                                z_0 \\
                                              \end{array}
                                            \right) + \partial_{(y,z)} h (y_0, z_0)=
- \left(
                                              \begin{array}{c}
                                                \varepsilon p_{010} \\
                                                \varepsilon p_{001} \\
                                              \end{array}
                                            \right).
\end{eqnarray}
Then Hamiltonian $(\ref{Eq31})$ reaches
\begin{eqnarray*}
H_+ &=& \bar{H}_+ \circ \phi = e_+ + \langle \omega_+, y\rangle + \frac{1}{2} \langle \left(
                                                           \begin{array}{c}
                                                             y \\
                                                             z \\
                                                           \end{array}
                                                         \right), M_+ \left(
                                                                      \begin{array}{c}
                                                                        y \\
                                                                        z \\
                                                                      \end{array}
                                                                    \right)
\rangle+ h_+(y,z) + P_+,
\end{eqnarray*}
where
\begin{eqnarray*}
e_+ &=& e + \langle \omega, y_0\rangle + \frac{1}{2} \langle \left(
                                                               \begin{array}{c}
                                                                 y_0 \\
                                                                 z_0 \\
                                                               \end{array}
                                                             \right), M \left(
                                                               \begin{array}{c}
                                                                 y_0 \\
                                                                 z_0 \\
                                                               \end{array}
                                                             \right)
\rangle +\varepsilon p_{000}+  \langle\varepsilon p_{010}, y_0\rangle  +h(y_0, z_0) \\
 &~&+ \langle \varepsilon p_{001}, z_0\rangle + \frac{1}{2} \langle \left(
                                                               \begin{array}{c}
                                                                 y_0 \\
                                                                 z_0 \\
                                                               \end{array}
                                                             \right), 2\varepsilon\left(
                                                                         \begin{array}{cc}
                                                                           p_{020} & \frac{1}{2} p_{011} \\
                                                                           \frac{1}{2} p_{011}^T & p_{002}\\
                                                                         \end{array}
                                                                       \right)
 \left(
                                                               \begin{array}{c}
                                                                 y_0 \\
                                                                 z_0 \\
                                                               \end{array}
                                                             \right)
\rangle,\\
\omega_+ &=& \omega + \langle \left(
                              \begin{array}{c}
                                y \\
                                z \\
                              \end{array}
                            \right), M \left(
                                         \begin{array}{c}
                                           y_0 \\
                                           z_0 \\
                                         \end{array}
                                       \right)
\rangle+ \langle \left(
                              \begin{array}{c}
                                y \\
                                z \\
                              \end{array}
                            \right),  \left(
                                         \begin{array}{c}
                                         \partial_y h   \\
                                           \partial_z h \\
                                         \end{array}
                                       \right)
\rangle + \langle \left(
                              \begin{array}{c}
                                y \\
                                z \\
                              \end{array}
                            \right),  \left(
                                         \begin{array}{c}
                                         p_{010} \\
                                         p_{001} \\
                                         \end{array}
                                       \right)
\rangle,\\
M_+ &=& M +  2\varepsilon\left(
                                                                         \begin{array}{cc}
                                                                           p_{020} & \frac{1}{2} p_{011} \\
                                                                           \frac{1}{2} p_{011}^T & p_{002}\\
                                                                         \end{array}
                                                                       \right)+ \partial_{(y,z)}^2 h,\\
h_+ &=& h(y+ y_0, z+ z_0) - h(y_0, z_0)- \langle \left(
                                                   \begin{array}{c}
                                                     y \\
                                                     z \\
                                                   \end{array}
                                                 \right),  \left(
                                                   \begin{array}{c}
                                                     \partial_y h \\
                                                     \partial_z h \\
                                                   \end{array}
                                                 \right)
\rangle\\
&~& - \frac{1}{2} \langle \left(
                                \begin{array}{c}
                                  y \\
                                  z \\
                                \end{array}
                              \right),  \partial_{(y, z)}^2 h\left(
                                \begin{array}{c}
                                  y \\
                                  z \\
                                \end{array}
                              \right)
\rangle,\\
P_+ &=&  \bar{P}_+ \circ\phi + \varepsilon \langle \left(
                                                     \begin{array}{c}
                                                       y \\
                                                       z \\
                                                     \end{array}
                                                   \right), \left(
                                                                \begin{array}{cc}
                                                                  p_{020} & \frac{1}{2} p_{011} \\
                                                                  \frac{1}{2} p_{011}^T & p_{020} \\
                                                                \end{array}
                                                              \right)
\left(
                                                     \begin{array}{c}
                                                       y_0\\
                                                       z_0 \\
                                                     \end{array}
                                                   \right) \rangle.
\end{eqnarray*}

\begin{lemma}\label{Lm1}
Assume that
\begin{itemize}
\item[${(H3)}$] $M^T M \geq  \big(\min\{\mu_1, \cdots, \mu_l\} \big)^2 I_{n+ 2m}.$
\end{itemize}
 Then there is a constant $c$ such that, for all $|l| \leq l_0$, $|\partial_\lambda^l e_+ - \partial_\lambda^l e|_{\bar{\mathcal{O}}} \leq \frac{\gamma^{3b} s \mu}{\eta^{l_0}},$ $|\left(
   \begin{array}{c}
     \partial_\lambda^l y_0 \\
     \partial_\lambda^l z_0 \\
   \end{array}
 \right)
|_{\bar{\mathcal{O}}}\leq \frac{\gamma^{3b} s \mu}{\eta^{l_0}}$ and $||\partial_\lambda^l M_+ - \partial_\lambda^l M||_{\bar{\mathcal{O}}} \leq \frac{\gamma^{3b} \mu}{\eta^{l_0}}$.
\end{lemma}
\begin{proof}
Obviously, $|\partial_\lambda^l p_{000}|_{\bar{\mathcal{O}}} \leq \frac{\gamma^{3b} s^2 \mu}{\eta^{l_0}}$, $|\partial_\lambda^l p_{010}|_{\bar{\mathcal{O}}},$ $|\partial_\lambda^l p_{001}|_{\bar{\mathcal{O}}} \leq \frac{\gamma^{3b} s \mu}{\eta^{l_0}}$, $|\partial_\lambda^l p_{020}|_{\bar{\mathcal{O}}},$ $|\partial_\lambda^l p_{011}|_{\bar{\mathcal{O}}},$ $|\partial_\lambda^l p_{002}|_{\bar{\mathcal{O}}} \leq \frac{\gamma^{3b} \mu}{\eta^{l_0}}.$
Then $||\partial_\lambda^l M_+ - \partial_\lambda^l M||_{\bar{\mathcal{O}}} \leq \frac{\gamma^{3b} \mu}{\eta^{l_0}}$ holds, if $s< \varepsilon^4$. Let $\mathfrak{B}=(M+ \partial_{(y, z)}^2 h)^T(M+ \partial_{(y, z)}^2 h)$. Then
$\hat{A} = M^T \partial_{(y, z)}^2 h + (\partial_{(y, z)}^2 h)^T M + (\partial_{(y, z)}^2 h)^T (\partial_{(y, z)}^2 h)$ is Hermitian. Then there is a unitary matrix $Q$ such that
$Q^* \hat{A} Q  = diag(\lambda_1, \cdots, \lambda_{2m}),$
where $\lambda_{\min} = \lambda_1 \leq \cdots \leq \lambda_{2m} = \lambda_{\max}.$ Moreover, if $s < \varepsilon^4$, we have
$|\lambda_{\min} |=| \min\limits_{\{x, 0\neq x \in S\}} \frac{x^* \hat{A} x}{x^* x}| \leq ||\hat{A}||_\infty\leq  s \leq \frac{1}{2} (\min\{\mu_1, \cdots, \mu_l\})^2,$
where $S = span\{x_1, \cdots, x_{n+2m}\}$, $\hat{A} x_i = \lambda_i x_i$, $1\leq i\leq n+2m,$ $||\hat{A}||_\infty =\max\limits_{i,j} \{ \hat{A}^{ij}\}$. According to Weyl Theorem, we have
\begin{eqnarray*}
\lambda_{\min} (\mathfrak{B}) = \lambda_{\min} (M^T M  + \hat{A})\geq \lambda_{\min} (M^T M ) +  \lambda_{\min} (\hat{A}) \geq \frac{1}{2} (\min\{ \mu_1, \cdots, \mu_l\})^2,
\end{eqnarray*}
which implies that $
\mathfrak{B} \geq \frac{1}{2} (\min\{ \mu_1, \cdots, \mu_l\})^2 I_{2m}.$
Then
\begin{eqnarray*}
\frac{1}{2} \big(\min\{\mu_1, \cdots, \mu_l\} \big)^2  (y_0^*, z_0^*) \left(
                                              \begin{array}{c}
                                                y_0 \\
                                                z_0 \\
                                              \end{array}
                                            \right) \leq
 (\varepsilon p_{010}^*, \varepsilon p_{001}^*)\left(
                                              \begin{array}{c}
                                                \varepsilon p_{010} \\
                                               \varepsilon p_{001} \\
                                              \end{array}
                                            \right).
\end{eqnarray*}
Therefore,
$|\left(
   \begin{array}{c}
     y_0 \\
     z_0 \\
   \end{array}
 \right)
|_{\mathcal{O}} \leq \gamma^{3b} s \mu.$
Hence, using Cauchy inequality, we have \\
$|\left(
   \begin{array}{c}
    \partial_\lambda^l y_0 \\
    \partial_\lambda^l z_0 \\
   \end{array}
 \right)
|_{\bar{\mathcal{O}}} \leq \frac{\gamma^{3b} s \mu}{\eta^{l_0}},$
which, together with the definition of $e_+$, also implies the first inequality of the lemma.

\end{proof}

\begin{remark}
As we have seen, the transformation $\phi$ remove the drift of frequency during the KAM step from $\nu$ to $\nu+1$. Consider the following transformation
$\phi_1:
x \rightarrow x, \left(
                          \begin{array}{c}
                            y \\
                            z \\
                          \end{array}
                        \right) \rightarrow \left(
                                              \begin{array}{c}
                                                y \\
                                                z \\
                                              \end{array}
                                            \right)+ \left(
                                              \begin{array}{c}
                                                y_1 \\
                                                z_1 \\
                                              \end{array}
                                            \right) $
such that for some $t$
\begin{eqnarray}\label{Eq50}
\left\{
  \begin{array}{ll}
  \frac{1}{2} \langle \left(
                                                               \begin{array}{c}
                                                                 y_1 \\
                                                                 z_1 \\
                                                               \end{array}
                                                             \right), \tilde{M} \left(
                                                               \begin{array}{c}
                                                                 y_1 \\
                                                                 z_1 \\
                                                               \end{array}
                                                             \right)
\rangle + \langle
\left(
  \begin{array}{c}
   \omega+ \varepsilon p_{010} \\
     \varepsilon p_{001} \\
  \end{array}
\right), \left(
           \begin{array}{c}
             y_1 \\
             z_1 \\
           \end{array}
         \right)\rangle
+ \tilde{p}= 0, \\
    M \left(
                                              \begin{array}{c}
                                                y_1 \\
                                                z_1 \\
                                              \end{array}
                                            \right)+\left(
                                              \begin{array}{c}
                                                \partial_y h \\
                                                \partial_z h \\
                                              \end{array}
                                            \right)   +  \left(
                                              \begin{array}{c}
                                                \varepsilon p_{010} \\
                                                \varepsilon p_{001} \\
                                              \end{array}
                                            \right) + \left(
                                                        \begin{array}{c}
                                                          t \omega \\
                                                          0 \\
                                                        \end{array}
                                                      \right) = 0,
  \end{array}~~~~~~
\right.
\end{eqnarray}
where $
\tilde{p}=\varepsilon p_{000} + h(y_1,z_1), \ \tilde{M} = M +  2\varepsilon\left(
                                                                         \begin{array}{cc}
                                                                           p_{k20} & \frac{1}{2} p_{k11} \\
                                                                           \frac{1}{2} p_{k11}^T & p_{k02}\\
                                                                         \end{array}
                                                                       \right).$
Under the following assumption
\begin{itemize}
\item[${(H3')}$] $\left(
            \begin{array}{cc}
              M & \omega \\
              \omega^T & 0 \\
            \end{array}
          \right)^T \left(
            \begin{array}{cc}
              M & \omega \\
              \omega^T & 0 \\
            \end{array}
          \right) \geq (\min\limits_{0\leq\hat{\iota}_1 \leq m, 0\leq\hat{\iota}_2 \leq l}\{\varepsilon_{\hat{\iota}_1}, \mu_{\hat{\iota}_2}\})^2 I_{n+ 2m+1},$
\end{itemize}
there is a $(x_1, y_1,t)$ such that $(\ref{Eq50})$ hold, which means that on a given energy surface there is a transformation $\phi_1$ such that the frequency of the transformed system is proportional to the frequency of initial system. Similar to the proof of Lemma \ref{Lm1}, we have $
|\left(
   \begin{array}{c}
    \partial_\lambda^l y_1 \\
    \partial_\lambda^l z_1 \\
   \end{array}
 \right)
|_{\bar{\mathcal{O}}} \leq \frac{\gamma^{3b} s \mu}{\eta^{l_0}}, \ for\ |l| \ \leq \ l_0.$
\end{remark}

\subsection{Homological Equations}\label{Homol.}
Denote
\begin{eqnarray*}
\mathcal{O}_+ &=& \{\lambda \in \mathcal{O}, |L_{k0}|\geq \frac{\min\limits_{0\leq\hat{\iota}_1 \leq m}\{\varepsilon_{\hat{\iota}_1}\} \gamma}{|k|^\tau}, \mathcal{A}_1^* \mathcal{A}_1 \geq (\frac{\min\limits_{\substack{0\leq\hat{\iota}_1 \leq m,\\ 0\leq\hat{\iota}_2 \leq l}}\{\varepsilon_{\hat{\iota}_1}, \mu_{\hat{\iota}_2}\} \gamma}{|k|^\tau})^2 I_{n + 2m},\\
&~&~~~\mathcal{A}_2^* \mathcal{A}_2 \geq (\frac{\min\limits_{0\leq\hat{\iota}_1 \leq m, 0\leq\hat{\iota}_2 \leq l}\{\varepsilon_{\hat{\iota}_1}, \mu_{\hat{\iota}_2}\} \gamma}{|k|^\tau})^2 I_{(n+ 2m)^2 }, 0 <|k| \leq K_+ \},
\end{eqnarray*}
where $\mathcal{A}_1$ $\mathcal{A}_2$, $L_{k0}$, $L_{k1}$ and $L_{k2}$ are defined as those in (\ref{Eq317}) - (\ref{Eq321}).

We will show that (\ref{706}) is solvable on $\mathcal{O}_+$. Let $\partial_z h(y,z) = \hat{h}_1 y+ \hat{h}_2 z.$ Substituting the Taylor-Fourier series of $F$ and $R$ into $(\ref{706})$ yields:
\begin{eqnarray}
\label{e1}\sqrt{-1} \langle k, \omega+ \Delta\rangle f_{k00} &=&  \varepsilon p_{k00},~~~~~~~\\
\label{e2}\sqrt{-1} \langle k, \omega+ \Delta\rangle f_{k10} - (M_{21}^T + \hat{h}_1) J f_{k01} &=& \varepsilon p_{k10},~~~~~~\\
\label{e3}\sqrt{-1} \langle k, \omega + \Delta\rangle f_{k01} - (M_{22}+ \hat{h}_2)J f_{k01}&= & \varepsilon p_{k01},~~~~~~\\
\label{e4}\sqrt{-1} \langle k, \omega + \Delta\rangle  f_{k20} -(M_{21}^T + \hat{h}_1) J f_{k11} &=& \varepsilon p_{k20},~~~~~~\\
\label{e5}\sqrt{-1} \langle k, \omega + \Delta\rangle  f_{k11} + 2 f_{k02} J (M_{21}+ \hat{h}_1) - (M_{22}+ \hat{h}_2) J f_{k11}&=& \varepsilon p_{k11},~~~~~~\\
\label{e6}\sqrt{-1} \langle k, \omega + \Delta\rangle  f_{k02} - (M_{22}^T+ \hat{h}_2) J f_{k02} + f_{k02} J (M_{22}+ \hat{h}_2) &=& \varepsilon p_{k02},~~~~~~
\end{eqnarray}
which are equivalent to
\begin{eqnarray}
\label{eq1}(L_{k0} + \varpi) f_{k00} &=& \varepsilon p_{k00},~~~~~~~~~~\\
\label{eq2} \left(
  \begin{array}{cc}
    L_{k0}I_n +\varpi I_n  & -(M_{21}^T+ \hat{h}_1) J \\
    0 & L_{k1} +\varpi I_{2m}- \hat{h}_2J \\
  \end{array}
\right) \left(
          \begin{array}{c}
            f_{k10} \\
            f_{k01} \\
          \end{array}
        \right)&=&  \left(
                   \begin{array}{c}
                     \varepsilon p_{k10} \\
                     \varepsilon p_{k01} \\
                   \end{array}
                 \right),~~~~~~\\
\label{eq3} \mathcal{A}\left(
         \begin{array}{c}
           f_{k20} \\
           f_{k11} \\
           f_{k02} \\
         \end{array}
       \right)&=&  \left(
                    \begin{array}{c}
                      \varepsilon p_{k20} \\
                      \varepsilon p_{k11} \\
                      \varepsilon p_{k02} \\
                    \end{array}
                  \right),~~~~~~
\end{eqnarray}
where $\varpi = \sqrt{-1} \langle k, \Delta\rangle$ and
\begin{eqnarray}
\nonumber \Delta &=&  M_{11} y + M_{12} z + \partial_y h(y,z),\\
\label{Eq323}\mathcal{A} &=&  \left(
  \begin{array}{ccc}
    I_n \otimes (L_{k0} I_n + \varpi I_n) & I_n \otimes ((M_{21}^T+ \hat{h}_1) J) & 0 \\
    0 & I_n \otimes (L_{k1} + \varpi I_{2m} -\hat{h}_2 J) & a_{23} \\
    0 & 0 & a_{33} \\
  \end{array}
\right),~~~\\
\label{Eq324} a_{23}&=&-((M_{21}^T+\hat{h}_1) J)\otimes I_{2m},\\
\label{Eq435} a_{33}&=& L_{k2} + \varpi I_{4m^2} - (\hat{h}_2J)\otimes I_{2m} - I_{2m} \otimes(\hat{h}_2J).
\end{eqnarray}

\begin{lemma}\label{Eq53}
Let $|s| \leq \varepsilon^4$ and assume
\begin{itemize}
\item[${(H4)}$] $s^{\frac{1}{2}} K_+^{\tau+1} = o(\gamma).$
\end{itemize}
 The following hold for all $0< |k| \leq K_+$.

\begin{itemize}
\item[{(i)}]
Homology equations (\ref{eq1})--(\ref{eq3}) can be solved on $\mathcal{O}_+$ successively to obtain functions $f_{k00},$ $f_{k01},$ $f_{k10},$ $f_{k11},$ $f_{k20},$ $f_{k02},$ $0 <|k| \leq K_+$, which are smooth in $\lambda \in \mathcal{O}_+$ and $
\bar{f}_{kij}(\bar{y}, \bar{z}) = f_{-kij}(y,z),$ for all $0\leq |i|+|j|\leq 2$, $0 < |k|\leq K_+$, $(y,z)\in D(s).$ Moreover, on $D(s)\times \bar{\mathcal{O}}_+$,
\begin{eqnarray*}
|\partial_\lambda^l f_{k00}|_{D(s)\times \bar{\mathcal{O}}_+} \leq \frac{\gamma^{3b-1} s^2 \mu \Gamma(r- r_+)}{ \eta^{l_0}},\\
|\partial_\lambda^l f_{k10}|_{D(s)\times \bar{\mathcal{O}}_+}, |\partial_\lambda^l f_{k01}|_{D(s)\times \bar{\mathcal{O}}_+} \leq \frac{\gamma^{3b-1} s \mu \Gamma(r- r_+)}{ \eta^{l_0}},\\
|\partial_\lambda^l f_{k02}|_{D(s)\times \bar{\mathcal{O}}_+}, |\partial_\lambda^l f_{k11}|_{D(s)\times \bar{\mathcal{O}}_+}, |\partial_\lambda^l f_{k20}|_{D(s)\times \bar{\mathcal{O}}_+} \leq \frac{\gamma^{3b-1} \mu  \Gamma(r- r_+)}{ \eta^{l_0}}.
\end{eqnarray*}
\item[{(ii)}] On $\hat{D}(s) \times \bar{\mathcal{O}}_+$,
$
|\partial_\lambda^l \partial_x^i \partial_{(y,z)}^j F| \leq \frac{ \gamma^{3b-1}  \mu  \Gamma(r - r_+) }{ \eta^{l_0}},~~ |l| \leq l_0,~~|i| \leq l_0,~~ |j| \leq 2.
$
\end{itemize}
\end{lemma}

\begin{proof}
According to $(\ref{eq1})$, we have $
f_{k00} = (L_{k0}+ \varpi)^{-1} p_{k00}.$
Since {(H4)} and $s\leq \varepsilon^2$, on $\mathcal{O}_+$ we have
$|L_{k0}+ \varpi|\geq \frac{1}{2} \frac{\min\{\varepsilon_1, \cdots, \varepsilon_m\} \gamma}{|k|^\tau}$,
which means that $|(L_{k0}+ \varpi)^{-1}| \leq \frac{2|k|^\tau}{\min\{\varepsilon_1, \cdots, \varepsilon_m\} \gamma}.$
Then $
|f_{k00}|_{D(s)\times\mathcal{O}_+} \leq  \gamma^{3b-1} s^2 \mu  \Gamma(r - r_+).$
Therefore,
\begin{eqnarray}
\label{Eq23}|\partial_\lambda^l f_{k00}|_{D(s)\times\bar{\mathcal{O}}_+} \leq  \frac{\gamma^{3b-1} s^2 \mu  \Gamma(r - r_+) }{ \eta^{l_0}}.
\end{eqnarray}

According to $(\ref{eq2})$, we have
\begin{eqnarray}\label{EQ1}
(f_{k10}^*, f_{k01}^*) \mathcal{A}_3^* \mathcal{A}_3  \left(
          \begin{array}{c}
            f_{k10} \\
            f_{k01} \\
          \end{array}
        \right) = \varepsilon^2 (p_{k10}^*, p_{k01}^*)\left(
                   \begin{array}{c}
                     p_{k10} \\
                     p_{k01} \\
                   \end{array}
                 \right),
\end{eqnarray}
where $\mathcal{A}_3 = \left(
  \begin{array}{cc}
    L_{k0}I_n +\varpi I_n  & -(M_{21}^T+ \hat{h}_1) J \\
    0 & L_{k1} +\varpi I_{2m}- \hat{h}_2J \\
  \end{array}
\right).$ Directly, $
\mathcal{A}_3^* \mathcal{A}_3 = \mathcal{A}_1^* \mathcal{A}_1 + \mathcal{A}_4,$
where $\hat{h}_1=O(s),$ $\hat{h}_2=O(s),$ $\mathcal{A}_4 = \left(
                                      \begin{array}{cc}
                                        a_{11} & a_{12}\\
                                        a_{21} & a_{22} \\
                                      \end{array}
                                    \right),$ and
\begin{eqnarray*}
a_{11}&=& (\bar{L}_{k0} \varpi + \bar{\varpi} L_{k0} + \bar{\varpi}\varpi)I_n,\\
a_{12}&=& - \bar{\varpi} I_n \cdot M_{21}^T J  - \bar{L}_{k0}I_n\cdot (\hat{h}_1 J)- \bar{\varpi}I_n \cdot (\hat{h}_1 J),\\
a_{21}&=&-(M_{21}^T J)^* \cdot\varpi I_n - (\hat{h}_1J)^*\cdot (L_{k0} I_n) - (\hat{h}_1J)^*\cdot(\varpi I_n),\\
a_{22}&=& L_{k1}^*\cdot\varpi I_{2m} + \bar{\varpi}I_{2m}\cdot L_{k1}+\bar{\varpi}\varpi I_{2m}+ (M_{21}^T J)^* \cdot (\hat{h}_1 J)+ (\hat{h}_1 J)^* (M_{21}^T J)\\
&~&+ (\hat{h}_1 J)^* \cdot (\hat{h}_1 J)+ L_{k1}^* \cdot (\hat{h}_2 J)- (\varpi I_{2m})^* \cdot (\hat{h}_2 J)\\
&~& - (\hat{h}_2 J)^* \cdot L_{k1}- (\hat{h}_2 J )^* \cdot (\varpi I_{2m}) + (\hat{h}_2 J)^* \cdot (\hat{h}_2 J).
\end{eqnarray*}
Obviously, $\mathcal{A}_4$ is Hermitian. Then there is a unitary matrix $Q$ such that
$Q^* \mathcal{A}_4 Q  = diag(\lambda_1, \cdots, \lambda_{n+ 2m}),$
where $\lambda_{\min} = \lambda_1 \leq \cdots \leq \lambda_{n+ 2m} = \lambda_{\max}.$ Moreover, if $s < \varepsilon^4$, we have
$$|\lambda_{\min} |=| \min\limits_{\{x, 0\neq x \in S\}} \frac{x^* \mathcal{A}_4 x}{x^* x}|\leq ||\mathcal{A}_4||_\infty\leq |k| s\leq |k|s^{\frac{1}{2}} \varepsilon^2 \leq \frac{1}{2} (\frac{\min\limits_{\substack{0\leq\hat{\iota}_1 \leq m,\\ 0\leq\hat{\iota}_2 \leq l}}\{\varepsilon_{\hat{\iota}_1}, \mu_{\hat{\iota}_2}\} \gamma}{|k|^\tau})^2,$$
where $S = span\{x_1, \cdots, x_{n+2m}\}$, $\mathcal{A}_4 x_i = \lambda_i x_i$, $1\leq i\leq n+2m,$ $||\mathcal{A}_4||_\infty =\max\limits_{i,j} \{\mathcal{A}_4^{ij}\}$. According to Weyl Theorem, we have
$$
\lambda_{\min} (\mathcal{A}_3^* \mathcal{A}_3) \geq \lambda_{\min} (\mathcal{A}_1^* \mathcal{A}_1) +  \lambda_{\min} (\mathcal{A}_4)\geq \frac{1}{2} (\frac{\min\limits_{0\leq\hat{\iota}_1 \leq m, 0\leq\hat{\iota}_2 \leq l}\{\varepsilon_{\hat{\iota}_1}, \mu_{\hat{\iota}_2}\} \gamma}{|k|^\tau})^2,$$
which implies that $\mathcal{A}_3^* \mathcal{A}_3 \geq \frac{1}{2} ( \frac{ \min\limits_{0\leq\hat{\iota}_1 \leq m, 0\leq\hat{\iota}_2 \leq l}\{\varepsilon_{\hat{\iota}_1}, \mu_{\hat{\iota}_2}\}  \gamma}{|k|^\tau})^2 I_{n+2m}.$
Therefore,
\begin{eqnarray*}
(f_{k10}^*, f_{k01}^*) \frac{1}{2} ( \frac{\min\limits_{\substack{0\leq\hat{\iota}_1 \leq m,\\ 0\leq\hat{\iota}_2 \leq l}}\{\varepsilon_{\hat{\iota}_1}, \mu_{\hat{\iota}_2}\}  \gamma}{|k|^\tau})^2 I_{n+2m} \left(
          \begin{array}{c}
            f_{k10} \\
            f_{k01} \\
          \end{array}
        \right) \leq  \varepsilon^2(p_{k10}^* p_{k10}+ p_{k01}^* p_{k01}),
\end{eqnarray*}
i.e., $|f_{k01}|\leq \gamma^{3b-1} s \mu \Gamma(r - r_+)$ and  $|f_{k10}|\leq \gamma^{3b-1} s \mu \Gamma(r - r_+).$
Hence
\begin{eqnarray}
\label{Eq21} |\partial_\lambda^l f_{k10}|, |\partial_\lambda^l f_{k01}|&\leq&\frac{ \gamma^{3b-1} s \mu \Gamma(r - r_+) }{ \eta^{l_0}}.
\end{eqnarray}

According to $(\ref{eq3}),$ we have
\begin{eqnarray*}
(f_{k20}^*, f_{k11}^*, f_{k02}^*)\mathcal{A}^*\mathcal{A}\left(
         \begin{array}{c}
           f_{k20} \\
           f_{k11} \\
           f_{k02} \\
         \end{array}
       \right)&=& (\varepsilon p_{k20}^*,\varepsilon p_{k11}^*,\varepsilon p_{k02}^*) \left(
                    \begin{array}{c}
                     \varepsilon p_{k20} \\
                     \varepsilon p_{k11} \\
                     \varepsilon p_{k02} \\
                    \end{array}
                  \right),
\end{eqnarray*}
where $\mathcal{A}$, $a_{23}$, $a_{33}$ defined as those in (\ref{Eq323}), (\ref{Eq324}), (\ref{Eq435}) and
\begin{eqnarray*}
\mathcal{A}^* &=&  \left(
  \begin{array}{ccc}
    I_n \otimes (L_{k0} I_n + \varpi I_n)^* & 0 & 0 \\
    I_n \otimes ((M_{21}^T+ \hat{h}_1) J)^* & I_n \otimes (L_{k1} + \varpi I_{2m} -\hat{h}_2 J)^* & 0 \\
    0 & -((M_{21}^T+\hat{h}_1) J)^*\otimes I_{2m} & \breve{a}_{33} \\
  \end{array}
\right),\\
\breve{a}_{33} &=& L_{k2}^* + \bar{\varpi} I_{4m^2} - (\hat{h}_2J)^*\otimes I_{2m} - I_{2m} \otimes(\hat{h}_2J)^*.
\end{eqnarray*}
Here, we use the fact that $(A\otimes B)^* = A^* \otimes B^*$.
Further, we have\\
$\mathcal{A}^*\mathcal{A} \geq \frac{1}{2}(\frac{\min\limits_{0\leq\hat{\iota}_1 \leq m, 0\leq\hat{\iota}_2 \leq l}\{\varepsilon_{\hat{\iota}_1}, \mu_{\hat{\iota}_2}\} \gamma}{|k|^\tau})^2 I_{(n+ 2m)^2}.$
(We place this proof on Appendix \ref{Nonsin}.) Hence, $|f_{k20}|, \ |f_{k11}|, \  |f_{k02}|\leq\gamma^{3b-1} \mu  \Gamma(r - r_+).$
Moreover,
\begin{eqnarray}
\label{Eq20}|\partial_\lambda^l f_{k20}|, |\partial_\lambda^l f_{k11}|, |\partial_\lambda^l f_{k02}|&\leq&\frac{\gamma^{3b-1} \mu  \Gamma(r - r_+) }{ \eta^{l_0}}.
\end{eqnarray}

Putting (\ref{Eq23}), (\ref{Eq21}) - (\ref{Eq20}) into (\ref{709}), by direct calculation, we get the conclusion (2).

\end{proof}

\subsection{New Perturbation}\label{NP}
Next, we will estimate the following new perturbation
$$P_+ = \big(\int_0 ^1 \{R_t,F\}\circ \phi_F^t dt + \varepsilon(P - R )\circ \phi_F^1 \big)\circ\phi + \varepsilon \langle \left(
                                                     \begin{array}{c}
                                                       y \\
                                                       z \\
                                                     \end{array}
                                                   \right), \left(
                                                                \begin{array}{cc}
                                                                  p_{k20} & \frac{1}{2} p_{k11} \\
                                                                  \frac{1}{2} p_{k11}^T & p_{k20} \\
                                                                \end{array}
                                                              \right)
\left(
                                                     \begin{array}{c}
                                                       y_0\\
                                                       z_0 \\
                                                     \end{array}
                                                   \right) \rangle. $$

Recall that $
R_t = (1-t) \{N, F\} +\varepsilon R = \varepsilon\big(t R + (1-t) [R]\big).$
Then
\begin{eqnarray*}
\{R_t, F\} &=& \partial_x R_t \partial_y F - \partial_y R_t \partial_x F + \partial_z R_t J \partial_z F\\
&=& \varepsilon\big(\partial_x (t R) \partial_y F - \partial_y (t R + (1-t) [R]) \partial_x F + \partial_z (t R + (1-t) [R]) J \partial_z F\big).
\end{eqnarray*}
Further, $
|\{R_t, F\}| \leq \varepsilon\gamma^{3b} s^2 \mu^2  \Gamma(r - r_+).$

Assume
\begin{itemize}
\item[{(H5)}] $c \mu\Gamma( r- r_+) < \frac{1}{8} (r -r_+),$
\item[{(H6)}] $c \mu \Gamma ( r- r_+) < \frac{1}{8}\alpha,$
\item[{(H7)}] $\mu  < \frac{1}{8}\alpha.$
\end{itemize}
Standardly, for all $0\leq t\leq1$,
\begin{eqnarray}\label{812}
\phi_F^t&:& D_{\frac{\alpha}{4}} \rightarrow D_{\frac{\alpha}{2}},\\
\phi&:& D_{\frac{\alpha}{8}} \rightarrow D_{\frac{\alpha}{4}}
\end{eqnarray}
are well defined, real analytic and dependent smoothly on $\lambda \in {\mathcal{O}}_+$. Moreover, there is a constant $c$ such that, for all $0\leq t\leq 1$, $|l|\leq l_0$, $
|\partial_\lambda^l \phi_F^t|_{D_{\frac{\alpha}{4} }\times \bar{\mathcal{O}}_+} \leq \mu \Gamma(r -r_+).$
Then $|\int_0^1 \{R_t, F\} \circ \phi_F^t dt|_{D_{\frac{\alpha}{4}} \times \mathcal{O}_+} \leq \varepsilon\gamma^{3b} s^2 \mu^2 \Gamma(r - r_+),$ and $|( P- R)\circ \phi_F^1|_{D_{\frac{\alpha}{4}} \times \mathcal{O}_+} \leq \gamma^{3b} s^2 \mu^2.$
Therefore, $ |\bar{P}_+|= |\int_0 ^1 \{R_t,F\}\circ \phi_F^t dt +\varepsilon (P - R )\circ \phi_F^1| \leq c\varepsilon\gamma^{3b} s^2 \mu^2  \Gamma(r - r_+).$
Moreover, $
|\partial_\lambda^l \bar{P}_+|\leq c \varepsilon \frac{\gamma^{3b} s^2 \mu^2  \Gamma(r - r_+)}{\eta^{l_0}}.$

According to {Lemma \ref{Lm1}}, we get $
|\partial_\lambda^l \phi|_{\bar{\mathcal{O}}_+} \leq \frac{\gamma^{3b} s \mu}{\eta^{l_0}}.$
Hence
\begin{eqnarray*}
&~& |\partial_\lambda ^l P_+|_{D_{\frac{\alpha}{8}}\times \bar{\mathcal{O}}_+}\\
&=&|\partial^{l} \bar{P}_+ \circ \phi+ \varepsilon\partial^{l} \langle \left(
                                                                   \begin{array}{c}
                                                                     y \\
                                                                     z \\
                                                                   \end{array}
                                                                 \right),
\left(
  \begin{array}{cc}
    p_{020} & \frac{1}{2} p_{011} \\
    \frac{1}{2} p_{011}^T & p_{020} \\
  \end{array}
\right)\left(
         \begin{array}{c}
           y_0 \\
           z_0 \\
         \end{array}
       \right)
\rangle|_{D_{\frac{\alpha}{8}}\times \bar{\mathcal{O}}_+}\\
&\leq& \varepsilon\frac{\gamma^{3b} s^2 \mu^2 \Gamma(r - r_+)}{\eta^{l_0}} + \varepsilon\frac{\gamma^{3b} s^2 \mu^2 \gamma^b}{\eta^{l_0}} \leq \varepsilon\frac{\gamma^{3b} s^2 \mu^2}{\eta^{l_0}} \big(\Gamma(r - r_+) + \gamma^b\big)\\
&\leq&\varepsilon \frac{\gamma_+^{3b} s_+^2 \mu_+}{\eta_+^{l_0}},
\end{eqnarray*}
under the following assumption
\begin{itemize}
\item[\bf{(H8)}] $\gamma^{3b} \mu^\sigma (\Gamma(r - r_+) + \gamma^b) \leq \gamma_+^{3b}.$
\end{itemize}

Besides, for $\lambda\in \mathcal{O}_+$, it is easy to check $|L_{k0}^+|  \geq \frac{\min \{\varepsilon_i, \cdots, \varepsilon_m\} \gamma_+}{|k|^\tau},$  $(\mathcal{A}_1^+)^* \mathcal{A}_1^+ \geq (\frac{\min\limits_{0\leq\hat{\iota}_1 \leq m, 0\leq\hat{\iota}_2 \leq l}\{\varepsilon_{\hat{\iota}_1}, \mu_{\hat{\iota}_2}\} \gamma_+}{|k|^\tau})^2 I_{n + 2m},$ $(\mathcal{A}_2^+)^* \mathcal{A}_2^+ \geq (\frac{\min\limits_{0\leq\hat{\iota}_1 \leq m, 0\leq\hat{\iota}_2 \leq l}\{\varepsilon_{\hat{\iota}_1}, \mu_{\hat{\iota}_2}\} \gamma_+}{|k|^\tau})^2 I_{(n+ 2m)^2 },$
where the definitions of $L_{k0}^+$, $\mathcal{A}_1^+$ and $\mathcal{A}_2^+$ are similar to $L_{k0}$, $\mathcal{A}_1$ and $\mathcal{A}_2$, respectively.

We have accomplished a KAM step from $\nu$ to $\nu+1$ under some conditions so far. Next we will introduce a series of iteration sequences, under which the KAM step mentioned above could be iterated infinitely. {Most of the iteration sequences comes from \cite{Li}. }

\subsection{Iteration Lemma}\label{727}
Let $r_0$, $\gamma_0$, $s_0$, $\mu_0$, $\eta_0$, $\mathcal{O}_0$, $\bar{\mathcal{O}}_0$, $H_0$, $N_0$, $P_0$ be given as above. For any $\nu = 0,1, \cdots,$  denote
\begin{eqnarray*}
r_\nu &=& r_0 (1 - \sum_{i=1}^\nu \frac{1}{2^{i+1}}), ~~~\gamma_\nu = \gamma_0 (1 - \sum_{i=1}^\nu \frac{1}{2^{i+1}}),~~~\eta_\nu = \eta_0 (1 - \sum_{i=1}^\nu \frac{1}{2^{i+1}}),\\
\mu_\nu &=& (64c_0)^{\frac{1}{1- \lambda_0}} \mu_{\nu-1}^{1 + \sigma}, ~~~K_\nu = ([\log\frac{1}{\mu_{\nu-1}}]+1)^{3{\eta}},~~~\alpha_\nu = \mu_\nu^{\frac{1}{3}},\\
D_\nu &=& D(r_\nu, s_\nu),~~~~~~~s_\nu = \frac{1}{8} \alpha_{\nu-1} s_{\nu-1},\\
\mathcal{O}_\nu &=& \{\lambda \in \mathcal{O}_{\nu-1}: |L_{k0}^\nu| > \frac{\min\limits_{0\leq\hat{\iota}_1 \leq m}\{\varepsilon_{\hat{\iota}_1}\} \gamma}{|k|^\tau},  (\mathcal{A}_1^*)^{\nu} \mathcal{A}_1^{\nu} \geq (\frac{\min\limits_{\substack{0\leq\hat{\iota}_1 \leq m,\\ 0\leq\hat{\iota}_2 \leq l}}\{\varepsilon_{\hat{\iota}_1}, \mu_{\hat{\iota}_2}\} \gamma}{|k|^\tau})^2 I_{n + 2m},\\
&~&~~~ (\mathcal{A}_2^*)^{\nu} (\mathcal{A}_2)^{\nu} \geq (\frac{\min\limits_{0\leq\hat{\iota}_1 \leq m, 0\leq\hat{\iota}_2 \leq l}\{\varepsilon_{\hat{\iota}_1}, \mu_{\hat{\iota}_2}\} \gamma}{|k|^\tau})^2 I_{(n+ 2m)^2 }, ~~~0<|k|\leq K_\nu\}.
\end{eqnarray*}

We have the following lemma.
\begin{lemma}
The KAM step described above is valid for all $\nu = 0,1,\cdots$, and the following hold for all $\nu = 1,2,\cdots.$
 \begin{itemize}
\item[{(i)}] $P_\nu$ is real analytic in $(x,y)\in D_\nu$ and smooth in $\lambda \in \mathcal{O}_\nu$, and moreover, $
|\partial_\lambda^l P_\nu |_{D_\nu \times \mathcal{O}_\nu} \leq \frac{\gamma_\nu^{3b} s_\nu^{2} \mu_\nu}{\eta_\nu^{l_0}}, ~~~|l| \leq l_0;$
\item[{(ii)}] $\phi_F^t: D_\nu \times {\mathcal{O}}_\nu \rightarrow D_{\nu-1}$, is symplectic for each $\lambda \in {\mathcal{O}}_\nu$, and is of class $C^{\alpha, l_0}$, respectively, where $\alpha$ stands for real analyticity. Moreover,
$H_\nu = H_{\nu-1}\circ \phi_F^t \circ \phi = N_\nu + P_\nu,$
on ${D_\nu} \times {\mathcal{O}}_\nu$, and $
|\phi_F^t - id | \leq c_0 \gamma^{3b-1} \frac{\mu_0}{2^\nu},$ $|\phi- id | \leq c_0 \gamma^{3b-1} \frac{\mu_0}{2^\nu};$
\item[\bf{(3)}] $\mathcal{O}_\nu = \{\lambda\in \mathcal{O}_{\nu-1}: |L_{k0}| > \frac{\min\limits_{0\leq\hat{\iota}_1 \leq m}\{\varepsilon_{\hat{\iota}_1}\}\gamma_\nu}{|k|^\tau}, \mathcal{A}_1^* \mathcal{A}_1 \geq (\frac{\min\limits_{\substack{0\leq\hat{\iota}_1 \leq m,\\ 0\leq\hat{\iota}_2 \leq l}}\{\varepsilon_{\hat{\iota}_1}, \mu_{\hat{\iota}_2}\} \gamma_\nu}{|k|^\tau})^2 I_{n + 2m},\\
 ~~~~~~~~\mathcal{A}_2^* \mathcal{A}_2 \geq (\frac{\min\limits_{0\leq\hat{\iota}_1 \leq m, 0\leq\hat{\iota}_2 \leq l}\{\varepsilon_{\hat{\iota}_1}, \mu_{\hat{\iota}_2}\} \gamma_\nu}{|k|^\tau})^2 I_{(n+ 2m)^2},~~K_{\nu-1}<|k|\leq K_\nu\}$.
\end{itemize}
\end{lemma}

\begin{proof}
The proof of this lemma is equivalent to verify conditions {(H1)}--{(H8)} under iteration sequences. The verifications of {(H1)}--{(H2)}, {(H5)}--{(H7)} are standard. For example, refer to \cite{Li,Li7}. And assumption {(H3)} holds obviously according to the definition of $M_+$ when $s< \varepsilon^4$. Directly,
 \begin{eqnarray*}
 \mu_\nu &=& (64c_0)^{\frac{(1+\sigma)^{\nu} -1}{(1- \lambda_0)\sigma} } \mu_0 ^{(1+\sigma)^\nu},\\
 s_\nu &=& (\frac{1}{8})^{\nu} (64c_0)^{\frac{1}{3\sigma (1-\lambda_0)}(\frac{(1+\sigma)^\nu -1 -\sigma}{\sigma} - \nu +1)} \mu_0^{\frac{(1+\sigma)^\nu - 1}{3\sigma}} s_0,
 \end{eqnarray*}
 where $\sigma < 1.$ Then
 \begin{eqnarray*}
K_+ &=& ([\log \frac{1}{\mu}]+1)^{3\eta} = ([\log \frac{1}{(64 c_0)^{\frac{1}{1-\lambda_0} \frac{(1+\sigma)^\nu -1}{\sigma}} \mu_0^{(1+\sigma)^{\nu}}}]+1)^{3\eta}\\
&=& ([- (1+\sigma)^{\nu}(\frac{1}{(1-\lambda_0)}\log 64c_0 + \log \mu_0) + \frac{1}{(1- \lambda_0) \sigma} \log 64c_0] +1)^{3\eta}.
\end{eqnarray*}
Therefore, assumption {(H4)} holds.
 Next, under iteration sequences we will check assumption {(H8)}, which is equivalent to $
 \mu^\sigma \Gamma(r - r_+) + \mu^\sigma \gamma^b \leq \frac{\gamma_+^{3b}}{ \gamma^{3b}}.$
The inequality $ \mu^\sigma \Gamma(r - r_+) \leq \frac{1}{2} \frac{\gamma_+^{3b}}{ \gamma^{3b}}$ is standard and we omit the detail. We will show
\begin{eqnarray}\label{Eq30}
\mu^\sigma \gamma^b \leq \frac{1}{2}\frac{\gamma_+^{3b}}{ \gamma^{3b}}.
\end{eqnarray}
On the one hand, we have $\mu^{\sigma} \gamma^b \leq (\frac{\mu_0}{2^\nu})^{\sigma} \gamma_0 (1- \sum\limits_{i=1}^{\nu} \frac{1}{2^{i+1}}) < \frac{\mu_0^{\sigma} \gamma_0}{2^{\sigma\nu}},$
and $(\frac{\gamma_+}{\gamma})^{3b} = (1 - \frac{1}{2^{\nu+1} -2 })^{3b}.$
On the other hand, we get $2^{\sigma x} (1- \frac{1}{2^{x+1} -2})^{3b} > 2^{\sigma - 3b},$
since
\begin{eqnarray*}
&~&\frac{d}{dx}2^{\sigma x} (1- \frac{1}{2^{x+1} -2})^{3b}\\
&=&\frac{d (2^{\sigma x})}{dx} (1- \frac{1}{2^{x+1} -2})^{3b} + 2^{\sigma x} \frac{d\big( (1- \frac{1}{2^{x+1} -2})^{3b} \big)}{ dx}\\
&=& 2^{\sigma x} ln 2 (1- \frac{1}{2^{x+1} -2})^{3b-1} \big( \sigma (1- \frac{1}{2^{x+1} -2}) + 3b \frac{2^{x+1}}{(2^{x+1} -2)^2}\big)\\
&>& 0,
\end{eqnarray*}
where $x>1.$ Then for sufficient small $\mu_0$ and $\gamma_0$ we prove (\ref{Eq30}), i.e., for sufficient small $\mu_0$ and $\gamma_0$ {(H8)} holds.

\end{proof}

\subsection{Convergence and Measure Estimate}

The convergence is standard and we omit the detail. Let $\mathcal{O}_* = \bigcap\limits_{\nu = 0}^\infty \mathcal{O}_\nu$. We now show the measure estimate of $|\mathcal{O} \setminus \mathcal{O}_*|$, i.e. the measure of a set in which the nonresonant condition does not hold.
Let
\begin{eqnarray*}
R_{k, \nu+1} &=& \{\lambda \in \mathcal{O}_\nu: |L_{k0,\nu}(\lambda)|< \frac{\min\limits_{0\leq\hat{\iota}_1 \leq m}\{\varepsilon_{\hat{\iota}_1}\}\gamma_\nu}{|k|^{\tau+1}}, \mathcal{A}_{1,\nu}^* \mathcal{A}_{1,\nu} < (\frac{\min\limits_{\substack{0\leq\hat{\iota}_1 \leq m,\\ 0\leq\hat{\iota}_2 \leq l}}\{\varepsilon_{\hat{\iota}_1}, \mu_{\hat{\iota}_2}\}\gamma_\nu}{|k|^\tau})^2 I_{n + 2m},\\
&~&~~~or~ \mathcal{A}_{2,\nu}^* \mathcal{A}_{2,\nu} < (\frac{\min\limits_{\substack{0\leq\hat{\iota}_1 \leq m,\\ 0\leq\hat{\iota}_2 \leq l}}\{\varepsilon_{\hat{\iota}_1}, \mu_{\hat{\iota}_2}\} \gamma_\nu}{|k|^\tau})^2 I_{(n+ 2m)^2 }, ~ K_\nu< |k| \leq K_{\nu+1}\}.
\end{eqnarray*}
\begin{lemma}\label{estimate measure}
 Let $\tau > n(N+1)-1.$ Assume ${(D)}$, ${(M1)}$ and ${(M2)}$ hold. Then $|\mathcal{O} \setminus \mathcal{O}_*| = O(\gamma^{\frac{1}{N+1}})$ as $\gamma \rightarrow 0.$
\end{lemma}
\begin{proof}
Rewrite $
R_{k, \nu+1} = \hat{R}_{k, \nu+1} \bigcup \check{R}_{k, \nu+1} \bigcup \breve{R}_{k, \nu+1},$
where
\begin{eqnarray*}
\hat{R}_{k, \nu+1} &=& \{\lambda \in \mathcal{O}_\nu: |L_{k0,\nu}|< \frac{\min\limits_{0\leq j\leq m}\{\varepsilon_j\}\gamma_\nu}{|k|^{\tau+1}}, K_\nu< |k| \leq K_{\nu+1}\},\\
\check{R}_{k, \nu+1}&=& \{\lambda \in \mathcal{O}_\nu: \mathcal{A}_{1,\nu}^* \mathcal{A}_{1,\nu} < (\frac{\min\limits_{0\leq\hat{\iota}_1 \leq m, 0\leq\hat{\iota}_2 \leq l}\{\varepsilon_{\hat{\iota}_1}, \mu_{\hat{\iota}_2}\} \gamma_\nu}{|k|^\tau})^2 I_{n + 2m},\\
&~&~~~~~~~~~ K_\nu< |k| \leq K_{\nu+1}\},\\
\breve{R}_{k, \nu+1} &=&\{\lambda \in \mathcal{O}_\nu: \mathcal{A}_{2,\nu}^* \mathcal{A}_{2,\nu} < (\frac{\min\limits_{0\leq\hat{\iota}_1 \leq m, 0\leq\hat{\iota}_2 \leq l}\{\varepsilon_{\hat{\iota}_1}, \mu_{\hat{\iota}_2} \} \gamma_\nu}{|k|^\tau})^2 I_{(n+ 2m)^2 },\\
&~&~~~~~~~~~ K_\nu< |k| \leq K_{\nu+1}\}.
\end{eqnarray*}
Since
\begin{eqnarray*}
\mathcal{O} \setminus \mathcal{O}_* &=&  \bigcup\limits_{\nu = 0}^\infty \bigcup\limits_{K_\nu < |k| \leq K_{\nu+1}} R_{k, \nu+1} =\bigcup\limits_{\nu = 0}^\infty \bigcup\limits_{K_\nu < |k| \leq K_{\nu+1}} (\hat{R}_{k, \nu+1} \bigcup \check{R}_{k, \nu+1} \bigcup \breve{R}_{k, \nu+1}),
\end{eqnarray*}
we will estimate $|\hat{R}_{k, \nu+1}|,$  $|\check{R}_{k, \nu+1}|$ and $|\breve{R}_{k, \nu+1}|$ first.

Let $L_{k0,\nu} = |k|\langle \varsigma, \omega_\nu(\lambda)\rangle$, where $\varsigma = \frac{k}{|k|}\in S^n$, $S^n$ is a $n$-dimensional ball. According to Taylor series, for given $\lambda_0\in \mathcal{O}_\nu$,
\begin{eqnarray*}
{L}_{k0, \nu}= |k|\varsigma^T \Omega_\nu (\lambda_0) \tilde{\lambda},
\end{eqnarray*}
where $\Omega_\nu (\lambda_0) = \big( \omega_\nu(\lambda_0), \cdots, \partial_\lambda^\alpha \omega_\nu(\lambda_0), \int_0^1 (1-t)^{|\alpha+1|} \partial_\lambda^ {\alpha+1}\omega_\nu(\lambda_0+ t \lambda) \big)$, $\hat{\lambda} = \lambda - \lambda_0 = (\hat{\lambda}_1, \cdots, \hat{\lambda}_n),$ $\tilde{\lambda}=(1, \hat{\lambda}, \cdots, \hat{\lambda}^\alpha, \hat{\lambda}^{\alpha+1})$.
Let $Q_{\lambda_0, \nu}= (q_{ij})_{\breve{\iota}\times \breve{\iota}}$ be a matrix such that $\Omega_\nu(\lambda_0) Q_{\lambda_0, \nu}$ means that only some columns of $\Omega_\nu(\lambda_0)$ are changed.
Using condition $(D)$, $rank \Omega_\nu(\lambda_0) = n$ for $\lambda_0\in \mathcal{O}_\nu\subset \mathcal{O}$, i.e., there is an matrix $Q_{\lambda_0,\nu}= (q_{ij})_{\breve{\iota}\times \breve{\iota}}$ such that $\Omega_\nu(\lambda_0)Q_{\lambda_0,\nu} = \big(A_\nu (\lambda_0), B_\nu(\lambda_0)\big)$, where $A_\nu(\lambda_0) = (a_{ij})_{n\times n}$ is nonsingular. Denote $\mathcal{O}_{\lambda_0,\nu}$ the neighborhood of $\lambda_0$ and $\bar{\mathcal{O}}_{\lambda_0,\nu}$ the closure of $\mathcal{O}_{\lambda_0,\nu}$. Then $\det A_\nu(\lambda) \neq 0$ for $\lambda \in \bar{\mathcal{O}}_{\lambda_0, \nu}$. Therefore, there is an orthogonal matrix $Q_{\lambda_0,\nu}$ such that $\Omega_\nu(\lambda) Q_{\lambda_0, \nu} = (A_{\nu}(\lambda), B_{\nu}(\lambda))$ for $\lambda\in \bar{\mathcal{O}}_{\lambda_0, \nu}$, where $\det A_{\nu}(\lambda) \neq 0$ on $\bar{\mathcal{O}}_{\lambda_0, \nu}.$
Let $\check{\lambda}_{1, \nu}\leq \cdots \leq\check{\lambda}_{n,\nu}$ be the eigenvalues of $(A_{\nu}(\lambda)A_{\nu}^* (\lambda)+ B_{\nu}(\lambda)B_{\nu}^* (\lambda))$. Since $rank (A_{\nu}(\lambda)A_{\nu}^* (\lambda)+ B_{\nu}(\lambda)B_{\nu}^* (\lambda)) = rank (A_{\nu}(\lambda), B_{\nu}(\lambda))$ (\cite{Horn}), there is a unitary $U_{\nu}$ and a real diagonal $V_{\nu} = diag(\check{\lambda}_{1,\nu}, \cdots, \check{\lambda}_{n,\nu})$ such that $(A_{\nu}(\lambda)A_{\nu}^* (\lambda)+ B_{\nu}(\lambda)B_{\nu}^* (\lambda)) = U_{\nu} V_{\nu} U_{\nu}^*$. Therefore, using Poincar\'{e} separation theorem and Lemma \ref{Eigenvalue} in Appendix \ref{Inverse},
\begin{eqnarray*}
\varsigma^* (A_{\nu}(\lambda)A_{\nu}^* (\lambda)+ B_{\nu}(\lambda)B_{\nu}^* (\lambda)) \varsigma &=& \varsigma^* (A_{\nu}(\lambda)A_{\nu}^* (\lambda)+ B_{\nu}(\lambda)B_{\nu}^* (\lambda)) \varsigma\\
 &=& \varsigma^*U_{\nu}  diag(\check{\lambda}_{1,\nu}, \cdots, \check{\lambda}_{n,\nu}) U_{\nu}^* \varsigma\\
 &\geq& \varsigma^*U_{\nu}  \check{\lambda}_{1,\nu} I_n U_{\nu}^*\varsigma\\
 &\geq& \check{\lambda}_{1, \nu}\\
 &\geq&c\min\limits_{0\leq j\leq m}\{\varepsilon_j^2\},
\end{eqnarray*}
where $c$ is positive and depends on the elements of $\Omega_\nu(\lambda)$.
Since the nonzero eigenvalues of $\left(
                                    \begin{array}{cc}
                                      A_{\nu}^T (\lambda) \varsigma\varsigma^T A_{\nu}(\lambda) & A_{\nu}^T (\lambda) \varsigma\varsigma^T B_{\nu}(\lambda) \\
                                      B_{\nu}^T (\lambda) \varsigma\varsigma^TA_{\nu}(\lambda) & B_{\nu}^T(\lambda) \varsigma\varsigma^T B_{\nu}(\lambda) \\
                                    \end{array}
                                  \right)
$ and $\varsigma^T (A_{\nu}(\lambda)A_{\nu}^T (\lambda)+ B_{\nu}(\lambda)B_{\nu}^T (\lambda))\varsigma$ are the same, there is an unitary matrix $\mathcal{U}_{\nu}(\lambda)$ such that
\begin{eqnarray*}
\left(
                                    \begin{array}{cc}
                                      A_{\nu}^T (\lambda) \varsigma\varsigma^T A_{\nu}(\lambda) & A_{\nu}^T (\lambda) \varsigma\varsigma^T B_{\nu}(\lambda) \\
                                      B_{\nu}^T (\lambda) \varsigma\varsigma^TA_{\nu}(\lambda) & B_{\nu}^T(\lambda) \varsigma\varsigma^T B_{\nu}(\lambda) \\
                                    \end{array}
                                  \right) = \mathcal{U}_{\nu}(\lambda) diag (0, \cdots,0, \check{\lambda}_\nu) \mathcal{U}_{\nu}^*(\lambda),
\end{eqnarray*}
where $\check{\lambda}_\nu = \varsigma^*U_{\nu}  diag(\check{\lambda}_{1,\nu}, \cdots, \check{\lambda}_{n,\nu}) U_{\nu}^* \varsigma.$ Let $a$ be the dimension of $\mathcal{U}_{\nu}$ and $\mathcal{U}_\nu = (u_{j_1i, \nu})_{a\times a}$. Using Hadamard's inequality (\cite{Horn}), $\det \mathcal{U}_{\nu}^* \mathcal{U}_{\nu} = \det I \leq \sum\limits_{0\leq j_1\leq a}u_{j_1i, \nu}^2$, i.e., there is a $i$ such that $\sum\limits_{0\leq j_1\leq a}u_{j_1i, \nu}^2\geq \frac{1}{2}$.  Denote $(\mathcal{U}_{\nu}^* Q_{\lambda_0,\nu} )_i$ the $i-$th row of $\mathcal{U}_{\nu}^* Q_{\lambda_0,\nu}$. Therefore, $||(\mathcal{U}_{\nu}^* Q_{\lambda_0,\nu}^{-1} \tilde{\lambda})_i||_2 = ||(\mathcal{U}_{\nu}^* Q_{\lambda_0, \nu}^{-1} )_i \tilde{\lambda} ||_2 \geq \sum\limits_{0\leq j_1\leq a}u_{j_1i,\nu}^2(\min\limits_{1\leq j\leq n} |\hat{\lambda}_j|)^{2N+2}$. Hence
\begin{eqnarray}
\label{Lk0} |{L}_{k0, \nu}^* {L}_{k0, \nu}| &=& |k|^2 |\tilde{\lambda}^T Q_{\lambda_0, \nu}Q_{\lambda_0, \nu}^{-1} \Omega_\nu^T \varsigma\varsigma^T \Omega_\nu Q_{\lambda_0, \nu}  Q_{\lambda_0, \nu}^{-1} \tilde{\lambda}|\\
\nonumber&=& |k|^2 |\tilde{\lambda}^T Q_{\lambda_0, \nu} \left(
                                          \begin{array}{cc}
                                            A_{\nu}^T (\lambda) \varsigma \varsigma^T A_{\nu}(\lambda) & A_{\nu}^T (\lambda) \varsigma\varsigma^T B_{\nu}(\lambda) \\
                                            B_{\nu}^T (\lambda) \varsigma\varsigma^T A_{\nu}(\lambda) & B_{\nu}^T(\lambda) \varsigma\varsigma^T B_{\nu}(\lambda) \\
                                          \end{array}
                                        \right)Q_{\lambda_0,\nu}^{-1} \tilde{\lambda}|\\
\nonumber&=&|k|^2 |\tilde{\lambda}^T Q_{\lambda_0, \nu}\mathcal{U}_{\nu} (\lambda) diag (0,\cdots, 0, \check{\lambda}_{\nu}) \mathcal{U}_{\nu}^*(\lambda)Q_{\lambda_0, \nu}^{-1} \tilde{\lambda}|\\
\nonumber&\geq& |k|^2 \check{\lambda}_{\nu} |(\mathcal{U}_{\nu}^*(\lambda) Q_{\lambda_0,\nu}^{-1} \tilde{\lambda})_{i}|\\
\nonumber&\geq& |k|^2 c\min\limits_{0\leq j\leq m}\{\varepsilon_j^2\} (\min\limits_{1\leq j \leq n} |\hat{\lambda}_i|)^{2N+2}.
\end{eqnarray}
Here, we use the following facts: for given matrix $A$, $B$ and orthogonal matrix $P$, $tr AB = tr BA$ and $tr P^{-1} A P = tr A$.
 Then
\begin{eqnarray*}
~&~& |\{ \lambda\in \mathcal{O}_\nu \bigcap \bar{\mathcal{O}}_{\lambda_0, \nu}: |{L}_{k0, \nu}^* {L}_{k0, \nu}| \leq \frac{\min\limits_{0\leq j\leq m} \{\varepsilon_j^2\}\gamma_\nu^2}{|k|^{2\tau}},K_{\nu}< |k|\leq K_{\nu+1} \}| \\
&<& |\{ \lambda\in \mathcal{O}_\nu \bigcap \bar{\mathcal{O}}_{\lambda_0, \nu}:  c\big(\min\limits_{1\leq j\leq n}|\hat{\lambda}_j|\big)^{2N + 2} \leq \frac{\gamma_\nu^2}{|k|^{2(\tau-1)}}, K_{\nu}< |k|\leq K_{\nu+1} \}|\\
 &\leq& c  D^{n-1} \frac{\gamma^{\frac{1}{N+1}}}{|k|^{\frac{\tau-1}{N+1}}},
\end{eqnarray*}
where $D$ is the exterior diameter of $\bar{\mathcal{O}}_{\lambda_0,\nu}$ with respect to the maximum norm, $n$ is the dimension of $\mathcal{O}_\nu$, $c$ is positive and depends on the elements of $\Omega_\nu(\lambda)$. Further, there are finite sets, $\bar{\mathcal{O}}_{\lambda_i, \nu}$, $1\leq i \leq \tilde{\iota}$, such that $\mathcal{O}_\nu\subset \bigcup\limits_{i=1}^{\tilde{\iota}} \bar{\mathcal{O}}_{\lambda_i, \nu}$ and
\begin{eqnarray*}
|{L}_{k0, \nu}^* (\lambda){L}_{k0, \nu}(\lambda)| > |k|^2 c\min\limits_{0\leq j\leq m}\{\varepsilon_j^2\} \big(\min\limits_{1\leq j\leq n}|\hat{\lambda}_j|\big)^{2N + 2} ~for ~\lambda\in \bar{\mathcal{O}}_{\lambda_i, \nu}.
\end{eqnarray*} Therefore,
\begin{eqnarray*}
|S_1| = |\{\lambda\in \mathcal{O}_\nu: |{L}_{k0,\nu}| \leq \frac{\min\limits_{0\leq j\leq m}\{\varepsilon_j\}\gamma_\nu}{|k|^\tau}, K_{\nu}< |k|\leq K_{\nu+1}\}|< c D^{n-1} \frac{\gamma_0^{\frac{1}{N+1}}}{|k|^{\frac{\tau-1}{N+1}}},
\end{eqnarray*}
where $c$ depends on $\mathcal{O}$, $D$, $n$ and $\Omega_{\nu}$.

According to Taylor series, we have $\mathcal{A}_{1,\nu}(\lambda) = \tilde{ \mathcal{A}}_{1,\nu} \Theta(\hat{\lambda}),$
where $\hat{\lambda}= \lambda - \lambda_0$, and
\begin{eqnarray*}
\tilde{ \mathcal{A}}_{1,\nu} &=&(\mathcal{A}_{1,\nu}(\lambda_0), \partial_\lambda \mathcal{A}_{1,\nu}(\lambda_0), \cdots, \partial_\lambda^\alpha \mathcal{A}_{1,\nu}(\lambda_0), \int_0^1 (1- t)^{|\alpha+1|} \partial_\lambda^{\alpha+1} \mathcal{A}_{1,\nu}(\lambda_0 + t \hat{\lambda}) dt),\\
\Theta(\hat{\lambda}) &=& (I_{n+2m},  I_{n+2m}\otimes \hat{\lambda}, \cdots, I_{n+2m} \otimes \hat{\lambda}^\alpha,  I_{n+2m} \otimes \hat{\lambda}^{\alpha+1})^T.
\end{eqnarray*}
By condition ${(M1)}$, there is an orthogonal matrix $Q_{\lambda_0,\nu}^1$ exchanging some column vector of $\tilde{ \mathcal{A}}_{1,\nu}(\lambda_0)$ such that $\tilde{ \mathcal{A}}_{1,\nu}(\lambda_0) Q_{\lambda_0,\nu}^1 = (A_\nu^1(\lambda_0), B_\nu^1(\lambda_0)),$
where $A_\nu^1(\lambda_0)$ is a $(n+2m) \times (n+2m)$ nonsingular matrix. Due to the continuity of determinant, there is a neighborhood $\mathcal{O}_{\lambda_0,\nu}$ of $\lambda_0$ with $\bar{\mathcal{O}}_{\lambda_0,\nu} \subset R^n$, such that $\det A_\nu^1(\lambda) \neq 0 ~for~ \lambda \in \bar{\mathcal{O}}_{\lambda_0,\nu}.$
In other words, there is an orthogonal matrix $Q_{\lambda_0,\nu}^1$ such that $\det A_\nu^1(\lambda) \neq 0 ~for~ \lambda \in \bar{\mathcal{O}}_{\lambda_0,\nu},$
where $\det A_\nu^1(\lambda) \neq 0$ on $\bar{\mathcal{O}}_{\lambda_0,\nu}$. Moreover,
\begin{eqnarray*}
\mathcal{A}_{1,\nu}(\lambda) &=& \tilde{\mathcal{A}}_{1,\nu}(\lambda_0)  Q_{\lambda_0,\nu}^1{ Q_{\lambda_0,\nu}^1}^{-1} \Theta = (A_\nu^1(\lambda_0), B_\nu^1(\lambda_0)) \tilde{\Theta},\\
\overline{\mathcal{A}_{1,\nu}^T}(\lambda) &=& \overline{\tilde{\Theta}^T} \left(
                                                     \begin{array}{c}
                                                      \overline{{ {A_\nu^1}^T}} \\
                                                       \overline{{ {B_\nu^1}^T}} \\
                                                     \end{array}
                                                   \right),\ \ \ \ \ \tilde{\Theta} = {Q_{\lambda_0,\nu}^1}^{-1} \Theta = \left(
                                                     \begin{array}{c}
                                                       \tilde{\Theta}_1 \\
                                                       \tilde{\Theta}_2 \\
                                                     \end{array}
                                                   \right),
\end{eqnarray*}
where $\tilde{\Theta}_1$ is a $(n+2m) \times (n+2m) $ matrix. We claim that $\tilde{\Theta}_1$ is nonsingular. In fact, there is only one nonzero element on each column of $\tilde{\Theta}_1$ according to the selection of $c_i$ in condition ${(M1)}$ and only one nonzero element on each row of $\tilde{\Theta}_1$ due to the definition of $\tilde{\Theta}$. Further, $\overline{\tilde{\Theta}_1^T} \tilde{\Theta}_1$ is a diagonal matrix and the corresponding elements on diagonal is greater than or equal to $(\min\limits_{1\leq j\leq n} |\hat{\lambda}_j|^{2N+2}),$ where $\hat{\lambda} = (\hat{\lambda}_1, \cdots, \hat{\lambda}_{n}).$ By condition ${(M1)}$ and Sylvester's law, $rank \left(
       \begin{array}{cc}
         \overline{{A_\nu^1}^T} A_\nu^1 & \overline{{A_\nu^1}^T} B_\nu^1 \\
        \overline{{B_\nu^1}^T} A_\nu^1 & \overline{{B_\nu^1}^T} B_\nu^1\\
       \end{array}
     \right)=n+ 2m.$
Then there is a unitary matrix $U_\lambda$ such that $$ \left(
       \begin{array}{cc}
         \overline{{A_\nu^1}^T} A_\nu^1 & \overline{{A_\nu^1}^T} B_\nu^1 \\
        \overline{{B_\nu^1}^T} A_\nu^1 & \overline{{B_\nu^1}^T} B_\nu^1\\
       \end{array}
     \right)  = U_\lambda \left(
                           \begin{array}{cc}
                             diag (\lambda_1, \cdots, \lambda_{n+2m}) & 0 \\
                             0 & 0 \\
                           \end{array}
                         \right)U_\lambda^*,$$
where $\lambda_i$, $1 \leq i\leq n+2m$, is nonvanishing eigenvalue of $\left(
       \begin{array}{cc}
         \overline{{A_\nu^1}^T} A_\nu^1 & \overline{{A_\nu^1}^T} B_\nu^1 \\
        \overline{{B_\nu^1}^T} A_\nu^1 & \overline{{B_\nu^1}^T} B_\nu^1\\
       \end{array}
     \right).$ Similarly to $(\ref{Lk0})$, we have
\begin{eqnarray*}
\mathcal{A}_{1,\nu}^* \mathcal{A}_{1,\nu} &=& |k|^2\overline{\tilde{\Theta}^T} \left(
       \begin{array}{cc}
         \overline{{A_\nu^1}^T} A_\nu^1 & \overline{{A_\nu^1}^T} B_\nu^1 \\
         \overline{{B_\nu^1}^T} A_\nu^1 & \overline{{B_\nu^1}^T} B_\nu^1\\
       \end{array}
     \right) \tilde{\Theta}\\
&=& |k|^2\overline{\tilde{\Theta}^T} U_{\lambda}\left(
       \begin{array}{cc}
         diag(\lambda_1, \cdots, \lambda_{n+2m}) & 0 \\
         0 & 0\\
       \end{array}
     \right) U_{\lambda}^* \tilde{\Theta}\\
&\geq&|k|^2 (\overline{\tilde{\Theta}_1^T},\overline{\tilde{\Theta}_2^T}) U_{\lambda} \left(
       \begin{array}{cc}
         \min\limits_{i} \lambda_i I_{n+2m} & 0 \\
         0 & 0\\
       \end{array}
     \right) U_{\lambda}^* \left(
               \begin{array}{c}
                  \tilde{\Theta}_1 \\
                 \tilde{\Theta}_2 \\
               \end{array}
             \right) \\
&\geq&|k|^2 c\min\limits_{0\leq\hat{\iota}_1 \leq m, 0\leq\hat{\iota}_2 \leq l}\{\varepsilon_{\hat{\iota}_1}^2, \mu_{\hat{\iota}_2}^2\} (\min\limits_{1\leq j \leq n} |\hat{\lambda}_i|)^{2N+2} I_{n+2m}.
\end{eqnarray*}
Therefore,
\begin{eqnarray*}
 &~& |\{ \lambda \in \mathcal{O}_{\nu}\cap \bar{\mathcal{O}}_{\lambda_0,\nu}:  \mathcal{A}_{1,\nu}^* \mathcal{A}_{1,\nu} < (\frac{\min\limits_{0\leq\hat{\iota}_1 \leq m, 0\leq\hat{\iota}_2 \leq l}\{\varepsilon_{\hat{\iota}_1}, \mu_{\hat{\iota}_2}\}  \gamma_\nu}{|k|^\tau})^2 I_{n + 2m} \}| \\
 &<& |\{\lambda \in \mathcal{O}_{\nu}\cap \bar{\mathcal{O}}_{\lambda_0,\nu}: c (\min\limits_{0\leq j \leq n} |\hat{\lambda}_j|^{2N+2}) I_{n+2m} < (\frac{ \gamma_\nu}{|k|^{\tau-1}})^2 I_{n + 2m} \}| \\
 &\leq& c D^{n-1} \frac{ \gamma_\nu^{\frac{1}{N +1}}}{|k|^{\frac{\tau-1}{N +1 }}},
\end{eqnarray*}
where $D$ denotes the exterior diameter of $\mathcal{O}_\nu$ with respect to the maximum norm, $n$ the dimension of the ambient space, $c$ is positive and depends on $\mathcal{O}_\nu$, $D$, $n$ and $\tilde{ \mathcal{A}}_{1,\nu}$. Since $\mathcal{O}_\nu$ is a closed bounded set, there are finite sets, $\bar{\mathcal{O}}_{\lambda_i, \nu}$, $0\leq i\leq n_1$, such that $\mathcal{O}_\nu \subset \bigcup\limits_{i=1}^{n_1} \bar{\mathcal{O}}_{\lambda_i,\nu}$ and $\mathcal{A}_{1,\nu}^*\mathcal{A}_{1,\nu} \geq c \min\limits_{0\leq\hat{\iota}_1 \leq m, 0\leq\hat{\iota}_2 \leq l}\{\varepsilon_{\hat{\iota}_1}^2, \mu_{\hat{\iota}_2}^2\} (\min\limits_{j} |\hat{\lambda}_j|^{2N+2}) I_{n+2m} ~ on ~\bar{\mathcal{O}}_{\lambda_i, \nu}.$ Then $\mathcal{A}_{1,\nu}^*\mathcal{A}_{1,\nu} \geq c \min\limits_{0\leq\hat{\iota}_1 \leq m, 0\leq\hat{\iota}_2 \leq l}\{\varepsilon_{\hat{\iota}_1}^2, \mu_{\hat{\iota}_2}^2\} (\min\limits_{j} |\hat{\lambda}_j|^{2N+2}) I_{n+2m}~ on ~\mathcal{O}_{\nu}.$
 Therefore, $|\check{R}_{k, \nu+1}|=|\{ \lambda \in \mathcal{O}_\nu:  \mathcal{A}_1^* \mathcal{A}_1 \leq (\frac{\min\limits_{0\leq\hat{\iota}_1 \leq m, 0\leq\hat{\iota}_2 \leq l}\{\varepsilon_{\hat{\iota}_1}, \mu_{\hat{\iota}_2}\}  \gamma}{|k|^\tau})^2 I_{n + 2m} \}|\leq c  D^{n-1} \frac{ \gamma^{\frac{1}{N +1}}}{|k|^{\frac{\tau-1}{N +1 }}},$
where $c$ is positive and depends on $\mathcal{O}$, $D$, $n$ and $\tilde{ \mathcal{A}}_{1,\nu}$.

Similarly, under condition ${(M2)}$, $$|\breve{R}_{k, \nu+1}| =
|\{ \lambda \in \mathcal{O}_\nu:  \mathcal{A}_{2,\nu}^* \mathcal{A}_{2,\nu} \leq (\frac{\min\limits_{0\leq\hat{\iota}_1 \leq m, 0\leq\hat{\iota}_2 \leq l}\{\varepsilon_{\hat{\iota}_1}, \mu_{\hat{\iota}_2}\}  \gamma_\nu}{|k|^\tau})^2 I_{(n + 2m)^2} \}| \leq c D^{n-1} \frac{ \gamma_\nu^{\frac{1}{N +1}}}{|k|^{\frac{\tau-1}{N +1 }}},$$
where $c$ depends on $\mathcal{O}$, $D$ and $n$.

Hence
\begin{eqnarray*}
|\mathcal{O} \setminus \mathcal{O}_*| &\leq&  |\bigcup\limits_{\nu = 0}^\infty \bigcup\limits_{K_\nu < |k| \leq K_{\nu+1}} R_{k, \nu+1}|\leq c O(\gamma^{\frac{1}{N +1}})\sum\limits_{\nu  =0}^\infty \sum\limits_{K_\nu < |k| \leq K_{\nu+1}} \frac{ 1}{|k|^{\frac{\tau-1}{N +1 }}}\\
&=& O(\gamma^{\frac{1}{N +1}}),
\end{eqnarray*}
as $\gamma\rightarrow 0$, where $\tau > n(N+1)+1.$

\end{proof}

\section{Normal Form}\setcounter{equation}{0} \label{041}
Consider the persistence of lower dimensional invariant tori for nearly integrable multi-scale Hamiltonian, for which there is no order relationship between the tangent frequency and the normal, which leads to the perturbation is not small enough to proceed KAM iteration and we need a KAM step on a neighbourhood of the origin that depends on $\varepsilon$.

In Section \ref{038}, we show KAM steps on $D(r,s)\times \mathcal{O}$ for the following nearly integrable multi-scale Hamiltonian systems:
\begin{eqnarray}
\nonumber H(x,y,z,\lambda)  = e+ \langle \omega , y\rangle + \frac{1}{2}\langle \left(
                                                                              \begin{array}{c}
                                                                                y \\
                                                                                z \\
                                                                              \end{array}
                                                                            \right),
          M  \left(
                                                                              \begin{array}{c}
                                                                                y \\
                                                                                z \\
                                                                              \end{array}
                                                                            \right)\rangle + h + \varepsilon P(x,y,z,\lambda),
\end{eqnarray}
where the definitions of $\omega,$ $M(\lambda)$ and $h$ are the same as those in (\ref{701}).
Moreover,  $|P(x,y,z,\lambda)|_{D(r,s) \times \mathcal{O}} \leq \gamma^{3b} s^2 \mu.$
And to insure the solvable of homological equations, we need to require $s \leq\varepsilon^4$, which means $|P| = o(\varepsilon^8)$. Actually, after $\tilde{\nu}_* = [\frac{\log 9}{\log 1+ \frac{1-\sigma}{a} }]+1$ KAM steps, $|P_{\tilde{\nu}_*}| = O(\varepsilon^9)$, where $\mu^a = \varepsilon.$ Next, we will show this process.

Recall the initial Hamiltonian system (\ref{701})
defined on $(x,y,z,\lambda) \in D(r,s)\times \mathcal{O}  = \{(x,y,z): |Im x|< r, |y|<s = \varepsilon^4, |z|<s= \varepsilon^4\} \times \{\lambda: |\lambda|\leq \delta_1\}\subset T^n\times G \times \mathcal{O}\subset T^n \times R^n\times R^{2m} \times R^n$, $\delta_1$ is a given constant, $0<\varepsilon\leq \varepsilon_i, \mu_j\ll 1$, $0\leq i\leq m$, $0\leq j\leq l$. Let $\bar{\mathcal{O}} = \{\lambda, |\lambda|\leq \delta_1- \eta\},$ where $\eta$ is to be determined later.
Using Cauchy inequality, we have $|\partial_\lambda^l P_0(x,y,z,\lambda)| = |\varepsilon^2 \partial_\lambda^l  P(x,y,z,\lambda)| \leq \frac{\varepsilon^2}{\eta^{l_0}}.$

Let $P = \sum\limits_{\substack{\imath\in Z_+^n,\jmath\in Z_+^{2m},\\ k \in Z^n}} p_{k\imath\jmath} y^{\imath} z^{\jmath} e^{\sqrt{-1}\langle k, x\rangle}$ and
$R = \sum\limits_{\substack{\imath+\jmath\leq 2, |k| \leq K_+,\\ \imath\in Z_+^n,\jmath\in Z_+^{2m},  k \in Z^n}} p_{k\imath\jmath} y^{\imath} z^{\jmath} e^{\sqrt{-1}\langle k, x\rangle},$
where $K_1 = ([\log \frac{1}{\varepsilon^{\frac{1}{a}}}]+1)^{3\eta}$). Standardly, $|\partial_\lambda^l  R|_{D_\alpha \times \bar{\mathcal{O}}} \leq c \frac{\varepsilon^2}{\eta^{l_0}}$ and $|\partial_\lambda^l  (P- R)|_{D_\alpha \times \bar{\mathcal{O}}} \leq c \frac{\varepsilon^2\mu}{\eta^{l_0}},$
if
\begin{itemize}
\item[{(H1)}] $K_1 \geq \frac{8(n+ l_0)}{r - r_+}$,
\item[{(H2)}] $\int_{K_1}^\infty \lambda^{n + l_0} e^{-\lambda \frac{r - r_+}{8}} d \lambda \leq \mu.$
\end{itemize}

Under the time $1$-map $\phi_F^1$ of the flow generated by a Hamiltonian\\ $ F = \sum\limits_{\substack{\imath+\jmath\leq 2, 0< |k| \leq K_+,\\ \imath\in Z_+^n,\jmath\in Z_+^{2m},  k \in Z^n}} f_{k\imath\jmath} y^{\imath} z^{\jmath} e^{\sqrt{-1}\langle k, x\rangle},$
 Hamiltonian (\ref{701}) arrives at
\begin{eqnarray}
\label{H1}\bar{H}_1=H\circ \phi_F^1 =N +  R+ \{N, F\} + \bar{P}_1,
\end{eqnarray}
where $\{N, F\}$, $R_t$ and $J$ defined as those in (\ref{Eq425}) and (\ref{Eq326}), and
\begin{eqnarray}
\label{EQ12} \bar{P}_1&=&  \int_0 ^1 \{R_t,F\}\circ \phi_F^t dt + (P - R )\circ \phi_F^1.
\end{eqnarray}
Similarly, homological equation
\begin{eqnarray}\label{EQ13}
\{N, F\} + R - [R]= 0,
\end{eqnarray}
is solvable on
\begin{eqnarray*}
\mathcal{O}_1 &=& \{\lambda\in \mathcal{O}, |L_{k0}|\geq \frac{\min\limits_{0\leq\hat{\iota}_1 \leq m}\{\varepsilon_{\hat{\iota}_1}\}  \gamma}{|k|^\tau},\ \mathcal{A}_1^* \mathcal{A}_1 \geq (\frac{\min\limits_{\substack{0\leq\hat{\iota}_1 \leq m,\\ 0\leq\hat{\iota}_2 \leq l}}\{\varepsilon_{\hat{\iota}_1}, \mu_{\hat{\iota}_2}\} \} \gamma}{|k|^\tau})^2 I_{n + 2m},\\
&~&~~~~~~ \mathcal{A}_2^* \mathcal{A}_2 \geq (\frac{\min\limits_{\substack{0\leq\hat{\iota}_1 \leq m,\\ 0\leq\hat{\iota}_2 \leq l}}\{\varepsilon_{\hat{\iota}_1}, \mu_{\hat{\iota}_2}\}  \gamma}{|k|^\tau})^2 I_{(n+ 2m)^2}, 0<|k|\leq K_1\},
\end{eqnarray*}
where $\mathcal{A}_1$ and $\mathcal{A}_2$ have similar definitions as those in (\ref{Eq317}) and (\ref{Eq318}). Concretely, we have the following lemma.
\begin{lemma} The following hold for all $0< |k| \leq K_1$.
\begin{itemize}
\item[{(i)}]
Homology equations can be solved on $\mathcal{O}_1$ successively to obtain functions $f_{k00},$ $f_{k01},$ $f_{k10},$ $f_{k11},$ $f_{k20},$ $f_{k02},$ $0 <|k| \leq K_1$, which are smooth in $\lambda \in \mathcal{O}_1$ and analytic in $(y,z) \in D(s)$, and $\bar{f}_{kij}(\bar{y}, \bar{z}) = f_{-kij}(y,z),$
for all $0\leq |i|+|j|\leq 2$, $0 < |k|\leq K_1$, $(y,z)\in D(s).$ Moreover, on $D(s)\times \bar{\mathcal{O}}_1$,
$|\partial_\lambda^l f_{kij}|_{D(s)\times \bar{\mathcal{O}}_1} \leq \frac{\varepsilon \Gamma(r - r_+)}{\gamma \eta^{l_0}}, \ 0\leq i,\ j,\ i + j \leq 2.$
\item[{(ii)}] On $\hat{D}(s) \times \bar{\mathcal{O}}_1$, $|\partial_\lambda^l \partial_x^i \partial_{(y,z)}^j F| \leq \frac{\varepsilon\Gamma(r - r_+) }{ \eta^{l_0}},~~ |l| \leq l_0, |i|+ |j| \leq 2.$
\end{itemize}
\end{lemma}
\begin{proof}
The proof of this lemma is similar to the one of {Lemma $\ref{Eq53}$}. And we omit the detail.

\end{proof}
Consider the translation $\phi : x \rightarrow x, \left(
                          \begin{array}{c}
                            y \\
                            z \\
                          \end{array}
                        \right) \rightarrow \left(
                                              \begin{array}{c}
                                                y \\
                                                z \\
                                              \end{array}
                                            \right)+ \left(
                                              \begin{array}{c}
                                                y_0 \\
                                                z_0 \\
                                              \end{array}
                                            \right),$
where
\begin{eqnarray}\label{EQ16}
M \left(
                                              \begin{array}{c}
                                                y_0 \\
                                                z_0 \\
                                              \end{array}
                                            \right) + \left(
                                                        \begin{array}{c}
                                                          \partial_y h \\
                                                          \partial_z h \\
                                                        \end{array}
                                                      \right)
                                             =
- \left(
                                              \begin{array}{c}
                                              \varepsilon^2  p_{010} \\
                                              \varepsilon^2  p_{001} \\
                                              \end{array}
                                            \right).
\end{eqnarray}
Then Hamiltonian $(\ref{H1})$ reaches
\begin{eqnarray*}
H_1 = \bar{H}_1 \circ \phi = e_1 + \langle \omega_1, y\rangle + \frac{1}{2} \langle \left(
                                                           \begin{array}{c}
                                                             y \\
                                                             z \\
                                                           \end{array}
                                                         \right), M_1 \left(
                                                                      \begin{array}{c}
                                                                        y \\
                                                                        z \\
                                                                      \end{array}
                                                                    \right)
\rangle + h_1(y,z)+ P_1,
\end{eqnarray*}
where
\begin{eqnarray*}
e_1 &=& e + \langle \omega, y_0\rangle + \frac{1}{2} \langle \left(
                                                               \begin{array}{c}
                                                                 y_0 \\
                                                                 z_0 \\
                                                               \end{array}
                                                             \right), \big(M + 2\varepsilon^2 \left(
                                                                         \begin{array}{cc}
                                                                           p_{k20} & \frac{1}{2} p_{k11} \\
                                                                           \frac{1}{2} p_{k11}^T & p_{k02}\\
                                                                         \end{array}
                                                                       \right) \big) \left(
                                                               \begin{array}{c}
                                                                 y_0 \\
                                                                 z_0 \\
                                                               \end{array}
                                                             \right)
\rangle\\
 &~&+\varepsilon^2 p_{000}+ \langle \left(
                                          \begin{array}{c}
                                            \varepsilon^2 p_{010} \\
                                            \varepsilon^2 p_{001} \\
                                          \end{array}
                                        \right)
, \left(
                                          \begin{array}{c}
                                             y_0 \\
                                             z_0 \\
                                          \end{array}
                                        \right)\rangle+ h(y_0, z_0),\\
\omega_+ &=& \omega + \langle \left(
                              \begin{array}{c}
                                y \\
                                z \\
                              \end{array}
                            \right), M \left(
                                         \begin{array}{c}
                                           y_0 \\
                                           z_0 \\
                                         \end{array}
                                       \right)
\rangle+ \langle \left(
                              \begin{array}{c}
                                y \\
                                z \\
                              \end{array}
                            \right),  \left(
                                         \begin{array}{c}
                                         \partial_y h   \\
                                           \partial_z h \\
                                         \end{array}
                                       \right)
\rangle + \langle \left(
                              \begin{array}{c}
                                y \\
                                z \\
                              \end{array}
                            \right),  \left(
                                         \begin{array}{c}
                                         p_{010} \\
                                         p_{001} \\
                                         \end{array}
                                       \right)
\rangle,\\
M_+ &=& M +  2\varepsilon\left(
                                                                         \begin{array}{cc}
                                                                           p_{020} & \frac{1}{2} p_{011} \\
                                                                           \frac{1}{2} p_{011}^T & p_{002}\\
                                                                         \end{array}
                                                                       \right)+ \partial_{(y,z)}^2 h,\\
h_+ &=& h(y+ y_0, z+ z_0) - h(y_0, z_0)- \langle \left(
                                                   \begin{array}{c}
                                                     y \\
                                                     z \\
                                                   \end{array}
                                                 \right),  \left(
                                                   \begin{array}{c}
                                                     \partial_y h \\
                                                     \partial_z h \\
                                                   \end{array}
                                                 \right)
\rangle\\
&~& -\frac{1}{2} \langle \left(
                                \begin{array}{c}
                                  y \\
                                  z \\
                                \end{array}
                              \right),  \partial_{(y, z)}^2 h\left(
                                \begin{array}{c}
                                  y \\
                                  z \\
                                \end{array}
                              \right)
\rangle,\\
P_1 &=&  \bar{P}_1 \circ\phi + \frac{\varepsilon^2}{2} \langle \left(
                                                     \begin{array}{c}
                                                       y \\
                                                       z \\
                                                     \end{array}
                                                   \right), 2 \left(
                                                                \begin{array}{cc}
                                                                  p_{k20} & \frac{1}{2} p_{k11} \\
                                                                  \frac{1}{2} p_{k11}^T & p_{k20} \\
                                                                \end{array}
                                                              \right)
\left(
                                                     \begin{array}{c}
                                                       y_0\\
                                                       z_0 \\
                                                     \end{array}
                                                   \right) \rangle.
\end{eqnarray*}

\begin{lemma}
Assume that
\begin{itemize}
\item[${(H3)}$] $M^T M \geq (\min\{\mu_1, \cdots, \mu_l\})^2 I_{n+ 2m}.$
\end{itemize}
 Then there is a constant $c$ such that for all $|l| \leq l_0$, we have
\begin{eqnarray*}
|\partial_\lambda^l e_1 - \partial_\lambda^l e|_{\bar{\mathcal{O}}}, |\partial_\lambda^l M_1 - \partial_\lambda^l M|_{\bar{\mathcal{O}}},\ |\left(
   \begin{array}{c}
     \partial_\lambda^l y_0 \\
     \partial_\lambda^l z_0 \\
   \end{array}
 \right)
|_{\bar{\mathcal{O}}} \leq \frac{\varepsilon}{\eta^{l_0}}.
\end{eqnarray*}
\end{lemma}
\begin{proof}
The proof of this lemma is similar to the one of Lemma \ref{Lm1}. And we omit the detail.

\end{proof}

Standardly, for all $0\leq t\leq1$,
\begin{eqnarray}\label{1012}
\phi_F^t&:& D_{\frac{\alpha}{4}} \rightarrow D_{\frac{\alpha}{2}},\\
\phi&:& D_{\frac{\alpha}{8}} \rightarrow D_{\frac{\alpha}{4}}
\end{eqnarray}
are well defined, real analytic and depend smoothly on $\lambda \in {\mathcal{O}}_+$ under the following assumptions
\begin{itemize}
\item[{(H4)}]  {$c \mu\Gamma( r- r_+) < \frac{1}{8} (r -r_+),$}
\item[{(H5)}]  {$c \mu \Gamma ( r- r_+) < \frac{1}{8}\alpha,$}
\item[{(H6)}]  {$\mu  < \frac{1}{8}\alpha.$}
\end{itemize}
 Moreover, there is a constant $c$ such that for all $0\leq t\leq 1$, $|l|\leq l_0$,
\begin{eqnarray*}
|\partial_\lambda^l  \phi_F^t|_{D_{\frac{\alpha}{4} }\times \mathcal{O}_+} \leq \varepsilon \Gamma(r -r_+).
\end{eqnarray*}

Since $
R_t = (1-t) \{N, F\} + R= t R + (1-t) [R],$
we have
\begin{eqnarray*}
\{R_t, F\} &=& \partial_x R_t \partial_y F - \partial_y R_t \partial_x F + \partial_z R_t J \partial_z F\\
&=& \partial_x (t R) \partial_y F - \partial_y (t R + (1-t) [R]) \partial_x F + \partial_z (t R + (1-t) [R]) J \partial_z F.
\end{eqnarray*}
Then $
|\{R_t, F\}| \leq \varepsilon^3 \Gamma(r - r_+).$
Further, $
|\int_0^1 \{R_t, F\} \circ \phi_F^t dt|_{D_{\frac{\alpha}{4}} \times \bar{\mathcal{O}}_+} \leq \varepsilon^3 \Gamma(r - r_+).$
Therefore, $
|( P- R)\circ \phi_F^1|_{D_{\frac{\alpha}{4}} \times \bar{\mathcal{O}}_+} \leq \varepsilon^2 \mu \Gamma(r -r_+).$
Let $\mu^a = \varepsilon,$ where $a>1$ is a given constant.
According to $
 \bar{P}_+ = \int_0 ^1 \{R_t,F\}\circ \phi_F^t dt + (P - R )\circ \phi_F^1,$
we have $
|\bar{P}_+|\leq c \varepsilon^2 \big(\varepsilon \Gamma(r - r_+)+ \mu \Gamma(r - r_+)\big) \leq \varepsilon^2 \mu^{1-\sigma}$
under the following assumption $\mu^\sigma \Gamma(r - r_+)<1,$ which is obvious.
Therefore, $
|\partial_\lambda^l \bar{P}_+|\leq c \frac{\varepsilon^2 \mu^{1-\sigma}}{\eta^{l_0}}.
$
Since $
|\partial_\lambda^l \phi|_{\bar{\mathcal{O}}_+} \leq \frac{\varepsilon}{\eta^{l_0}},$
we have
\begin{eqnarray*}
|\partial^{l} \bar{P}_+ \circ \phi+ \frac{\varepsilon^2}{2}\partial^{l} \langle \left(
                                                                   \begin{array}{c}
                                                                     y \\
                                                                     z \\
                                                                   \end{array}
                                                                 \right),
\left(
  \begin{array}{cc}
    p_{020} & \frac{1}{2} p_{011} \\
    \frac{1}{2} p_{011}^T & p_{020} \\
  \end{array}
\right)\left(
         \begin{array}{c}
           y_0 \\
           z_0 \\
         \end{array}
       \right)
\rangle|_{D_{\frac{\alpha}{8}}\times \bar{\mathcal{O}}_+} \leq \frac{\varepsilon^2 \mu^{1-\sigma}}{\eta^{l_0}} \leq \frac{\varepsilon \varepsilon_1 }{\eta_+^{l_0}},
\end{eqnarray*}
where $\varepsilon_1 = \varepsilon^{1+ \frac{1-\sigma}{a}}$.
\begin{remark}
For next KAM step, $\mu_1^a = \varepsilon_1$.
\end{remark}

\begin{remark}
As stated in ${\cite{Arnold1}}$, the procedure of successive changes of variable has the remarkable property of quadratic convergence: after $m$ changes of variables the phase-depending discrepancy in the Hamiltonian is of order $\varepsilon^{2m}$ (disregarding small denominators). Due to small denominators, the convergence of KAM iteration is slower than quadratic convergence. In our procedure, after $m$ KAM step, the new perturbation $|P_m| = O( \varepsilon^{1+ (1+ \frac{1-\sigma}{a})m})$, where $0< \sigma <1$, $\mu^a = \varepsilon$.
\end{remark}

After $\tilde{\nu}_* = [\frac{\log 9}{\log 1+ \frac{1-\sigma}{a} }]+1$ KAM steps mentioned above, we get the following system
\begin{eqnarray*}
H_{\tilde{\nu}_*}(x,y,z,\lambda)&=& e_{\tilde{\nu}_*}+ \langle \omega_{\tilde{\nu}_*}, y\rangle+ \frac{1}{2}\langle \left(
                                                                              \begin{array}{c}
                                                                                y \\
                                                                                z \\
                                                                              \end{array}
                                                                            \right),
          M_{\tilde{\nu}_*} \left(
                                                                              \begin{array}{c}
                                                                                y \\
                                                                                z \\
                                                                              \end{array}
                                                                            \right)\rangle
                                                                            + h_{\tilde{\nu}_*}(y,z)\\
&&+ \varepsilon^{10} P(x,y, z,\lambda).~~~~~~~~~~
\end{eqnarray*}
Let $\gamma = \varepsilon^{\frac{\sigma}{3b}}$, $s = \varepsilon^4$, $\mu = \varepsilon^{1- \sigma+\iota}$, $\eta = \varepsilon^{\frac{\iota}{l_0}}$, $\iota < \sigma$. Obviously, $|\partial_\lambda^l (\varepsilon^9 P)| =  |\partial_\lambda^l \bar{P}| \leq \frac{\gamma^{3b} s^2 \mu}{\eta^{l_0}}.$
Then, using KAM steps in Section $\ref{038}$, we finish the proof of Theorem \ref{TH1}.

\begin{remark}\label{remark1}
According to the proof, $\varepsilon^2 P(x,y, z,\lambda)$ could be reduced to $\varepsilon P(x,y, z,\lambda)$ satisfying
 $\varepsilon ||P(x,y,z,\lambda)||_{C^2} = \max\limits_{|i|+ |j| \leq2} \sup\limits_{|Im x|\leq r, |\lambda|\leq \delta_1,y,z}|\frac{\partial^{i+j} P}{\partial y^i \partial z^j}|<c \varepsilon \varepsilon^{\frac{1}{a}},$ for some constant $c$ and $a$.
\end{remark}

\section{Persistence of Resonant tori}\setcounter{equation}{0} \label{Resonant}
Consider the following resonant multi-scale Hamiltonian systems
 \begin{eqnarray}\label{Eq55}
H(x,y)= \varepsilon_0 h_0(y) + \cdots+ \varepsilon_m h_m(y) + \varepsilon^2 P(x,y),
\end{eqnarray}
where $0<\varepsilon\leq \varepsilon_i\ll1$, $0\leq i\leq m$, $(x, y) \in T^d \times G \subset T^d \times R^d$, $G$ is a bounded closed region

 The resonant frequency $\omega$ is $m_0$-resonant if there is  rank $m_0$ subgroup $g$ of $Z^d$ generated by $\tau_1,$ $\cdots,$ $\tau_{m_0}$ such that $\langle k, \omega\rangle = 0$, $\forall$ $k \in g$ and $\langle k, \omega\rangle \neq 0$ for all $k \in Z^d/ g$. According to group theory, there are integer vectors $\tau_1'$, $\cdots,$ $\tau_n'$ $\in$ $Z^d$ such that $Z^d$ is generated by $\tau_1$, $\cdots,$ $\tau_{m_0}$, $\tau_1'$, $\cdots,$ $\tau_n'$ and $\det K_0 = 1$, where $K_0 = (K_*, K')$, $K_* = (\tau_1', \cdots, \tau_n'),$ $K' = (\tau_1, \cdots, \tau_{m_0})$ are $d\times d$, $d\times n$, $d \times m_0$, respectively.

Assume $\omega(y) = \varepsilon_0 \partial_y h_0(y) + \cdots+ \varepsilon_m \partial_y h_m(y)$ is $m_0$-resonant. We set the $g$-resonant manifold $O(g, G) = \{y\in G: \langle k, \omega(y)\rangle= 0, k\in g\},$
which is an $n = d- m_0$ dimensional surface.  For any $y_0 \in O(g, G),$ according to Taylor series and the following symplectic transformation: $y -y_0 = K_0 p$, $q = K_0^T x$, we have
\begin{eqnarray}
\nonumber H(p,q) &=& e + \langle K_0^T \omega, p\rangle + \frac{1}{2} \langle K_0^T \partial_y^2 H K_0p, p\rangle + O(|K_0p|^3)\\
\nonumber &~&+ \varepsilon^2 P((K_0^T)^{-1}q, y_0 + K_0 p)\\
\label{Eq57}  &=&e + \langle {\omega}_*, p'\rangle + \frac{1}{2} \langle K_0^T \partial_y^2 H K_0p, p\rangle + O(|K_0p|^3)+ \varepsilon^2 \tilde{P}(q, p),~~~~~
\end{eqnarray}
where $\omega_* = K_*^T \omega $, $p = (p', p'')$, $q= (q', q'')$, $ \tilde{P}(q, p) = P((K_0^T)^{-1}q, y_0 + K_0 p)$, the dimension of $p'$ and $q'$ is $m$, the dimension of $p''$ and $q''$ is $m_0 = n- m.$ Obviously, in a $\varepsilon$-neighbourhood of $p_0$,
\begin{eqnarray*}
\tilde{P}(q, p) = \tilde{P}(q, p_0) + O(|p - p_0|)=\tilde{P}(q, p_0)+ O(\varepsilon).
\end{eqnarray*}
 Introduce a generating function $S = \langle Y,q\rangle + \varepsilon^2 \sum\limits_{k\in Z^{m}\setminus0} \frac{\sqrt{-1} h_k}{\langle k, \omega_*\rangle} (q'', \omega_*) e^{\sqrt{-1}\langle k, q'\rangle},$
where $h_k = \int_0^{2\pi} \tilde{P} (p_0, q) e^{-\sqrt{-1} \langle k, q'\rangle} dq'$.
Under the symplectic transformation generated by $S$: $(p, q~mod~2\pi)\mapsto (Y, X~mod~2\pi): p = \partial_q S(q,Y),~X =  \partial_Y S(q,Y),$
i.e., $p' = Y' - \varepsilon^2 \sum\limits_{k\in Z^m \setminus 0} k \frac{ h_k (p_0, q'')}{\langle k, \omega_*\rangle}  e^{\sqrt{-1}\langle k, q'\rangle},$ $X' = q',$ $p'' = Y'' +  \varepsilon^2  \sum\limits_{k\in Z^m \setminus 0} \partial_{q''} \big(\frac{\sqrt{-1} h_k(p_0, q'')}{\langle k, \omega_*\rangle}\big) e^{\sqrt{-1}\langle k, q'\rangle},$ $ X''= q'',$ (\ref{Eq57}) is changed to
\begin{eqnarray}
\nonumber H(X, Y) &=& e + \langle {\omega_*}, Y' \rangle + \frac{1}{2} \langle K_0^T \partial_y^2 H K_0 \left(
                                                                   \begin{array}{c}
                                                                     Y'  \\
                                                                     Y'' \\
                                                                   \end{array}
                                                                 \right),
\left(
                                                                   \begin{array}{c}
                                                                     Y'  \\
                                                                     Y'' \\
                                                                   \end{array}
                                                                 \right)\rangle+\tilde{h}_1+ \tilde{h}_2\\
\nonumber &~& + O\big((|Y'|+|Y''|)^3\big) + \varepsilon^2 \big( \tilde{P}(X, p_0)\\
\label{Eq58} &~&- \sum\limits_{0\neq k \in Z^m} h_k e^{\sqrt{-1}\langle k, X'\rangle} + O(\varepsilon)\big),~~~~~
\end{eqnarray}
where
\begin{eqnarray*}
A&=&- \varepsilon^2 \sum\limits_{k\in Z^m \setminus 0} k \frac{ h_k (p_0, q'')}{\langle k, \omega_*\rangle}  e^{\sqrt{-1}\langle k,q'\rangle},\\
B&=& \varepsilon^2  \sum\limits_{k\in Z^m \setminus 0}  \partial_{q''} \big(\frac{\sqrt{-1} h_k(p_0, q'')}{\langle k, \omega_*\rangle}\big)e^{\sqrt{-1}\langle k,q'\rangle},\\
\tilde{h}_1&=& \frac{1}{2} \langle K_0^T \partial_y^2 H K_0 \left(
                                                                   \begin{array}{c}
                                                                     Y'  \\
                                                                     Y'' \\
                                                                   \end{array}
                                                                 \right),
\left(
                                                                   \begin{array}{c}
                                                                     A  \\
                                                                     B \\
                                                                   \end{array}
                                                                 \right)\rangle +
                                                                 \frac{1}{2} \langle K_0^T \partial_y^2 H K_0 \left(
                                                                   \begin{array}{c}
                                                                     A  \\
                                                                     B \\
                                                                   \end{array}
                                                                 \right),
\left(
                                                                   \begin{array}{c}
                                                                     Y'  \\
                                                                     Y'' \\
                                                                   \end{array}
                                                                 \right)\rangle\\
&~&+ \frac{1}{2} \langle K_0^T \partial_y^2 H K_0 \left(
                                                                   \begin{array}{c}
                                                                     A \\
                                                                     B \\
                                                                   \end{array}
                                                                 \right),
\left(
                                                                   \begin{array}{c}
                                                                     A  \\
                                                                     B \\
                                                                   \end{array}
                                                                 \right)\rangle\\
&=& O(\varepsilon^2 Y) + O(\varepsilon^4),\\
\tilde{h}_2 &=&\varepsilon^2 \big( O(|Y'|^2+ |Y''|^2 + |Y' + Y''|)+ \varepsilon^2 (|Y'|+ |Y''|) + \varepsilon^4\big).
\end{eqnarray*}
Assume $(X'', p_0) = (0, p_0)$ is the nondegenerate critical point of $[\tilde{P}](X'', p_0) = \tilde{P}(X, p_0) - \sum\limits_{0\neq k \in Z^m} h_k e^{\sqrt{-1}\langle k, X'\rangle}$ for some $p_0$.
Then $(\ref{Eq58})$ arrives at
\begin{eqnarray}
\nonumber H(X, Y) &=& e + \langle {\omega_*}, Y' \rangle + \frac{1}{2} \langle K_0^T \partial_y^2 H K_0 \left(
                                                                   \begin{array}{c}
                                                                     Y'  \\
                                                                     Y'' \\
                                                                   \end{array}
                                                                 \right),
\left(
                                                                   \begin{array}{c}
                                                                     Y'  \\
                                                                     Y'' \\
                                                                   \end{array}
                                                                 \right)\rangle+\tilde{h}_1+ \tilde{h}_2~~~~~~~\\
\nonumber && + O\big((|Y'|+|Y''|)^3\big) + \varepsilon^2 \big( \frac{1}{2}\langle \partial_{ X''}^2 [\tilde{P}](0, p_0) X'', X''\rangle~~~~~~~\\
\label{Eq59}&& + O(|X''|^3)  + O(\varepsilon^3)\big),~~~~~~~
\end{eqnarray}
where $\tilde{h}_1$ and $\tilde{h}_2$ are defined as above.
With the following transformation:
\begin{eqnarray*}
\left(
     \begin{array}{c}
       Y' \\
       Y''\\
     \end{array}
   \right) \rightarrow\varepsilon^{\frac{1}{2}} \left(
     \begin{array}{c}
       Y' \\
       Y''\\
     \end{array}
   \right), \left(
     \begin{array}{c}
       X' \\
       X''\\
     \end{array}
   \right) \rightarrow \left(
     \begin{array}{c}
       X' \\
       X''\\
     \end{array}
   \right), H \rightarrow \varepsilon^{-\frac{1}{2}} H,
\end{eqnarray*}
 Hamiltonian $(\ref{Eq59})$ is changed to
\begin{eqnarray}
\nonumber H(X, Y) &=& \frac{e}{\varepsilon^2}+ \langle {\omega_*}, Y' \rangle + \frac{\varepsilon^{\frac{1}{2}}}{2} \langle K_0^T \partial_y^2 H K_0 \left(
                                                                   \begin{array}{c}
                                                                     Y'  \\
                                                                     Y'' \\
                                                                   \end{array}
                                                                 \right),
\left(
                                                                   \begin{array}{c}
                                                                     Y'  \\
                                                                     Y'' \\
                                                                   \end{array}
                                                                 \right)\rangle \\~~~~~~~~~~~
\nonumber &~&+ \frac{\varepsilon^{\frac{3}{2}}}{2}\langle \partial_{X''}^2 [\tilde{P}](0, p_0) X'', X''\rangle  + \varepsilon O\big((|Y'|+|Y''|)^3\big)~~~~~~~~~~~\\
\label{Eq61} &~&+ \hat{P}(X', X'',Y',Y''),
\end{eqnarray}
where
\begin{eqnarray*}
\hat{P}&=&\hat{h}_1+\hat{h}_2+ \varepsilon^{\frac{3}{2}} O(|X''|^3) + O(\varepsilon^{\frac{5}{2}}), \ \ \hat{h}_1 =O(\varepsilon^2 |Y|) + O(\varepsilon^{\frac{7}{2}}),\\
\hat{h}_2 &=&\varepsilon^{\frac{5}{2}} \big( O(|Y'|^2+ |Y''|^2 + |Y'| |Y''|)+ \varepsilon (|Y'|+ |Y''|) + \varepsilon^{3}\big).
\end{eqnarray*}

Let $Y' = y,$ $Y'' = v$, $X'' = u$, $X' = x$, $z = (v,u)^T.$ Rewrite $(\ref{Eq61})$ as
\begin{eqnarray}
\nonumber H(x,y,u,v) &=& \frac{e}{\varepsilon^2}+ \langle {\omega_*}, y\rangle + \frac{1}{2} \langle \left(
                                                                                \begin{array}{c}
                                                                                  y \\
                                                                                  z \\
                                                                                \end{array}
                                                                              \right), M \left(
                                                                                \begin{array}{c}
                                                                                  y \\
                                                                                  z \\
                                                                                \end{array}
                                                                              \right)
\rangle+ \varepsilon h(y,v)\\
\label{Eq62}&~& + \varepsilon^{\frac{3}{2}} \hat{P}(x,y,u,v),
\end{eqnarray}
where $h(y,v) = O\big((|y|+ |v|)^3\big)$, $M = \varepsilon^{\frac{1}{2}} \left(
        \begin{array}{cc}
          K_0^T \partial_y^2 H K_0 & 0 \\
          0 & \varepsilon \partial_{X''}^2 [\tilde{P}](0, p_0)  \\
        \end{array}
      \right)$ and
\begin{eqnarray*}
\hat{P} &=& O(\varepsilon^{\frac{1}{2}} |\left(
                                          \begin{array}{c}
                                            y \\
                                            v \\
                                          \end{array}
                                        \right)|
 ) + O(\varepsilon^{\frac{7}{2}}) + \varepsilon \big( O(|y|^2 + |v|^2 + |y||v|)\\
 &~&+ \varepsilon O(|y| + |v|) + \varepsilon^3 \big) + O(|u|^3) + O(\varepsilon).
\end{eqnarray*}
Combining Remark \ref{remark1} and Theorem \ref{TH1}, under conditions ${(S1)}$, ${(S2)}$, ${(S3)}$, ${(S4)}$, ${(S5')}$ and ${(S6')}$ we finish the proof of Theorem \ref{resonant}.

\section{Proof of Corollary \ref{cor.}}\label{POC}
According to {Theorem~\ref{TH1}}, the proof of this corollary is equivalent to verify conditions ${(M1)}$ and ${(M2)}$ under conditions ${(M1')}$ and ${(M2')}$. Denote
\begin{eqnarray*}
&~&\tilde{ \mathcal{A}}_1 =(\mathcal{A}_1(\lambda_0), \partial_\lambda \mathcal{A}_1(\lambda_0), \cdots, \partial_\lambda^\alpha \mathcal{A}_1(\lambda_0)),\\
&=&\left(
    \begin{array}{ccccccc}
      L_{k0} I_n & -M_{21}^TJ & \partial_\lambda (L_{k0} I_n) & \partial_\lambda(-M_{21}^TJ) & \cdots & \partial_\lambda^\alpha (L_{k0} I_n) & \partial_\lambda^\alpha (-M_{21}^TJ)\\
      0 & L_{k1} & 0 & \partial_\lambda L_{k1} & \cdots & 0 & \partial_\lambda^\alpha L_{k1}   \\
    \end{array}
  \right).
\end{eqnarray*}
There is an orthogonal matrix $Q_1$ exchanging some column vector of $\tilde{ \mathcal{A}}_1$ such that  $
\tilde{ \mathcal{A}}_1 Q_1 = \left(
                             \begin{array}{cc}
                               a_{11} & a_{12} \\
                               0 & a_{22} \\
                             \end{array}
                           \right),$
where $a_{11} = \big(L_{k0}I_n, \partial_\lambda (L_{k0} I_n), \cdots, \partial_\lambda^\alpha (L_{k0} I_n)\big),$ $a_{22}\\ = \big(L_{k1}, \partial_\lambda L_{k1}, \cdots ,  \partial_\lambda^\alpha L_{k1}\big),$ $a_{12} = \big(-M_{21}^TJ, \partial_\lambda(-M_{21}^TJ), \cdots, \partial_\lambda^\alpha (-M_{21}^TJ)\big).$
According to conditions ${(M1')}$, there is an orthogonal matrix $Q_2$ exchanging some column vectors of $a_{11}$ such that $
a_{11} Q_2 = (A,B),$ where $A$ is nonsingular. Therefore, there is an orthogonal matrix $Q_3 = \left(
                                                                            \begin{array}{cc}
                                                                              Q_2 & 0 \\
                                                                              0 & I \\
                                                                            \end{array}
                                                                          \right)
$ such that $
\tilde{ \mathcal{A}}_1 Q_1 Q_3 = \left(
                             \begin{array}{cc}
                               A & \tilde{a}_{12} \\
                               0 & \tilde{a}_{22} \\
                             \end{array}
                           \right),$
where  $\tilde{a}_{12} = (B, {a}_{12})$, $\tilde{a}_{22} = (0, {a}_{22})$.  Let $Q_4= \left(
                                                                                      \begin{array}{cc}
                                                                                        I_n & - A^{-1} \tilde{a}_{12}  \\
                                                                                        0 &  I_a\\
                                                                                      \end{array}
                                                                                    \right).
$
Then $
\tilde{ \mathcal{A}}_1 Q_1 Q_3 Q_4 = \left(
                             \begin{array}{cc}
                               A & 0 \\
                               0 & \tilde{a}_{22} \\
                             \end{array}
                           \right).$
According to condition $\bf{(M2')}$, there is an orthogonal matrix $Q_5$ such that $\tilde{a}_{22}Q_5 = (A_1, B_1)$, where $A_1$ is nonsingular. Therefore, there is an orthogonal matrix $Q_6 = \left(
                                                  \begin{array}{cc}
                                                    I & 0 \\
                                                    0 & Q_5 \\
                                                  \end{array}
                                                \right)
$ such that $\tilde{ \mathcal{A}}_1 Q_1 Q_3 Q_4 Q_6 = (\hat{A}, \hat{B}),$ where $\hat{A} = \left(
                                                                                           \begin{array}{cc}
                                                                                             A & 0 \\
                                                                                             0 & A_1 \\
                                                                                           \end{array}
                                                                                         \right),
$ $\hat{B} = \left(
               \begin{array}{c}
                 0 \\
                 B_1 \\
               \end{array}
             \right),
$
which, using the general Laplace expansion theorem, implies condition ${(M1)}$ hold.

Denote
\begin{eqnarray*}
\tilde{ \mathcal{A}}_2 &=&(\mathcal{A}_2(\lambda_0), \partial_\lambda \mathcal{A}_2(\lambda_0), \cdots, \partial_\lambda^\alpha \mathcal{A}_2(\lambda_0))\\
&=& \left(
    \begin{array}{ccc}
     \widetilde{ I_n\otimes L_{k0}I_n} & \widetilde{I_n \otimes (M_{21}^T J)} & 0 \\
      0 & \widetilde{I_n \otimes L_{k1}} & -\widetilde{(M_{21}^T J)\otimes I_{2m}} \\
      0 & 0 & \widetilde{L_{k2}} \\
    \end{array}
  \right),
\end{eqnarray*}
where
\begin{eqnarray*}
\widetilde{ I_n\otimes L_{k0}I_n} &=& \big(I_n\otimes L_{k0}I_n, \partial_\lambda (I_n\otimes L_{k0}I_n), \ldots, \partial_\lambda^\alpha(I_n\otimes L_{k0}I_n) \big),\\
\widetilde{I_n \otimes (M_{21}^T J)}&=& \big(I_n \otimes (M_{21}^T J), \partial_\lambda (I_n \otimes (M_{21}^T J)), \ldots, \partial_\lambda^\alpha(I_n \otimes (M_{21}^T J)) \big),\\
\widetilde{I_n \otimes L_{k1}} &=& \big(I_n \otimes L_{k1}, \partial_\lambda (I_n \otimes L_{k1}), \ldots, \partial_\lambda^\alpha(I_n \otimes L_{k1}) \big),\\
\widetilde{(M_{21}^T J)\otimes I_{2m}} &=& \big((M_{21}^T J)\otimes I_{2m}, \partial_\lambda ((M_{21}^T J)\otimes I_{2m}), \ldots, \partial_\lambda^\alpha((M_{21}^T J)\otimes I_{2m}) \big),\\
\widetilde{L_{k2}}&=& \big(L_{k2}, \partial_\lambda L_{k2}, \ldots, \partial_\lambda^\alpha L_{k2} \big).
\end{eqnarray*}
Due to conditions ${(D)}$, ${(M1')}$ and ${(M2')}$,  there are orthogonal matrices $\tilde{Q}_1$, $\tilde{Q}_2$ and $\tilde{Q}_3$ such that $\widetilde{ I_n\otimes L_{k0}I_n} \tilde{Q}_1 = (\tilde{A}_1, \tilde{B}_1),$ $\widetilde{I_n \otimes L_{k1}} = (\tilde{A}_2, \tilde{B}_2),$ $\widetilde{L_{k2}} = (\tilde{A}_3, \tilde{B}_3),$
where $n^2\times n^2$ matrix $\tilde{A}_1$, $2mn\times 2mn$ matrix $\tilde{A}_2$ and $4m^2\times 4m^2$ matrix $\tilde{A}_3$ are nonsingular.
Denote $\widetilde{I_n \otimes (M_{21}^T J)} \tilde{Q}_2 = (\tilde{A}_4, \tilde{B}_4)$, $\widetilde{(M_{21}^T J)\otimes I_{2m}} \tilde{Q}_3 = (\tilde{A}_5, \tilde{B}_5)$. Then $\tilde{ \mathcal{A}}_2 \left(
                         \begin{array}{ccc}
                           \tilde{Q}_1 & 0 & 0 \\
                           0 & \tilde{Q}_2 & 0 \\
                           0 & 0 & \tilde{Q}_3 \\
                         \end{array}
                       \right) = \left(
      \begin{array}{cccccc}
        \tilde{A}_1 & \tilde{B}_1 & \tilde{A}_4 & \tilde{B}_4 & 0 & 0 \\
        0 & 0 & \tilde{A}_2 & \tilde{B}_2 & \tilde{A}_5 & \tilde{B}_5 \\
        0 & 0 & 0 & 0 & \tilde{A}_3 & \tilde{B}_3 \\
      \end{array}
    \right). $
Furthermore, there is an orthogonal matrix $\tilde{Q}_4$ such that $\tilde{ \mathcal{A}}_2 \left(
                         \begin{array}{ccc}
                           \tilde{Q}_1 & 0 & 0 \\
                           0 & \tilde{Q}_2 & 0 \\
                           0 & 0 & \tilde{Q}_3 \\
                         \end{array}
                       \right) \tilde{Q}_4 = (\breve{A}, \breve{B}),$
where $\breve{A} = \left(
                \begin{array}{ccc}
                 \tilde{A}_1 &\tilde{A}_4 &  0 \\
                  0 & \tilde{A}_2 & \tilde{A}_5 \\
                   0 & 0 & \tilde{A}_3 \\
                \end{array}
              \right),\ \ \ \breve{B} =\left(
      \begin{array}{ccc}
          \tilde{B}_4 & \tilde{B}_1 & 0 \\
         \tilde{B}_2 & 0 & \tilde{B}_5 \\
        0 & 0 & \tilde{B}_3 \\
      \end{array}
    \right).$
Let $\tilde{Q}_5  = \left(
                      \begin{array}{cc}
                        \tilde{Q}_6 & 0 \\
                        0 & I \\
                      \end{array}
                    \right),
$
where $\tilde{Q}_6 = \left(
                \begin{array}{ccc}
                  I_{n^2} & - \tilde{A}_1^{-1} \tilde{A}_4 &  \tilde{A}^{-1} \tilde{A}_4 \tilde{A}_2^{-1} \tilde{A}_5\\
                  0 & I_{2mn} & -\tilde{A}_2^{-1} \tilde{A}_5 \\
                  0 & 0 & I_{4m^2} \\
                \end{array}
              \right).$
Then\\ $\tilde{ \mathcal{A}}_2 \left(
                         \begin{array}{ccc}
                           \tilde{Q}_1 & 0 & 0 \\
                           0 & \tilde{Q}_2 & 0 \\
                           0 & 0 & \tilde{Q}_3 \\
                         \end{array}
                       \right) \tilde{Q}_4 \tilde{Q}_5 =  \left(
      \begin{array}{cccccc}
       \tilde{A}_1 &0 &  0&   \tilde{B}_4 & \tilde{B}_1 & 0 \\
        0 &\tilde{A}_2 &  0&  \tilde{B}_2 & 0 & \tilde{B}_5 \\
       0 & 0 &  \tilde{A}_3 & 0 & 0 & \tilde{B}_3 \\
      \end{array}
    \right),$
which implies that condition ${(M2)}$ holds. We finish the proof of {Corollary \ref{cor.}}.

\section{Some Examples}\label{EXAM}

\begin{example}\label{Ex1}
Consider the following Hamiltonian systems coupling resonant harmonic oscillators with potential function:
\begin{eqnarray}\label{E1}
H(p, q) = \sum\limits_{i = 1} ^n \frac{1}{2}(p_i^2 + \omega_i^2 q_i^2) + U(q),
\end{eqnarray}
where $p= (p_1, \cdots, p_n)$, $q = (q_1, \cdots,q_n)$, $k_1 \omega_1+ \cdots + k_{n_1} \omega_{n_1} = 0$ for some $k=(k_1, \cdots, k_{n_1}) \in Z^{n_1}$, $n_1 < n$, $\omega_i = O(1)$, $1 \leq i \leq n_1$, $\omega_j = \varepsilon \tilde{{\omega}}_j$, $\tilde{{\omega}}_j = O(1)$, $n_1+1 \leq j \leq n$, and $U(q) = \varepsilon^2 \tilde{{U}}(q)$, $\tilde{{U}}(q)= O(1)$. We have:
\begin{theorem}\label{Meln}
Denote $\mathcal{O} = \{\omega: \langle k, \hat{\omega}\rangle \neq 0, \forall k \in Z^{n-n_1}\}$, where $\hat{\omega} = (\omega_{n_1+1}, \cdots, \omega_n)$, $\omega = (\omega_1, \cdots, \omega_n)\in G \subset R^{n}$, $G$ is a bounded closed region. Assume
\begin{itemize}
  \item [$\bf{(C1)}$] $\omega_{i} \neq 0$, $1 \leq i \leq n_1$.
\end{itemize}
Then there exist a $\varepsilon_0>0$ and a family of Cantor set $\mathcal{O}_\varepsilon \subset \mathcal{O},$ $0<\varepsilon<\varepsilon_0$, such that each $(n-n_1)$-torus $T_\omega$, $\omega \in \mathcal{O}_\varepsilon$, persists and gives rise to a perturbed $(n-n_1)$-torus $T_{\varepsilon,\omega}$ with slightly deformed frequencies. Moreover, the relative Lebesgue measure $|\mathcal{O}\setminus \mathcal{O}_\varepsilon| \rightarrow 0$ as $\varepsilon \rightarrow 0.$

\end{theorem}
\begin{proof}

Consider the following transformation:
\begin{eqnarray*}
\phi: \left\{
        \begin{array}{ll}
          p_i \rightarrow \omega_i \sqrt{1+ \frac{1}{\omega_i}} p_i+ \omega_i q_i, q_i \rightarrow p_i + \sqrt{1+ \frac{1}{\omega_i}} q_i , & \hbox{$1\leq i \leq n_1$;} \\
          p_i \rightarrow \sqrt{2I_i \omega_i } \cos \varphi_i, q_i \rightarrow \sqrt{\frac{2I_i}{ \omega_i} } \sin \varphi_i, & \hbox{$n_1+ 1\leq i \leq n$,}
        \end{array}
      \right.
\end{eqnarray*}
which satisfies
\begin{eqnarray*}
\sum\limits_{i = 1}^{n_1} d p_i \wedge d q_i + \sum\limits_{i = n_1+ 1}^{n} d p_i \wedge d q_i  = \sum\limits_{i = 1}^{n_1} d p_i \wedge d q_i+ \sum\limits_{i = n_1+1}^{n} d I_i \wedge d \varphi_i.
\end{eqnarray*}
Then Hamiltonian $(\ref{E1})$ is changed to
\begin{eqnarray}\label{E2}
\nonumber H &=& \sum\limits_{i = n_1+1} ^n \omega_i I_i + \sum\limits_{i = 1} ^{n_1} \frac{\omega_i^2}{2} \big((\sqrt{1+ \frac{1}{\omega_i}} p_i q_i)^2 + (p_i + \sqrt{1+ \frac{1}{\omega_i}} q_i)^2\big)  + \bar{U}\\
&=& \varepsilon \langle \tilde{\omega}, I \rangle + \frac{1}{2} \langle z, M z\rangle + \varepsilon^2 \hat{{U}},
\end{eqnarray}
where $\bar{U} = \varepsilon^2 \hat{U}( \hat{q}_1, \cdots, \hat{q}_{n_1},   \sqrt{\frac{2I_{n_1+1}}{ \omega_{n_1+1}} } \sin \varphi_{n_1+1}, \cdots, \sqrt{\frac{2I_{n}}{ \omega_{n}} } \sin \varphi_{n})$, $\hat{q}_i = p_i + \sqrt{1+ \frac{1}{\omega_i}} q_i$, $1\leq i\leq n_1$, $\tilde{\omega} = (\tilde{\omega}_{n_1 +1}, \cdots ,\tilde{\omega}_n)$, $I= (I_{n_1+1}, \ldots, I_{n})$,  $z= (u, v)$, $u=(\sqrt{1+ \frac{1}{\omega_1}}p_1 + q_1, \cdots, \sqrt{1+ \frac{1}{\omega_{n_1}}}p_{n_1} + q_{n_1})$, $v= (p_1 + \sqrt{1+ \frac{1}{\omega_1}}q_1, \cdots, p_{n_1} + \sqrt{1+ \frac{1}{\omega_{n_1}}}q_{n_1})$, $M = diag (\omega_1^2, \cdots, \omega_{n_1}^2,\omega_1^2, \cdots, \omega_{n_1}^2 )$.

In this system, the parameters are frequencies. The R\"{u}ssmann-type nondegenerate condition ${(D)}$ hold obviously.  Due to condition ${(C1)}$, $M$ is nonsingular. Directly,
 \begin{eqnarray*}
 MJ  &= &\left(
         \begin{array}{cc}
           diag(\omega_1^2, \cdots, \omega_{n_1}^2) & 0 \\
           0 & diag(\omega_1^2, \cdots, \omega_{n_1}^2) \\
         \end{array}
       \right) \left(
                 \begin{array}{cc}
                   0 & I_{n_1} \\
                  - I_{n_1} & 0 \\
                 \end{array}
               \right)\\
       &=& \left(
             \begin{array}{cc}
               0 & diag(\omega_1^2, \cdots, \omega_{n_1}^2) \\
               -diag(\omega_1^2, \cdots, \omega_{n_1}^2) & 0 \\
             \end{array}
           \right).
 \end{eqnarray*}
 Since
 \begin{eqnarray*}
 \det (\lambda I_{2n_1} - MJ) &=&\det \left(
             \begin{array}{cc}
               \lambda I_{n_1} &- diag(\omega_1^2, \cdots, \omega_{n_1}^2) \\
               diag(\omega_1^2, \cdots, \omega_{n_1}^2) &  \lambda I_{n_1} \\
             \end{array}
           \right)\\
 &=& \det \big( \lambda^2 I_{n_1} + diag(\omega_1^4, \cdots, \omega_{n_1}^4)\big),
 \end{eqnarray*}
 the eigenvalues of $MJ$ are $\pm \sqrt{-1} \omega_1^2$, $\cdots$, $\pm \sqrt{-1} \omega_{n_1}^2$. Here, we use the following fact $\det \left(
        \begin{array}{cc}
          A_{11} & A_{12} \\
          A_{21} & A_{22} \\
        \end{array}
      \right) = \det \big( A_{11} A_{22} - A_{21} A_{12}\big)$ when $A_{11}$ is nonsingular and commutes with $A_{21}$. Therefore, there is an unitary matrix $Q$ such that
 \begin{eqnarray*}
 Q^{*} L_{k1} Q &=& Q^{*}(\sqrt{-1} \langle k, \omega\rangle I_{2n_1} - MJ )Q\\
 &=& \sqrt{-1} diag \big( \langle k, \omega\rangle - \omega_1^2,  \cdots, \langle k, \omega\rangle - \omega_{n_1}^2,  \langle k, \omega\rangle + \omega_1^2,  \cdots, \langle k, \omega\rangle + \omega_{n_1}^2\big),
 \end{eqnarray*}
 which implies the ${(M1'')}$ hold. According $(A \otimes B) (C\otimes D) = (AC)\otimes(BD)$, there is an unitary matrix $Q$ such that
 \begin{eqnarray*}
 &~&(Q^*\otimes Q^*) L_{k2} (Q \otimes Q)\\
  &=&(Q^*\otimes Q^*) \big( \sqrt{-1} \langle k, \omega\rangle I_{4n_1^2} - (MJ)\otimes I_{2n_1}- I_{2n_1} \otimes (MJ)\big) (Q \otimes Q)\\
  &=& diag\big( a_1, \cdots, a_{2n_1}\big),
 \end{eqnarray*}
 where $a_i$, $1\leq i \leq2n_1$, have the following three forms:  (1) $\sqrt{-1} (\langle k, \omega\rangle - 2 \omega_j^2)$, $1\leq j\leq n_1$,
   ${(2)}$ $\sqrt{-1} \langle k, \omega\rangle$,
   ${(3)}$ $\sqrt{-1} (\langle k, \omega\rangle  \pm \omega_i^2 \pm\omega_j^2)$, $1\leq i\neq j\leq n_1$,
which imply condition ${(M2'')}$ hold. Combining {Corollary \ref{cor.}} and {Remark \ref{equi}}, we finish the proof of {Theorem \ref{Meln}}.

\end{proof}

\end{example}
\begin{remark}
This is an example showing the persistence of resonant tori, a special lower dimensional invariant tori, without Melnikov's condition. The differential equations corresponding to Hamiltonian $(\ref{E1})$ are $\ddot{q}_j + \omega_j^2 q_j = - \partial_{q_j} U(q),$
which have many applications ({\cite{Gauckler}}).
\end{remark}

\begin{example}\label{Exa1}
Consider a particular properly degenerate Hamiltonian system with the following form:
\begin{eqnarray}
H(x,y) = H_0(y_0) + \varepsilon H_1 (y_1) + \varepsilon^2 P(x,y),
\end{eqnarray}
where $x\in T^n$, $y = (y_0, y_1) \in G \subset R^n $, $G$ is bounded closed region, the dimensions of $y_0$ and $y_1$ are $n_0$ and $n_1 = n- n_0$, respectively.

Assume $H_0(y_0)$ is $n_0$-resonant, i.e., there are $n_0$ linearly independent $\hat{k}_i$ such that $\langle \hat{{k}}_i, \partial_{y_0} H_0\rangle = 0$, and $H_1 (y_1)$ is nonresonant, i.e., $\langle \tilde{k}, \partial_{y_1} H_1\rangle \neq 0$, $\forall ~ \tilde{k} \in Z^{n_1}.$ Let $K_* = (\hat{k}_1, \cdots, \hat{k}_{n_0})$. Then there is a $K_0 = \left(
               \begin{array}{cc}
                 K_* & 0 \\
                 0 & K' \\
               \end{array}
             \right) $ such that $\det K_0 = 1$, where $K' = (\tilde{k}_1,\cdots, \tilde{k}_{n_1})$, $\tilde{k}_{i} \in Z^{n_1}$ satisfies $\langle \tilde{k}, \partial_{y_1} H_1\rangle \neq 0$. We set the $g-$resonant manifold $O(g, G) = \{y = (y_0, y_1)\in G: \langle k, \partial_{y_0} H_0 \rangle= 0, k\in g\},$ where group $g$ generated by $\tilde{k}_1$, $\cdots$, $\tilde{k}_{n_0}.$

Let $L_{k1}=\sqrt{-1} \langle k, \omega\rangle I_{2n_0} - M J$ and $
L_{k2}=\sqrt{-1} \langle k, \omega\rangle I_{4n_0^2} - (M J) \otimes I_{2n_0} - I_{2n_0}\otimes(M J),$
where $M = \left(
         \begin{array}{cc}
           \varepsilon^{\frac{2}{3}} (K_*)^T \partial_{y_0}^2 H_0 K_* & 0 \\
           0 &  \varepsilon^{\frac{4}{3}}\partial_{X'}^2 h_0 \\
         \end{array}
       \right).$
For resonant properly degenerate Hamiltonian systems $(\ref{Exa1})$, we have the following result.
\begin{theorem}\label{prop}
For Hamiltonian systems $(\ref{Exa1})$ with $n_0$-resonant $H_0$ and nonresonant $H_1$, assume the following conditions:
\begin{itemize}
\item [${(A1)}$] $K_*^T \partial_{y_0}^2 H_0 K_*$ and $(K')^T\partial_{y_1}^2 H_1 K'$ are nonsingular;
\item [${(A2)}$] there is a positive constant $\tilde \sigma$ (independent of $\varepsilon$) such that\\
$|\det \partial_{x_1}^2 \int_{T^{n_1}}P (x,0) d x_2| > \tilde \sigma,$
where $x_1 = (K_*^T, 0) x$, $x_2 = \big((K')^T, 0\big) x$;
 \item [${(A3)}$] there is a finite positive integer $N$ such that for $\lambda\in \mathcal{O}$, $rank \{c_i: 1\leq i\leq 2n_0\} = 2n_0,$
where $c_i$ is a column of $\{\partial_\lambda^\alpha {L}_{k1}^i: 0\leq |\alpha|\leq N \}$, ${L}_{k1}^i$ is the $i$-th column of ${L}_{k1}$;
  \item [${(A4)}$] there is a finite positive integer $N$ such that for $\lambda\in \mathcal{O}$, $rank \{\bar{c}_i: 1\leq i\leq 4n_0^2\} = 4n_0^2,$
where $\bar{c}_i$ is a column of $\{\partial_\lambda^\alpha {L}_{k2}^i: 0\leq |\alpha|\leq N \}$, ${L}_{k2}^i$ is the $i$-th column of ${L}_{k2}$.
 \end{itemize}

Then there exist a $\varepsilon_0 >0 $ and a family of Cantor sets $O_\varepsilon(g, G) \subset O (g, G)$, $0<\varepsilon < \varepsilon_0$, such that for each $y \in O_\varepsilon(g, G)$, system $(\ref{Eq63})$ admits $2^{m_0}$ families of invariant tori, possessing  hyperbolic, elliptic or mixed types, associated to nondegenerate relative equilibria. Moreover, the relative Lebesgue measure $|O(g,G) \setminus O_{\varepsilon}(g,G)|$ tends to 0 as $\varepsilon \rightarrow 0$.
\end{theorem}

\begin{proof}
With the Taylor series, for any $y_0=(\hat{y}_0, \hat{y}_1) \in O(g, G),$ we have
\begin{eqnarray}
\nonumber H(x,y) &=& e+ \langle \omega_0, \tilde{y}_0\rangle + \frac{1}{2} \langle\tilde{y}_0, \partial_{y_0}^2 H_0 \tilde{y}_0 \rangle + O(| \tilde{y}_0|^3) + \langle \omega_1, \tilde{y}_1\rangle \\
\label{Eq67}&~&+ \frac{\varepsilon}{2} \langle \tilde{y}_1, \partial_{y_1}^2 H_1 \tilde{y}_1 \rangle + O(|\tilde{y}_1|^3) + \varepsilon^{2} P(x, y),
\end{eqnarray}
where $\omega_0 = \partial_{y_0} H_0$, $\omega_1 = \partial_{y_1} H_1$, $\tilde{y}_0 = y_0 - \hat{y}_0$, $\tilde{y}_1 =  y_1 - \hat{y}_1.$
Under the following symplectic transformation:
$\tilde{y}_0 =  K_* p_1, \tilde{y}_1 =  K' p_2, q= (q_1, q_2) = K_0^T x,$
Hamiltonian $(\ref{Eq67})$ is changed to
\begin{eqnarray}
\nonumber H(p,q) &=& e+ \varepsilon \langle \omega_*, p_2\rangle + \frac{1}{2} \langle p_1, K_* \partial_{y_0}^2 H_0 K_* p_1 \rangle + \frac{\varepsilon}{2} \langle p_2, (K')^T \partial_{y_1}^2 H_1 K' p_2\rangle\\
\label{Eq68}&~&+ O(| K_* p_1|^3) + \varepsilon O(|K' p_2|^3) + \varepsilon^{2} \bar{P}(q,p),
\end{eqnarray}
where $\omega_* =  (K')^T \omega_1$, $\bar{P}(q,p) = P((K_0^T)^{-1} q, {y_0} + K_0 p)$, $p = (p_1, p_2)$. Obviously, in a $\varepsilon$-neighbourhood of $p_0$, $\bar{{P}}(q, p) =\tilde{P}(q, p_0) + O(|p - p_0|)=\tilde{P}(q, p_0)+ O(\varepsilon).$
Let $S (Y, q) = \langle Y, q \rangle + \varepsilon^2 \sum\limits_{k \in Z^{n_1}\setminus 0} \frac{\sqrt{-1} h_k}{\langle k, \omega_*\rangle}(q_1, \omega_*) e^{\sqrt{-1} \langle k, q_2\rangle},$ where $h_k = \int_{T^{n_1}} \bar{P}(q,p_0) e^{-\sqrt{-1}\langle k, q_2\rangle } dq_2$. Under the following symplectic transformation given by generating function $S$:\\ $(p, q~mod~2\pi)\mapsto (Y, X~mod~2\pi): p = \partial_q S(q,Y),~X =  \partial_Y S(q,Y),$
i.e., $p_2 = Y_2 - \varepsilon^2 \sum\limits_{k\in Z^{n_1} \setminus 0} k \frac{ h_k (p_0, q_1)}{\langle k, \omega_*\rangle}  e^{\sqrt{-1}\langle k, q_2\rangle},$ $X' = q_1,$ $p_1 = Y_1 +  \varepsilon^2  \sum\limits_{k\in Z^{n_1} \setminus 0} \partial_{q_1} \big(\frac{\sqrt{-1} h_k(p_0, q_1)}{\langle k, \omega_*\rangle}\big) e^{\sqrt{-1}\langle k, q_2\rangle},$ $X''= q_2,$ $(\ref{Eq68})$ is changed to
\begin{eqnarray}
\nonumber H(X,Y) & =&e+ \varepsilon \langle \omega_* , Y_2\rangle + \frac{1}{2} \langle Y_1, K_* \partial_{y_0}^2 H_0 K_* Y_1\rangle + \frac{\varepsilon}{2} \langle Y_2, (K')^T \partial_{y_1}^2 H_1 K' Y_2\rangle~~~~~\\
\label{Eq69} &~& + O(|K_* Y_1|^3) +\varepsilon O(|K'Y_2|^3)+ \check{h}_1+ \check{h}_2 + \check{h}_3,~~~~~
\end{eqnarray}
where  $\check{h}_1= \varepsilon^2 O(|Y_1|+ \varepsilon^2 + \varepsilon |Y_2|),$  $\check{h}_2=\varepsilon^2 O( \varepsilon^4 + |Y_1|^2 + \varepsilon^2 |Y_1| + \varepsilon |Y_2|^2 + \varepsilon^3 |Y_2|),$ $\check{h}_3=\frac{\varepsilon^2}{2} \langle X', \partial_{X'}^2 h_0 X'\rangle + \varepsilon^2 O(|X'|^3) + O(\varepsilon^3),$ $h_0 = \int_{T^{n_1}} \bar{P}(X,p_0)  dX''.$
Consider the following transformation: $Y_1 \rightarrow \varepsilon^{\frac{2}{3}}  Y_1, Y_2 \rightarrow \varepsilon^{\frac{2}{3}}  Y_2, X' \rightarrow X', X'' \rightarrow X'', H \rightarrow \varepsilon^{-\frac{2}{3}} H.$ Then Hamiltonian $(\ref{Eq69})$ is changed to
\begin{eqnarray}
\nonumber H(X,Y) &=&\frac{e}{\varepsilon^{\frac{2}{3}}} + \varepsilon \langle \omega_*, Y_2\rangle + \frac{\varepsilon^{\frac{2}{3}}}{2} \langle Y_1, K_* \partial_{y_0}^2 H_0 K_* Y_1\rangle  \\
\label{Eq70} &~&+ \frac{\varepsilon^{\frac{4}{3}}}{2} \langle X', \partial_{X'}^2 h_0 X'\rangle+ \varepsilon^{\frac{4}{3}} P(Y_1, Y_2,X),
\end{eqnarray}
where $P(Y_1, Y_2,q) = \frac{\varepsilon^{\frac{1}{3}}}{2} \langle Y_2, (K')^T \partial_{y_1}^2 H_1 K' Y_2\rangle + \varepsilon^{\frac{2}{3}} \hat{h}_1 + O(|K_* Y_1|^3)
 + O(|K' Y_2|^3) + \varepsilon^{\frac{2}{3}}\hat{h}_2 + \varepsilon^{\frac{4}{3}} O(|X'|^3) +   O(\varepsilon^{\frac{7}{3}}),$ $\hat{h}_1 = \varepsilon^{\frac{2}{3}} (|Y_1| + \varepsilon^{\frac{4}{3}} + \varepsilon |Y_2|),$
$\hat{h}_2 = \varepsilon^{\frac{2}{3}} ( \varepsilon^{\frac{10}{3}} + \varepsilon^{\frac{2}{3}} |Y_1|^2 + \varepsilon^2 |Y_1| + \varepsilon^{\frac{5}{3}} |Y_2|^2 + \varepsilon^3 |Y_2| ).$

Let $Y_2 = y$, $Y_1 = v$, $q_1 = u$, $q_2 = x$, $z = (u,v).$ Then $(\ref{Eq70})$ becomes
\begin{eqnarray}\label{End}
H(x,y,z) = \varepsilon \langle \omega_*, y\rangle + \frac{1}{2} \langle z, M z\rangle  + \varepsilon^{\frac{4}{3}} P(x,y,z),
\end{eqnarray}
where $M = \left(
         \begin{array}{cc}
           \varepsilon^{\frac{2}{3}} (K_*)^T \partial_{y_0}^2 H_0 K_* & 0 \\
           0 &  \varepsilon^{\frac{4}{3}}\partial_{X'}^2 h_0 \\
         \end{array}
       \right)
 ,$
$P= \frac{\varepsilon^{\frac{1}{3}}}{2} \langle y, (K')^T \partial_{y_1}^2 H_1 K' y\rangle + \varepsilon^{\frac{2}{3}} \hat{h}_1 + O(|K_* v|^3)  + O(|K' y|^3) + \varepsilon^{\frac{2}{3}}\hat{h}_2 +\varepsilon^{\frac{4}{3}} O(|u|^3) +   O(\varepsilon^{\frac{7}{3}}),$
$ \hat{h}_1 = \varepsilon^{\frac{2}{3}} (|v| + \varepsilon^{\frac{4}{3}} + \varepsilon |y|),\ \ \ \hat{h}_2 = \varepsilon^{\frac{2}{3}} ( \varepsilon^{\frac{10}{3}} + \varepsilon^{\frac{2}{3}} |v|^2 + \varepsilon^2 |v| + \varepsilon^{\frac{5}{3}} |y|^2 + \varepsilon^3 |y| ).$
Considering Hamiltonian (\ref{End}) on $D(r,s) = \{(x,y,z): |Im x|< r, |y|<s^2, |z|<s\}$ and combining the proof of the {Corollary \ref{cor.}}, {Theorem \ref{prop}} holds obviously under conditions ${(A1)}$, ${(A2)}$, ${(A3)}$ and ${(A4)}$.
\end{proof}

\end{example}

\begin{example}
Consider the following Hamiltonian systems:
\begin{eqnarray*}
H(x,y,z) = \varepsilon^2 \langle \omega, y\rangle + \frac{1}{2} \langle \left(
                                                                          \begin{array}{c}
                                                                            y \\
                                                                            u \\
                                                                            v \\
                                                                          \end{array}
                                                                        \right), \left(
                                                                                   \begin{array}{ccc}
                                                                                     \varepsilon \omega^2 I_2 & 0 & 0 \\
                                                                                     0 & \varepsilon^3 a & \omega^2 \\
                                                                                     0 & \omega^2 & \varepsilon^3 b \\
                                                                                   \end{array}
                                                                                 \right)\left(
                                                                          \begin{array}{c}
                                                                            y \\
                                                                            u \\
                                                                            v \\
                                                                          \end{array}
                                                                        \right)
\rangle + \varepsilon^4 P(x,y,z),
\end{eqnarray*}
where $(x,y,z)\in D(r,s) = \{(x,y,z): |Im x|< r, |y|< s= \varepsilon^{16}, |z|<s= \varepsilon^{16}\}$, $z= (u,v)$, $\omega= (\omega_1, \omega_2)\in \mathcal{O}$, a bounded closed region, $\omega^2  = \omega_1^2 + \omega_2^2$. Condition ${(D)}$ holds obviously. Let $M= \left(
                                                                                   \begin{array}{ccc}
                                                                                     \varepsilon \omega^2I_2 & 0 & 0 \\
                                                                                     0 & \varepsilon^3 a & \omega^2 \\
                                                                                     0 & \omega^2 & \varepsilon^3 b \\
                                                                                   \end{array}
                                                                                 \right).$ It is easy to check that
\begin{eqnarray*}
M^T M &=&\left(
        \begin{array}{ccc}
          \varepsilon^2 \omega^2 I_2 & 0 & 0 \\
          0 & \varepsilon^2 a^2 +\omega^4 & \varepsilon^3 a \omega^2 +\varepsilon^3 b\omega^2 \\
          0 & \varepsilon^3 \omega^2 a+ \varepsilon^3 \omega^2 b & \omega^4 + \varepsilon^6 b^2 \\
        \end{array}
      \right)\geq\varepsilon^6 I_{4}.
\end{eqnarray*}
Denote $\tilde{L}_{k1}=\sqrt{-1} \varepsilon^2\langle k, \omega\rangle I_{2} - {M_{22}} J,$ $
\tilde{L}_{k2}=\sqrt{-1} \varepsilon^2\langle k, \omega\rangle I_{4} - ({M_{22}} J) \otimes I_{2} - I_{2}\otimes({M_{22}} J),$
where $M_{22} = \left(
                  \begin{array}{cc}
                    \varepsilon^3 a & \omega^2 \\
                    \omega^2 & \varepsilon^3 b \\
                  \end{array}
                \right).
$
Then
\begin{eqnarray*}
\tilde{L}_{k1} &=& \left(
                   \begin{array}{cc}
                     \sqrt{-1}\varepsilon^2 \langle k, \omega\rangle+\omega^2 & -\varepsilon^3 a \\
                     \varepsilon^3 b & \sqrt{-1} \varepsilon^2 \langle k, \omega\rangle- \omega^2\\
                   \end{array}
                 \right),\\
\tilde{L}_{k2} &=& \left(
                     \begin{array}{cccc}
                       \sqrt{-1}\varepsilon ^2 \langle k, \omega\rangle+ 2 \omega^2 & -\varepsilon^3 a & -\varepsilon^3 a &0 \\
                       \varepsilon^3 b & \sqrt{-1}\varepsilon ^2 \langle k, \omega\rangle & 0 & -\varepsilon^3 a \\
                       \varepsilon^3 b & 0 & \sqrt{-1}\varepsilon ^2 \langle k, \omega\rangle & -\varepsilon^3 a \\
                       0 & \varepsilon^3 b & \varepsilon^3 b & \sqrt{-1}\varepsilon ^2 \langle k, \omega\rangle- 2 \omega^2 \\
                     \end{array}
                   \right).
\end{eqnarray*}
Therefore, $\partial_{\omega_1}\tilde{L}_{k1} = \left(
                   \begin{array}{cc}
                     2 & 0 \\
                     0 & -2\\
                   \end{array}
                 \right),$\\ $(\partial_{\omega}\tilde{L}_{k2}, \partial_{\omega}^2\tilde{L}_{k2})= \left(
                                                                          \begin{array}{ccccccc}
                                                                            0 & 0 & 0 & 0 & 4 & 0 & *  \\
                                                                            \sqrt{-1}\varepsilon^2 k_1 & \sqrt{-1}\varepsilon^2 k_2 & 0 & 0 & 0 & 0 & *  \\
                                                                            0 & 0 & \sqrt{-1}\varepsilon^2 k_1 & \sqrt{-1}\varepsilon^2 k_2 & 0 & 0 & *  \\
                                                                            0 & 0 & 0 & 0 & 0 & -4 & *  \\
                                                                          \end{array}
                                                                        \right),$\\
which implies conditions ${(M1')}$ and ${(M2')}$ hold. Therefore, using {Corollary \ref{cor.}}, there exist a $\varepsilon_0>0$ and a family of Cantor set $\mathcal{O}_\varepsilon \subset \mathcal{O},$ $0<\varepsilon<\varepsilon_0$, such that each $2$-torus $T_\omega$, $\omega \in \mathcal{O}_\varepsilon$, persists and gives rise to a perturbed $2$-torus $T_{\varepsilon,\omega}$ with the same frequency. Moreover, the relative Lebesgue measure $|\mathcal{O}\setminus \mathcal{O}_\varepsilon| \rightarrow 0$ as $\varepsilon \rightarrow 0.$
\end{example}

\begin{appendices}

      \section{Multi-scale matrix}\label{Inverse}
     \begin{lemma}
Let $D$ be a multi-scale matrix linearly dependent on $\varepsilon_i$, $0\leq i\leq m$, $\mu_j$, $0\leq j\leq l$, where $0< \varepsilon <\varepsilon_i, \mu_j \ll1$. Then the denominator of each elements of $\varepsilon^2 D^{-1}$ is independent of scales.
\end{lemma}
\begin{proof}
Directly,
\begin{eqnarray*}
D^{-1} &=& \frac{adj D}{\det D}=\frac{\big((-1)^{i+j} \tilde{D}_{ji}\big)_{1\leq i,j\leq n}}{\sum\limits_{k=1}^n (-1)^{k+j} d_{kj} \bar{D}_{kj}}\\
&=&\left(
      \begin{array}{cccc}
        \frac{\bar{D}_{11}}{\sum\limits_{k=1}^n(-1)^{k+1} a_{k1}\bar{D}_{k1}} &  -\frac{\bar{D}_{21}}{\sum\limits_{k=1}^n(-1)^{k+1} a_{k1}\bar{D}_{k1}} & \cdots &  \frac{(-1)^{n+1}\bar{D}_{n1}}{\sum\limits_{k=1}^n(-1)^{k+1} a_{k1}\bar{D}_{k1}} \\
        -\frac{\bar{D}_{12}}{\sum\limits_{k=1}^n(-1)^{k+2} a_{k2}\bar{D}_{k2}} & \frac{\bar{D}_{22}}{\sum\limits_{k=1}^n(-1)^{k+2} a_{k2}\bar{D}_{k2}}  & \cdots & \frac{(-1)^{n+2}\bar{D}_{n2}}{\sum\limits_{k=1}^n(-1)^{k+2} a_{k2}\bar{D}_{k2}}  \\
        \vdots & \vdots & \vdots & \vdots \\
        \frac{(-1)^{1+n}\bar{D}_{1n}}{\sum\limits_{k=1}^n(-1)^{k+n} a_{kn}\bar{D}_{kn}}  &  \frac{(-1)^{n}\bar{D}_{2n}}{\sum\limits_{k=1}^n(-1)^{k+n} a_{kn}\bar{D}_{kn}} & \cdots &  \frac{\bar{D}_{nn}}{\sum\limits_{k=1}^n(-1)^{k+n} a_{kn}\bar{D}_{kn}} \\
      \end{array}
    \right),
\end{eqnarray*}
where $\bar{D}_{ij}$ is the determinant of $(n+2m-1)\times(n+2m-1)$ matrix that results from deleting row $i$ and column $j$ of $D$.
The denominator of $\varepsilon^2 \frac{\bar{D}_{ji}}{\sum\limits_{k=1}^n(-1)^{k+1} a_{ki}\bar{D}_{ki}}$ is independent of $\varepsilon_{\iota_1}$, $1\leq\iota_1\leq m,$ $\mu_{\iota_2}$, $1\leq \iota_2 \leq l$ and $\varepsilon$. If $a_{ji}\neq 0$, it is obvious. If $a_{ji}= 0$, we claim there is a $\bar{D}_{ki} \neq 0$ such that the denominator of $\varepsilon \frac{\bar{D}_{ji}}{\bar{D}_{ki}}$ is independent of scales. Since $\sum\limits_{k=1}^n(-1)^{k+1} a_{ki}\bar{D}_{ki} \neq 0$, the existence of $\bar{D}_{ki}\neq 0$ is obvious. If $\bar{D}_{ji} = 0$, the result holds obviously. Since $\bar{D}_{ji}$ has only one different row with $\bar{D}_{ki}$, using the following equation $\det A = \sum\limits_{\sigma} ( sgn \sigma \prod\limits_{i=1}^n a_{i \sigma(i)}),$ in which the sum is over all $n!$ permutations of ${1, \cdots , n}$ and $sgn \sigma$, the "sign" of a permutation $\sigma$, is $+1$ or $-1$ according to whether the minimum
number of transpositions (pairwise interchanges) necessary to achieve it starting from ${1, \cdots , n}$ is even or odd, the claim holds obviously.

\end{proof}

The following Lemma comes from paper $\cite{Horn}$.
\begin{lemma}\label{LMU}
Let $A$, $E$ $\in$ $M_n$ and suppose that $A$ is normal. If $\hat{\lambda}$ is an eigenvalue of $A+E$, then there is an eigenvalue $\lambda$ of $A$ such that $|\hat{\lambda} - \lambda| \leq |||E|||_2$, where $|||E|||_2 = (tr E^* E)^{\frac{1}{2}}$.
\end{lemma}

\begin{lemma}\label{Eigenvalue}
Let $A= \varepsilon_0 A_0 + \cdots + \varepsilon_m A_m$ $\in$ $M_{n,n},$ where $0< \varepsilon_k \ll1$, $A_k = (a_{ij}^k)_{n\times n}$, $0\leq k\leq m$ . Denote the eigenvalue of $AA^*$ by $\lambda_1\leq \cdots\leq \lambda_n$. Then $\lambda_1 > c \min\limits_{1\leq j \leq n} \{\varepsilon_j\}$, where $c$ is positive and depends on $a_{ij}^k$, $1\leq i,j \leq n$, $0\leq k\leq m$.
\end{lemma}

\begin{proof}
Denote $A^*$ the conjugate transpose of $A$.
Directly, we have
\begin{eqnarray*}
AA^* &=& (\varepsilon_0 A_0 + \cdots + \varepsilon_m A_m) (\varepsilon_0 A_0^* + \cdots + \varepsilon_m A_m^*) \\
&=& \sum\limits_{j_1=0}^m \varepsilon_{j_1} A_{j_1}\sum\limits_{j_2 =0}^m \varepsilon_{j_2} A_{j_2}^*\\
&=&\sum\limits_{j_1=0}^m \sum\limits_{j_2=0}^m \varepsilon_{j_1}\varepsilon_{j_2} A_{j_1} A_{j_2}^*\\
&=& (\sum\limits_{k=1}^n \sum\limits_{j_1=0}^m \sum\limits_{j_2=0}^m \varepsilon_{j_1}\varepsilon_{j_2} a_{ik}^{j_1} a_{kj}^{j_2})_{n\times n}.
\end{eqnarray*}
Let $\hat{A} = \big(\max\limits_{1\leq j\leq m} \{\varepsilon_j\}\big)^2 I_n$ and $P_{\varepsilon} = AA^*- \hat{A}.$ Denote the eigenvalue of $AA^*$ by $\breve{\lambda}_1\leq \cdots\leq \breve{\lambda}_n.$ Obviously, $\hat{A}^* \hat{A} = \hat{A} \hat{A}^*$, i.e., $\hat{A}$ is normal.  According to Lemma \ref{LMU}, we have
\begin{eqnarray}\label{EQM1}
|\breve{\lambda}_1 -  \max\limits_{1\leq j\leq m} \{\varepsilon_j^2\}| \leq |||P_{\varepsilon}|||_2,
\end{eqnarray}
where $|||P_{\varepsilon}|||_2 = (tr P_{\varepsilon}^* P_{\varepsilon} )^{\frac{1}{2}}$. Denote $b_{ij}=\sum\limits_{k=1}^n \sum\limits_{j_1=0}^m \sum\limits_{j_2=0}^m \varepsilon_{j_1}\varepsilon_{j_2} a_{ik}^{j_1} a_{kj}^{j_2}$. Then $AA^* = (b_{ij})_{n\times n}.$ Let $P_{\varepsilon} = (c_{ij})_{n\times n}$, where $c_{ij}= b_{ij}$, $i\neq j$, $c_{ij} = b_{ij} - \big(\max\limits_{1\leq j\leq m} \{\varepsilon_j\}\big)^2$. Then
\begin{eqnarray*}
|||P_{\varepsilon}|||_2 &=& (tr P_{\varepsilon}^* P_{\varepsilon} )^{\frac{1}{2}}\\
&=& ( \sum\limits_{j=1}^n\sum\limits_{i=1}^n \bar{c}_{ij} c_{ij} )^{\frac{1}{2}},
\end{eqnarray*}
which implies $\min\limits_{j}\{\varepsilon_j^2\}\leq |||P_{\varepsilon}|||_2 \leq \max\limits_{j}\{\varepsilon_j^2\}.$ Hence,
\begin{eqnarray*}
\breve{\lambda}_1 &\geq& \max\limits_{j}\{\varepsilon_j^2\} -|||P_\varepsilon|||_2\\
&\geq& \max\limits_{j}\{\varepsilon_j^2\} - ( \sum\limits_{j=1}^n\sum\limits_{i=1}^n \bar{c}_{ij} c_{ij} )^{\frac{1}{2}}\\
&\geq& c \min\limits_{j}\{\varepsilon_j^2\},
\end{eqnarray*}
where $c$ is positive and depends on $a_{ij}^k$, $1\leq i,j \leq n$, $0\leq k\leq m$.
\end{proof}

\section{The nonsingular of $\mathcal{A}^*\mathcal{A}$}\label{Nonsin}
Directly, $\mathcal{A}^*\mathcal{A}
 = \left(
                               \begin{array}{ccc}
                                  a_{11}& a_{12} & 0 \\
                                 a_{21} & a_{22} & a_{23} \\
                                 0 & a_{32} & a_{33} \\
                               \end{array}
                             \right)= \mathcal{A}_2^* \mathcal{A}_2 + A_1,$
where
\begin{eqnarray*}
A_1 &=& \left(
                               \begin{array}{ccc}
                                  \hat{a}_{11}& \hat{a}_{12} & 0 \\
                                 \hat{a}_{21} & \hat{a}_{22} & \hat{a}_{23} \\
                                 0 & \hat{a}_{32} & \hat{a}_{33} \\
                               \end{array}
                             \right),\\
a_{11} &=& I_n \otimes (L_{k0}I_n)^* \cdot I_n \otimes (L_{k0}I_n) + I_n \otimes (L_{k0} I_n)^* \cdot I_n \otimes (\varpi I_n)\\
 &~&+ I_n \otimes (\varpi I_n)^*\cdot I_n \otimes (L_{k0} I_n) + I_n \otimes (\varpi I_n)^*\cdot  I_n \otimes \varpi I_n,\\
a_{12} &=& I_n\otimes (L_{k0} I_n)^* \cdot I_n \otimes (M_{21}^T J) +  I_n \otimes (L_{k0}I_n)^* \cdot I_n \otimes (\hat{h}_1 J)\\
&~&+I_n \otimes (\varpi I_n)^* \cdot I_n \otimes (M_{21}^T J) +  I_n \otimes (\varpi I_n)^* \cdot I_n\otimes (\hat{h}_1 J),\\
a_{21} &=& I_n\otimes (M_{21}^T J )^* \cdot I_n \otimes L_{k0}I_n + I_n\otimes (M_{21}^T J )^* \cdot I_n\otimes \varpi I_n\\
&~&I_n\otimes (\hat{h}_1 J )^* \cdot  I_n \otimes L_{k0}I_n +I_n\otimes (\hat{h}_1 J )^* \cdot I_n \otimes \varpi I_n,\\
a_{22} &=& I_n\otimes (M_{21}^T J)^* \cdot I_n \otimes (M_{21}^T J) + I_n \otimes (M_{21}^T J )^* \cdot I_n \otimes (\hat{h}_1 J)\\
 &~& +I_n \otimes (\hat{h}_1 J)^* \cdot I_n\otimes(M_{21}^T J) + I_n \otimes(\hat{h}_1 J)^*\cdot I_n \otimes (\hat{h}_1 J) \\
 &~&+I_n \otimes L_{k1}^*\cdot I_n\otimes L_{k1} - I_n \otimes L_{k1}^* \cdot I_n \otimes \varpi I_{2m}\\
 &~&- I_n \otimes L_{k1}^* \cdot I_n \otimes (\hat{h}_2 J) - I_n \otimes (\varpi I_{2m})^* \cdot I_n \otimes L_{k1}\\
 &~&+I_n \otimes (\varpi_{2m})^* \cdot I_n \otimes \varpi I_{2m} + I_n \otimes (\varpi_{2m})^* \cdot I_n \otimes (\hat{h}_2 J)\\
 &~&- I_n\otimes (\hat{h}_2 J)^* \cdot I_n \otimes L_{k1} + I_n\otimes (\hat{h}_2 J)^* \cdot I_n \otimes \varpi I_{2m} \\
 &~&+I_n\otimes (\hat{h}_2 J )^* \cdot I_n\otimes (\hat{h}_2 J ),\\
a_{23} &=& -I_n \otimes L_{k1}^* \cdot (M_{21}^T J) \otimes I_{2m} - I_n \otimes L_{k1}^* \cdot (\hat{h}_1 J) \otimes I_{2m}\\
 &~&+ I_n \otimes (\varpi I_{2m}) \cdot (M_{21}^T J)\otimes I_{2m}+I_n \otimes (\varpi I_{2m}) \cdot (\hat{h}_1 J)\otimes I_{2m}\\
 &~&+I_n \otimes (\hat{h}_2 J)^* \cdot (M_{21}^T J)\otimes I_{2m} + I_n \otimes (\hat{h}_2 J)^* \cdot (\hat{h}_1J) \otimes I_{2m},\\
a_{32} &=& -(M_{21}^T J)^* \otimes I_{2m}\cdot I_n\otimes L_{k1} +  (M_{21}^T J)^* \otimes I_{2m}\cdot I_n \otimes \varpi I_{2m}\\
 &~&+(M_{21}^T J)^*\otimes I_{2m}\cdot I_n \otimes (\hat{h}_2 J) - (\hat{h}_1 J )^*\otimes I_{2m} \cdot I_n \otimes L_{k1} \\
 &~& - (\hat{h}_1 J )^*\otimes I_{2m} \cdot I_n \otimes L_{k1} + (\hat{h}_1 J )^*\otimes I_{2m} \cdot I_n \otimes \varpi I_{2m}\\
 &~&+ (\hat{h}_1 J )^*\otimes I_{2m} \cdot I_n \otimes (\hat{h}_2J),\\
a_{33} &=&  (M_{21}^T J)^* \otimes I_{2m} \cdot M_{21}^T J \otimes I_{2m}+ (M_{21}^T J)^* \otimes I_{2m} \cdot (\hat{h}_1J) \otimes I_{2m}\\
 &~& (\hat{h}_1J)^* \otimes I_{2m} \cdot (M_{21}^T J) \otimes I_{2m} +(\hat{h}_1J)^* \otimes I_{2m} \cdot (\hat{h}_1J) \otimes I_{2m}\\
 &~&+L_{k2}^* L_{k2} + L_{k2}^*\cdot \varpi I_{4m^2} - L_{k2}^* \cdot(\hat{h}_2 J) \otimes I_{2m}- L_{k2}^* \cdot I_{2m} \otimes (\hat{h}_2 J) \\
 &~&  + (\varpi I_{4m^2})^* \cdot L_{k2} +(\varpi I_{4m^2})^* \cdot(\varpi I_{4m^2}) \\
 &~& - (\varpi I_{4m^2})^*\cdot(\hat{h}_2 J)\otimes I_{2m} - (\varpi I_{4m^2})^* \cdot I_{2m}\otimes (\hat{h}_2 J)\\
 &~& - (\hat{h}_2 J) \otimes I_{2m}\cdot L_{k2} -  (\hat{h}_2 J )^* \otimes I_{2m}\cdot \varpi I_{4m^2}\\
 &~& +(\hat{h}_2 J)^* \otimes I_{2m}\cdot (\hat{h}_2J)\otimes I_{2m} + (\hat{h}_2 J)^* \otimes I_{2m}\cdot I_{2m} \otimes(\hat{h}_2 J)\\
 &~& -I_{2m} \otimes (\hat{h}_2 J)^* \cdot L_{k2} -  I_{2m} \otimes (\hat{h}_2 J)^* \cdot(\varpi I_{4m^2})\\
 &~&  + I_{2m} \otimes (\hat{h}_2 J)^* \cdot (\hat{h}_2 J) \otimes I_{2m} + I_{2m} \otimes (\hat{h}_2 J)^* \cdot I_{2m}\otimes (\hat{h}_2 J),\\
\hat{a}_{11} &=&  I_n \otimes (L_{k0} I_n)^* \cdot I_n \otimes (\varpi I_n)+ I_n \otimes (\varpi I_n)^*\cdot I_n \otimes (L_{k0} I_n)  \\
 &~& + I_n \otimes (\varpi I_n)^*\cdot  I_n \otimes \varpi I_n,\\
\hat{a}_{12} &=&  I_n \otimes (L_{k0}I_n)^* \cdot I_n \otimes (\hat{h}_1 J)+I_n \otimes (\varpi I_n)^* \cdot I_n \otimes (M_{21}^T J)  \\
&~& +  I_n \otimes (\varpi I_n)^* \cdot I_n\otimes (\hat{h}_1 J),\\
\hat{a}_{21} &=&  I_n\otimes (M_{21}^T J )^* \cdot I_n\otimes \varpi I_n + I_n\otimes (\hat{h}_1 J )^* \cdot  I_n \otimes L_{k0}I_n\\
&~& +I_n\otimes (\hat{h}_1 J )^* \cdot I_n \otimes \varpi I_n,\\
\hat{a}_{22} &=&  I_n \otimes (M_{21}^T J )^* \cdot I_n \otimes (\hat{h}_1 J)+I_n \otimes (\hat{h}_1 J)^* \cdot I_n\otimes(M_{21}^T J)\\
 &~&  + I_n \otimes(\hat{h}_1 J)^*\cdot I_n \otimes (\hat{h}_1 J) - I_n \otimes L_{k1}^* \cdot I_n \otimes \varpi I_{2m} \\
 &~& - I_n \otimes L_{k1}^* \cdot I_n \otimes (\hat{h}_2 J) - I_n \otimes (\varpi I_{2m})^* \cdot I_n \otimes L_{k1}\\
 &~& +I_n \otimes (\varpi_{2m})^* \cdot I_n \otimes \varpi I_{2m} + I_n \otimes (\varpi_{2m})^* \cdot I_n \otimes (\hat{h}_2 J)\\
 &~& - I_n\otimes (\hat{h}_2 J)^* \cdot I_n \otimes L_{k1} + I_n\otimes (\hat{h}_2 J)^* \cdot I_n \otimes \varpi I_{2m} \\
 &~& +I_n\otimes (\hat{h}_2 J )^* \cdot I_n\otimes (\hat{h}_2 J ),\\
\hat{a}_{23} &=& - I_n \otimes L_{k1}^* \cdot (\hat{h}_1 J) \otimes I_{2m}~+ I_n \otimes (\varpi I_{2m}) \cdot (M_{21}^T J)\otimes I_{2m} \\
 &~& +I_n \otimes (\varpi I_{2m}) \cdot (\hat{h}_1 J)\otimes I_{2m}+I_n \otimes (\hat{h}_2 J)^* \cdot (M_{21}^T J)\otimes I_{2m}\\
 &~& + I_n \otimes (\hat{h}_2 J)^* \cdot (\hat{h}_1J) \otimes I_{2m},\\
\hat{a}_{32} &=&  (M_{21}^T J)^* \otimes I_{2m}\cdot I_n \otimes \varpi I_{2m}+(M_{21}^T J)^*\otimes I_{2m}\cdot I_n \otimes (\hat{h}_2 J) \\
 &~& - (\hat{h}_1 J )^*\otimes I_{2m} \cdot I_n \otimes L_{k1} - (\hat{h}_1 J )^*\otimes I_{2m} \cdot I_n \otimes L_{k1}\\
 &~&  + (\hat{h}_1 J )^*\otimes I_{2m} \cdot I_n \otimes \varpi I_{2m}+ (\hat{h}_1 J )^*\otimes I_{2m} \cdot I_n \otimes (\hat{h}_2J),\\
\hat{a}_{33} &=&   (M_{21}^T J)^* \otimes I_{2m} \cdot (\hat{h}_1J) \otimes I_{2m} +(\hat{h}_1J)^* \otimes I_{2m} \cdot (M_{21}^T J) \otimes I_{2m} \\
 &~&  +(\hat{h}_1J)^* \otimes I_{2m} \cdot (\hat{h}_1J) \otimes I_{2m} + L_{k2}^*\cdot \varpi I_{4m^2} - L_{k2}^* \cdot(\hat{h}_2 J) \otimes I_{2m} \\
 &~&  - L_{k2}^* \cdot I_{2m} \otimes (\hat{h}_2 J) + (\varpi I_{4m^2})^* \cdot L_{k2} +(\varpi I_{4m^2})^* \cdot(\varpi I_{4m^2}) \\
 &~&  - (\varpi I_{4m^2})^*\cdot(\hat{h}_2 J)\otimes I_{2m} - (\varpi I_{4m^2})^* \cdot I_{2m}\otimes (\hat{h}_2 J)\\
 &~&  - (\hat{h}_2 J) \otimes I_{2m}\cdot L_{k2} -  (\hat{h}_2 J )^* \otimes I_{2m}\cdot \varpi I_{4m^2}\\
 &~&  +(\hat{h}_2 J)^* \otimes I_{2m}\cdot (\hat{h}_2J)\otimes I_{2m} + (\hat{h}_2 J)^* \otimes I_{2m}\cdot I_{2m} \otimes(\hat{h}_2 J)\\
 &~&  -I_{2m} \otimes (\hat{h}_2 J)^* \cdot L_{k2} -  I_{2m} \otimes (\hat{h}_2 J)^* \cdot(\varpi I_{4m^2})\\
 &~&  + I_{2m} \otimes (\hat{h}_2 J)^* \cdot (\hat{h}_2 J) \otimes I_{2m} + I_{2m} \otimes (\hat{h}_2 J)^* \cdot I_{2m}\otimes (\hat{h}_2 J).
\end{eqnarray*}
For Hermitian matrix $A_1$, there is a unitary matrix $Q_2$ such that
$Q_2^* A_1 Q_2  = diag(\lambda_1, \cdots, \lambda_{(n+ 2m)^2 }),$
where $\lambda_{\min} = \lambda_1 \leq \cdots \leq \lambda_{(n+ 2m)^2} = \lambda_{\max}.$ Moreover,
\begin{eqnarray*}
|\lambda_{\min} |=| \min\limits_{\{x, 0\neq x \in S\}} \frac{x^* A_1 x}{x^* x}| &\leq& ||{A}_1||_\infty \leq |k| s\leq |k| s^{\frac{1}{2}} \varepsilon^2 \\
&\leq& \frac{1}{2} (\frac{\min\{\varepsilon_1, \cdots, \varepsilon_m, \mu_1, \cdots, \mu_l\}\gamma}{|k|^\tau})^2,
\end{eqnarray*}
where $S = span\{x_1, \cdots, x_{(n+ 2m)^2}\}$, ${A}_1 x_i = \lambda_i x_i$, $1\leq i\leq (n+ 2m)^2.$
According to Weyl Theorem, we have
\begin{eqnarray*}
\lambda_{\min} (\mathcal{A}^*\mathcal{A}) &=& \lambda_{\min} (\mathcal{A}_2^* \mathcal{A}_2 + {A}_1) \geq \lambda_{\min} (\mathcal{A}_2^* \mathcal{A}_2) +  \lambda_{\min} ({A}_1)\\
&\geq& \frac{1}{2} (\frac{\min\{\varepsilon_1, \cdots, \varepsilon_m, \mu_1, \cdots, \mu_l\}\gamma}{|k|^\tau})^2,
\end{eqnarray*}
which means that $\mathcal{A}^*\mathcal{A} \geq \frac{1}{2}(\frac{\min\{\varepsilon_1, \cdots, \varepsilon_m, \mu_1, \cdots, \mu_l\} \gamma}{|k|^\tau})^2 I_{(n+ 2m)^2}.$
\end{appendices}

\section*{Acknowledgements} The first author was supported by China Postdoctoral Science Foundation (2021M701396, 2022T150262), NSFC(12201243), ERC (PR1062ERC01). He sincerely thanks Professor Vadim Kaloshin for his support during the stay in ISTA, where part of the work was done. The second author was supported by NSFC (12225103, 12071065 and 11871140) and the National Key Research and Development Program of China (2020YFA0713602). The third author was supported by National Basic Research Program of China (2013CB834100), NSFC (12071175 and 11571065), JilinDRC (2017c028-1), JilinSTD (20190201302JC).
\section*{Data Availability Statements} Data sharing not applicable to this article as no datasets were generated or analysed during the current study.

\baselineskip 9pt \renewcommand{\baselinestretch}{1.08}

\end{document}